\documentclass[11pt,a4paper]{amsart}
\title{Algebraic models of change of groups functors in (co)free rational equivariant spectra}
\date{}
\usepackage{amsmath,amssymb,amsthm,tikz, dsfont, stmaryrd, float,tikz-cd,mathrsfs, hyperref, mathtools, todonotes, enumerate, tipa, fullpage, microtype}
\hypersetup{
	colorlinks=true,
	citecolor=blue,
	linkcolor=blue,
	filecolor=magenta,      
	urlcolor=cyan,
}
\usepackage[T1]{fontenc}
\usepackage{textcomp}
\usepackage{lmodern}
\usepackage[mathscr]{euscript}
\tikzset{
	labl/.style={anchor=south, rotate=270, inner sep=.5mm}
}
\numberwithin{equation}{section}
\theoremstyle{plain}
\newtheorem{thm}[equation]{Theorem}

\newtheorem{prop}[equation]{Proposition}
\newtheorem{lem}[equation]{Lemma}
\newtheorem{cor}[equation]{Corollary}

\theoremstyle{definition}
\newtheorem{defn}[equation]{Definition}

\newtheorem{rem}[equation]{Remark}

\newtheorem{nota}[equation]{Notation}
\newtheorem{eg}[equation]{Example}
\newtheorem{egs}[equation]{Examples}

\newtheorem{hyp}[equation]{Hypothesis}

\newcommand{\mc}{\mathcal}

\renewcommand{\lim}[1]{\mathop{\text{lim}}_{#1}}

\renewcommand{\phi}{\varphi}
\renewcommand{\epsilon}{\varepsilon}
\newcommand{\dual}{\mathbb{D}}
\renewcommand{\L}{\mathbb{L}}
\newcommand{\R}{\mathbb{R}}

\newcommand{\res}[1][\theta]{{#1}^*}
\newcommand{\ext}[1][\theta]{{#1}_*}
\newcommand{\psext}[1][\theta]{{#1}_{(*)}}
\newcommand{\coext}[1][\theta]{{#1}_!}
\newcommand{\twext}[1][\theta]{{#1}^\dagger}

\newcommand{\ind}{i_*}
\newcommand{\coind}{i_!}
\newcommand{\forget}{i^*}

\newcommand{\reskg}{\kappa_G^*}
\newcommand{\extkg}{\kappa^G_*}
\newcommand{\reskh}{\kappa_H^*}
\newcommand{\extkh}{\kappa^H_*}
\newcommand{\exta}{\phi_*}
\newcommand{\twexta}{\phi^\dagger}
\newcommand{\resa}{\phi^*}
\newcommand{\psexta}{\phi_{(*)}}

\renewcommand{\mod}[1]{\mathrm{Mod}_{#1}}
\newcommand{\calg}[1]{\mathrm{CAlg}_{#1}}

\newcommand{\symsp}{\mathrm{Sp}^\Sigma}

\newcommand{\sset}{\mathrm{sSet}}

\newcommand{\sq}{\mathrm{s}\mathbb{Q}\text{-}\mathrm{mod}}

\newcommand{\chq}{\mathrm{Ch}_\mathbb{Q}^+}

\newcommand{\Sp}{\mathrm{Sp}}
\newcommand{\D}{\mathscr{D}}
\newcommand{\cell}{\mathrm{Cell}}

\newcommand{\N}{\mathscr{N}}

\newcommand{\C}{\mathscr{C}}
\newcommand{\E}{\mathscr{E}}
\newcommand{\Dsf}{\mathsf{D}}

\makeatletter
\@namedef{subjclassname@2020}{\textup{2020} Mathematics Subject Classification}
\makeatother

\begin{document}
\author{Jordan Williamson}
\address[Williamson]{Department of Algebra, Faculty of Mathematics and Physics, Charles University in Prague, Sokolovsk\'{a} 83, 186 75 Praha, Czech Republic}
\email{williamson@karlin.mff.cuni.cz}
\subjclass[2020]{55P91, 55U35, 55P92}
\begin{abstract}
	Free and cofree equivariant spectra are important classes of equivariant spectra which represent equivariant cohomology theories on free equivariant spaces. Greenlees-Shipley~\cite{GreenleesShipley11, GreenleesShipley14} and Pol and the author~\cite{PolWilliamson} have given an algebraic model for rational (co)free equivariant spectra. In this paper, we extend this framework by proving that the Quillen functors of induction-restriction-coinduction between categories of (co)free rational equivariant spectra correspond to Quillen functors between the algebraic models in the case of connected compact Lie groups. This is achieved using new abstract techniques regarding correspondences of Quillen functors along Quillen equivalences, which we expect to be of use in other applications.
\end{abstract}
	\maketitle
	\setcounter{tocdepth}{1}
	\tableofcontents 

\section{Introduction}
An equivalence between homotopy theories (for example, a Quillen equivalence) gives a powerful method of translating homotopical information between different categories. In many cases, one is interested not just in a single Quillen equivalence, but rather a family of Quillen equivalences $\C(i) \simeq_Q \D(i)$. In this setting, one would then like to understand how relations between the $\C(i)$'s correspond to relations between the $\D(i)$'s under the Quillen equivalences. More precisely, given a Quillen adjunction $\C(i) \rightleftarrows \C(j)$, one would like to find another Quillen adjunction $\D(i) \rightleftarrows \D(j)$ which corresponds to the original adjunction under the Quillen equivalence; diagrammatically, we want to find a diagram
\begin{center}
	\begin{tikzcd}
	\C(i) \arrow[rr, "\simeq_Q"] \arrow[dd, xshift=1mm] & & \D(i) \arrow[ll] \arrow[dd, xshift=1mm] \\
	& & \\
	\C(j) \arrow[rr, "\simeq_Q"']  \arrow[uu, xshift=-1mm] & & \D(j) \arrow[ll] \arrow[uu, xshift=-1mm] 
	\end{tikzcd}
\end{center}
which commutes homotopically. We refer the reader to later on in this introduction for a more precise definition, but remark now that such a result requires care since Quillen equivalences often take the form of zig-zags, leading to compositions of left and right Quillen functors which in general are poorly behaved. 

In the first part of this paper, we develop some abstract machinery which gives techniques to prove such correspondences, with a particular focus on the case of correspondences of adjoint triples. We use the theory of mates to give sufficient conditions to prove a correspondence of Quillen functors, and then verify these conditions in several general cases of interest; for example, for correspondences of change of rings adjunctions along strong or weak monoidal Quillen equivalences. In the second part of the paper, we then apply these results to a case of interest coming from rational equivariant stable homotopy theory. We shall now turn to giving a more detailed overview of the motivation and setup, focusing on the example from equivariant homotopy theory to give a concrete illustration of the abstract formalism.

\subsection{Rational equivariant cohomology theories}
Rational equivariant cohomology theories are represented by objects called rational $G$-spectra. Greenlees~\cite{Greenlees06} conjectured that for each compact Lie group $G$, there is an abelian category $\mc{A}(G)$ and a zig-zag of Quillen equivalences $$\Sp_G \simeq_Q d\mc{A}(G)$$ between rational $G$-spectra and differential objects in $\mc{A}(G)$. The conjecture has been proved in many cases: $G$ finite~\cite{Barnes09}, $G=SO(2)$~\cite{Shipley02}, $G=O(2)$~\cite{Barnes17}, $G=SO(3)$~\cite{Kedziorek17} and $G$ a torus of any rank~\cite{GreenleesShipley18}.

Rather than focusing on a particular group $G$, we are interested in certain classes of rational $G$-spectra: those of free and cofree $G$-spectra. These are the $G$-spectra for which the natural map $EG_+ \wedge X \to X$ or the natural map $X \to F(EG_+,X)$ are equivalences respectively. There are several reasons why these objects are of particular interest. Firstly, they represent equivariant cohomology theories on free $G$-spaces. In addition, these cases provide insight into the general case, where the algebraic models are built from contributions at each closed subgroup, where the model resembles that of the free case. Greenlees-Shipley~\cite{GreenleesShipley11, GreenleesShipley14} constructed an algebraic model for free $G$-spectra and Pol and the author have given an algebraic model for cofree $G$-spectra~\cite{PolWilliamson}, where $G$ is any compact Lie group.

\subsection{Change of groups}
The inclusion of a subgroup $i\colon H \to G$ in a compact Lie group gives rise to an adjoint triple. The restriction functor $\forget\colon \Sp_G \to \Sp_H$ has both a left adjoint $\ind = G_+ \wedge_H -$ called induction and a right adjoint $\coind = F_H(G_+,-)$ called coinduction. Moreover the adjoint triple
\begin{center}
	\begin{tikzcd}
	\Sp_G \arrow[rr, "\forget" description] & & \Sp_H \arrow[ll, yshift=3mm, "\ind" description] \arrow[ll, yshift=-3mm, "\coind" description] 
	\end{tikzcd}
\end{center}
is a Quillen adjoint triple. In other words, both adjunctions are Quillen with respect to the \emph{same} model structures. Such a situation is a rare occurrence since it forces the middle functor to be both left and right Quillen. In particular, such a functor must preserve all of the classes of maps in the model structure, including the weak equivalences.

In a setting where we have algebraic models for $G$-spectra and $H$-spectra, it is a natural question to ask what functors between the algebraic models correspond to this adjoint triple. Diagrammatically, we want to find functors
\begin{center}
	\begin{tikzcd}
	\Sp_G \arrow[rr, "\simeq_Q"] \arrow[dd, "\forget" description] & & d\mc{A}(G) \arrow[ll] \arrow[dd, "?" description] \\
	& & \\
	\Sp_H \arrow[rr, "\simeq_Q"'] \arrow[uu, xshift=-4mm, "\ind" description] \arrow[uu, xshift=4mm, "\coind" description]& & d\mc{A}(H) \arrow[ll] \arrow[uu, xshift=-4mm, "?" description] \arrow[uu, xshift=4mm, "?" description] 
	\end{tikzcd}
\end{center}
which are algebraic counterparts of the functors in topology. We emphasise that our approach to this is model categorical; we view the functors as Quillen functors and we want to show that there are natural maps at the point-set level which realise the correspondence of functors. At the homotopy category level, the correspondence of functors has been studied by Greenlees-Shipley~\cite{GreenleesShipley11} and Greenlees~\cite{Greenleeschangeofgroups} in the free case. We also note that in the non-free case, Greenlees~\cite{Greenlees99} has given an account of the correspondence of change of groups functors between $SO(2)$-spectra and $H$-spectra for $H$ a subgroup of $SO(2)$, again at the derived level.

If $G$ is connected, the algebraic model for free $G$-spectra is $I$-power torsion modules over the polynomial ring $H^*BG$~\cite[1.1]{GreenleesShipley11} where $I$ is the augmentation ideal of $H^*BG$. The algebraic model for cofree $G$-spectra is $L$-complete modules over $H^*BG$~\cite[8.4]{PolWilliamson}, where a $H^*BG$-module $M$ is said to be $L$-complete if the natural map $M \to L_0^IM$ is an isomorphism, where $L_0^I$ is the zeroth left derived functor of the $I$-adic completion. The inclusion $i\colon H \to G$ gives rise to a ring map $\theta\colon H^*BG \to H^*BH$, and therefore an adjoint triple between the categories of modules. The adjoint triple
\begin{center}
	\begin{tikzcd}
	\mod{H^*BG} \arrow[rr, yshift=3.2mm, "\ext" description] \arrow[rr, yshift=-3.2mm, "\coext" description]  & & \mod{H^*BH} \arrow[ll, "\res" description] 
	\end{tikzcd}
\end{center}
is given by the restriction of scalars $\res$, extension of scalars $\ext = H^*BH \otimes_{H^*BG} -$ and coextension of scalars $\coext = \mathrm{Hom}_{H^*BG}(H^*BH,-)$. 

Despite this, one notices that it is not routine to write down the algebraic models as there is a mismatch in directions. In topology, two of the functors go from $H$-spectra to $G$-spectra, but in algebra only one functor goes in this direction. Therefore, one must construct extra functors in algebra to model the adjoint triple in topology. 

Before we can state the main theorem of this paper, we must explain what we mean by a correspondence of Quillen functors. We do this by considering the following example. The notation we have chosen is suggestive of the special case we have in mind. In particular, we are not assuming the existence of any group actions on the model categories in this general framework. Suppose that we have a diagram 
\begin{center}
\begin{tikzcd}
\C_G \arrow[r, "F_G", yshift=1mm] \arrow[d, "L"', xshift=-1mm] & \D_G \arrow[l, "U_G", yshift=-1mm] \arrow[r, "U_G'"', yshift=-1mm] & \E_G \arrow[l, "F_G'"', yshift=1mm] \arrow[d, "L'"', xshift=-1mm] \\
\C_H \arrow[r, "F_H", yshift=1mm] \arrow[u,"R"', xshift=1mm] & \D_H \arrow[l, "U_H", yshift=-1mm] \arrow[r, "U_H'"', yshift=-1mm] & \E_H \arrow[l, "F_H'"', yshift=1mm] \arrow[u,"R'"', xshift=1mm]
\end{tikzcd}
\end{center}
of model categories where each of the horizontal adjunctions is a Quillen equivalence, and $(L,R)$ and $(L',R')$ are Quillen adjunctions. We say that $(L',R')$ corresponds to $(L,R)$ if there exists a Quillen adjunction $L'': \D_G \rightleftarrows \D_H:R''$ together with natural weak equivalences $F_HL \simeq L''F_G$ and $L''F_G' \simeq F_H'L'$ on cofibrant objects. Such a correspondence of Quillen adjunctions gives natural isomorphisms of derived functors $$\R U_H' \circ \L F_H \circ \L L \cong \L L' \circ \R U_G'\circ \L F_G$$ and $$\R U_G' \circ \L F_G \circ \R R \cong \R R' \circ \R U_H' \circ \L F_H$$ using the theory of mates. For more details, see Section~\ref{sec:compare}. In a similar way one can define correspondences of Quillen adjunctions along zig-zags of Quillen equivalences of any length. We note that this is a particularly structured form of correspondence since the intermediate steps are required to be Quillen adjunctions too.

We can now state the main theorem of this paper. Recall that the restriction functor $\forget$ is both left and right Quillen. Therefore, there are two functors which correspond to it in algebra; one as a left Quillen functor and one as a right Quillen functor. This can be seen in the diagram below where there are four functors in algebra rather than the three in topology.
\begin{thm}\label{thm:main}
	Let $i\colon H \to G$ be the inclusion of a connected subgroup into a connected compact Lie group. We have the following correspondence of Quillen functors
	\begin{center}
		\begin{tikzcd}
		\text{free $G$-spectra} \arrow[rr, "\simeq_Q"] \arrow[dd, "\forget" description] & & \text{derived torsion $H^*BG$-modules} \arrow[ll] \arrow[dd, "\Sigma^{-a}\coext" description, xshift=-5mm] \arrow[dd, "\ext" description, xshift=5mm] \\
		& & \\
		\text{free $H$-spectra} \arrow[rr, "\simeq_Q"'] \arrow[uu, xshift=-6mm, "\ind" description] \arrow[uu, xshift=6mm, "\coind" description]& & \text{derived torsion $QH^*BH$-modules} \arrow[ll] \arrow[uu, xshift=-15mm, "\Sigma^a\res" description] \arrow[uu, xshift=15mm, "\res" description] 
		\end{tikzcd}
	\end{center}
	where $a=\mathrm{dim}(G/H)$ and $QH^*BH$ is a cofibrant replacement of $H^*BH$ as a commutative $H^*BG$-algebra. In other words, $(\ind, \forget)$ corresponds to $(\Sigma^a\res, \Sigma^{-a}\coext)$ and $(\forget, \coind)$ corresponds to $(\ext, \res).$ Similarly, when the induction, forgetful functor and coinduction functors are viewed as functors between the categories of cofree spectra, they correspond to the same functors as in the free case, now viewed as functors between the categories of derived complete modules.
\end{thm}
\begin{rem}\label{rem:notleftadjoint}
	The functors $\ext$ and $\Sigma^{-a}\coext$ are not isomorphic in general at the model categorical level. However, at the derived level this is true. For more details, see Proposition~\ref{prop:functorsetup}.
\end{rem}

\begin{rem}
It would be interesting to investigate whether the natural isomorphism $\L\theta_* \cong \R\Sigma^{-a}\coext$ in algebra corresponds to the natural isomorphism $\L\forget \cong \R\forget$ in topology. However, since the natural isomorphism $\L\theta_* \cong \R\Sigma^{-a}\coext$ is a zig-zag of equivalences due to the relative Gorenstein condition, it is not clear how to approach this at a model categorical level. For more details on these natural isomorphisms see Section~\ref{sec:functors}. We also note that the existing applications do not require this additional correspondence.
\end{rem}

If the ranks of $G$ and $H$ are equal there is a stronger statement, see Remark~\ref{rem:equalranks} and Proposition~\ref{prop:cofibrantrank}. In this case Remark~\ref{rem:notleftadjoint} does not apply, and the restriction of scalars is both left and right adjoint (without a shift) to the extension of scalars along the ring map $\theta\colon H^*BG \to H^*BH$.
\begin{thm}\label{thm:main2}
	Let $i\colon H \to G$ be the inclusion of a connected subgroup into a connected compact Lie group and assume that $\mathrm{rk}G=\mathrm{rk}H$. Then we have a correspondence of Quillen functors 
	\begin{center}
	\begin{tikzcd}
	\text{free $G$-spectra} \arrow[rr, "\simeq_Q"] \arrow[dd, "\forget" description] & & \text{$I$-power torsion $H^*BG$-modules} \arrow[ll] \arrow[dd, "\ext" description] \\
	& & \\
	\text{free $H$-spectra} \arrow[rr, "\simeq_Q"'] \arrow[uu, xshift=-6mm, "\ind" description] \arrow[uu, xshift=6mm, "\coind" description]& & \text{$J$-power torsion $H^*BH$-modules} \arrow[ll] \arrow[uu, xshift=-6mm, "\res" description] \arrow[uu, xshift=6mm, "\res" description] 
	\end{tikzcd}
	\end{center}
	where $I$ and $J$ are the augmentation ideals of $H^*BG$ and $H^*BH$ respectively. Similarly, when the induction, forgetful functor and coinduction are viewed as functors between the categories of cofree spectra, they correspond to $\res$, $\ext$ and $\res$ respectively, between the categories of $L$-complete modules.	
\end{thm}

\subsection{Summary of the method}\label{sec:summary of method}
The first part of the paper is dedicated to producing the required setup to prove the main theorem. In particular, the results are not restricted to this example, and we expect that they can be applied in other cases, such as the disconnected case or the non-free case, and moreover in examples beyond equivariant stable homotopy theory. We consider diagrams of model categories and Quillen functors of the form
\begin{center}
	\begin{tikzcd}
	\quad\C_G\quad \arrow[dd, xshift=2.5mm, "\ext" description] \arrow[dd, xshift=-2.5mm, "\psext" description] \arrow[rr, yshift=1.7mm, "F_G"' description] & & \quad\D_G\quad \arrow[dd, xshift=2.75mm, "\exta" description] \arrow[dd, xshift=-2.75mm, "\psexta" description] \arrow[ll, yshift=-1.7mm, "U_G"' description] \\
	& & \\
	\quad\C_H\quad \arrow[uu, xshift=7.5mm, "\res"' description] \arrow[uu, xshift=-7.5mm, "\twext" description] \arrow[rr, yshift=1.7mm, "F_H"' description] & & \quad\D_H\quad \arrow[uu, xshift=8mm, "\resa"' description] \arrow[uu, xshift=-8mm, "\twexta" description] \arrow[ll, yshift=-1.7mm, "U_H"' description]
	\end{tikzcd}
\end{center}
where $\R\psext \cong \L\ext$ and $\R\psext[\phi] \cong \L\ext[\phi]$, and $F_G \dashv U_G$ and $F_H \dashv U_H$ are Quillen equivalences. As described above, in this general framework the notation is purely suggestive and does not indicate the existence of any group actions.

This encompasses all of the cases that arise in our proof of the correspondence for (co)free spectra and connected groups. There are eight squares of derived functors which one would like to show commute:
\begin{enumerate}
\item $\L\exta \circ\L F_G \cong \L F_H \circ \L\ext$,
\item $\R \res \circ\R U_H \cong \R U_G \circ \R\resa$,
\item $\L F_G \circ\L \twext \cong \L\twexta \circ \L F_H$,
\item $\R\psext \circ \R U_G \cong \R U_H \circ \R \psexta$,
\item $\L\twext \circ \R U_H \cong\R U_G \circ \L\twexta$, 
\item $\L F_G \circ \R \res \cong \R\resa \circ \L F_H$,
\item $\R \psexta \circ \L F_G \cong \L F_H \circ \R \psext$,
\item $\L\ext \circ \R U_G \cong \R U_H \circ \L \exta$.
\end{enumerate} 
By virtue of the natural isomorphisms $\R\psext \cong \L\ext$ and $\R\psext[\phi] \cong \L\ext[\phi]$, we note that $(7)$ is equivalent to $(1)$ and $(8)$ is equivalent to $(4)$. Since the functors in $(1)$ and $(4)$ have the same handedness, these conditions are the ones which are more natural to check. Therefore, from now on we refer to \emph{six} derived squares, rather than eight. 

We show that instead of checking that all six squares commute, it is sufficient to prove that only two do, see Theorem~\ref{thm:compare}. Moreover, we check these conditions at a model categorical level; that is, we construct natural weak equivalences which realize the isomorphisms of derived functors. 

We then verify that these squares do commute in many general settings. In particular, we treat the cases where the horizontals are strong monoidal Quillen adjunctions (Proposition~\ref{prop:comparestrong}), weak monoidal Quillen equivalences (Proposition~\ref{prop:compareweak}) and Quillen equivalences arising from Quillen invariance of modules (Proposition~\ref{prop:ringsquare}), and the vertical functors are given by change of rings adjunctions. The proof of our main theorem then reduces to showing that the zig-zag of Quillen equivalences between (co)free equivariant spectra and their algebraic models satisfy the relevant hypotheses. This requires the construction of twelve different squares of the form shown above which satisfy the relevant hypotheses. See Section~\ref{sec:strategy} for a more comprehensive discussion. We note that the monoidality of the functors is not relevant to the statement of the main theorems; however, the monoidality of the functors is heavily used in the proof of the main theorems through the application of Propositions~\ref{prop:comparestrong},~\ref{prop:ringsquare} and~\ref{prop:compareweak}.

\subsection{Conventions}
We follow the convention of writing the left adjoint on the left in an adjoint pair displayed vertically, and on top in an adjoint pair displayed horizontally. In the second part of this paper, everything is rationalized without comment. In particular, $H^*(-)$ denotes the (unreduced) cohomology with rational coefficients. We write $q\colon QX \to X$ for cofibrant replacement. All monoidal model categories are assumed to be \emph{symmetric} monoidal as in~\cite{SchwedeShipley00}. We use the standard convention that `subgroup' means `closed subgroup'.

\subsection*{Acknowledgements}
I am grateful to John Greenlees for his comments on this paper and many helpful discussions. I would also like to thank Brooke Shipley and Sarah Whitehouse for many useful conversations and suggestions.

\part{The formal setup}
\section{Background}\label{sec:background}
In this section we provide the necessary background. In particular, we recap flat model structures and their importance for our method. We also recall key facts about monoidal structures and Quillen pairs from~\cite{SchwedeShipley03}.
\subsection{Flat model structures}\label{sec:flat}
Let $\C$ be a monoidal model category, and let $S$ be a monoid in $\C$. We write $\mod{S}(\C)$ for the category of (left) $S$-modules. If $S$ is in addition commutative, we may consider the category of commutative $S$-algebras which we denote by $\calg{S}(\C)$. If the underlying category is clear, we will often omit it from the notation. When they exist, we will equip the categories $\calg{S}(\C)$ and $\mod{S}(\C)$ with the model structure whose weak equivalences and fibrations are created by the forgetful functor to $\C$, and we call these model structures \emph{right lifted}. For general conditions under which the right lifted model structures on $\calg{S}(\C)$ and $\mod{S}(\C)$ exist, we refer the reader to~\cite{SchwedeShipley00} and~\cite{White17}. We next record the key examples for this paper.
\begin{egs}\label{rem:existence}
	There is a right lifted model structure on modules in each of the following cases:
	\begin{itemize}
		\item dg-modules over a DGA in the projective model structure~\cite[3.3]{BarthelMayRiehl14};
		\item modules over a ring spectrum in the stable model structure~\cite[III.7.6]{MandellMaySchwedeShipley01} and flat model structure, see~\cite[2.6]{Shipley04} and~\cite[3.2.1]{PavlovScholbach18};
		\item modules over a ring $G$-spectrum in the stable model structure~\cite[III.7.6]{MandellMay02} and flat model structure~\cite[2.3.33]{Stolz11}.
	\end{itemize}
	Similarly, there is a right lifted model structure on commutative algebras in each of the following cases:
	\begin{itemize}
	\item commutative DGAs over a field of characteristic zero in the projective model structure~\cite[\S 5.1]{White17};
	\item commutative algebra spectra in the positive stable model structure and in the positive flat model structure, see~\cite[3.2]{Shipley04} and~\cite[4.1]{PavlovScholbach18};
	\item equivariant commutative algebra spectra in the positive stable model structure and in the positive flat model structure, see~\cite[III.8.1]{MandellMay02} and~\cite[2.3.40]{Stolz11}.
\end{itemize}	
\end{egs}

\begin{defn}\label{defn:convenientmodelstructure}
	Let $\C$ be a monoidal model category and suppose that there exists another monoidal model structure $\widetilde{\C}$ on the same underlying category as $\C$, which has the same weak equivalences and for which the identity functor $\widetilde{\C} \to \C$ is left Quillen. Suppose in addition that the right lifted model structures on $\calg{S}(\widetilde{\C})$ and $\mod{S}(\C)$ exist. We say that $(\C, \widetilde{\C})$ is \emph{convenient} if the forgetful functor $\calg{S}(\widetilde{\C}) \to \mod{S}(\C)$ preserves cofibrant objects, for all commutative monoids $S \in \C$. 
\end{defn}

The following lemma summarizes the key feature of a pair $(\C, \widetilde{\C})$ of convenient model structures.
\begin{lem}\label{lem:convenient}
Suppose that $(\C,\widetilde{\C})$ is convenient and let $\theta\colon S \to R$ be a map of commutative monoids in $\C$. A cofibrant replacement of $R$ as a commutative $S$-algebra in the model structure right lifted from $\widetilde{\C}$ is cofibrant as an $S$-module in the model structure right lifted from $\C$.
\end{lem}

Before turning to the examples, we describe the importance of this property. The restriction of scalars functor along a map of commutative monoids $\theta\colon S \to R$ in a monoidal model category $\C$ is always right Quillen, but it is not left Quillen in general. If the monoidal unit of the underlying category is cofibrant, then restriction of scalars is left Quillen if and only if $R$ is cofibrant as an $S$-module, see Proposition~\ref{prop:extrescoext}. Since a key step in the proof of algebraic models is a formality argument based on the fact that polynomial rings are intrinsically formal as commutative DGAs, one needs to be able to replace $R$ in such a way that it is still a commutative $S$-algebra, and is cofibrant as an $S$-module. If $(\C,\widetilde{\C})$ is convenient, then this is possible by replacing $R$ as a commutative $S$-algebra in the model structure right lifted from $\widetilde{\C}$.

The pair of model structures $(\mathrm{stable}, \mathrm{positive~stable})$ on spectra is \emph{not} convenient. To rectify this, Shipley~\cite{Shipley04} constructs the flat and positive flat model structures on symmetric spectra which are convenient. This has since been generalised by Stolz~\cite{Stolz11} to equivariant spectra and by Pavlov-Scholbach~\cite{PavlovScholbach18} to symmetric spectra in general model categories. It is important to note that in the flat model structures on (equivariant) spectra, the weak equivalences are the same as in the stable model structure. Therefore for each of these flat model structures, the identity functor is a right Quillen equivalence from it to the stable model structure.

In summary, we have the following crucial result. 
\begin{thm}
	The following pairs of model categories are all convenient in the sense of Definition~\ref{defn:convenientmodelstructure}:
	\begin{itemize}
		\item $(\text{projective}, \text{projective})$ on chain complexes over a field of characteristic zero;
		\item $(\text{flat}, \text{positive~flat})$ on symmetric spectra;
		\item $(\text{flat}, \text{positive~flat})$ on symmetric spectra in simplicial $\mathbb{Q}$-modules;
		\item $(\text{flat}, \text{positive~flat})$ on symmetric spectra in non-negatively graded chain complexes of $\mathbb{Q}$-modules;
		\item $(\text{flat}, \text{positive~flat})$ on equivariant spectra.
	\end{itemize}
\end{thm}
\begin{proof}
	The proofs can be found in~\cite[\S 5.1]{White17},~\cite[4.1]{Shipley04}, ~\cite[4.4]{PavlovScholbach18} together with~\cite[4.8]{FlatnessandShipley'sTheorem}, and~\cite[2.3.40]{Stolz11} respectively.
\end{proof}


\subsection{Flat cofibrants}
In a monoidal model category, if $X$ is cofibrant, then $X \otimes -$ is left Quillen and hence preserves weak equivalences between cofibrant objects by Ken Brown's lemma. It is often convenient to work with monoidal model categories which satisfy something stronger. 
\begin{defn}\label{defn:flatcofibrants}
	We say that \emph{cofibrants are flat} in a monoidal model category, if for every cofibrant object $X$, the functor $X \otimes -$ preserves all weak equivalences.
\end{defn}
This is satisfied in all of the model categories of interest in this paper. In particular it holds in the following examples:
\begin{itemize}
	\item \emph{Projective model structure on dg-modules}: The cofibrant objects in this model category are the semi-projective modules $P$, and these have the property that $P \otimes -$ preserves quasiisomorphisms~\cite[11.1.6, 11.2.1]{AvramovFoxbyHalperin03}.
	\item \emph{Equivariant spectra}: Cofibrants are flat in both the stable and flat model structure on $G$-spectra and in modules over a ring $G$-spectrum, see~\cite[7.3, 7.7]{MandellMay02} and~\cite[2.3.40]{Stolz11}.
	\item \emph{Symmetric spectra}: Cofibrant objects are flat in both the stable and flat model structure on symmetric spectra (in general model categories) and in modules over a ring spectrum, see~\cite[5.3.10]{HoveyShipleySmith00} and~\cite[3.5.1]{PavlovScholbach18}.
\end{itemize}

\subsection{Lifting Quillen adjunctions to module categories}\label{sec:liftadjunctionstomodules}
Recall that a Quillen adjunction $F: \C \rightleftarrows \D : U$ between monoidal model categories is said to be a \emph{weak monoidal Quillen adjunction} if the right adjoint $U$ is lax symmetric monoidal (which gives the left adjoint $F$ an oplax symmetric monoidal structure) and the following conditions hold:
\begin{enumerate}
	\item for cofibrant $A$ and $B$ in $\C$, the oplax monoidal structure map $\phi\colon F(A \otimes B) \to FA \otimes FB$ is a weak equivalence in $\D$;
	\item for a cofibrant replacement $c\mathds{1}_\C$ of the unit in $\C$, the map $F(c\mathds{1}_\C) \to \mathds{1}_\D$ is a weak equivalence in $\D$.
\end{enumerate}
If moreover the oplax monoidal structure maps are isomorphisms, we say that it is a \emph{strong monoidal Quillen adjunction}. 

Throughout the paper we will make use of categories of modules and how Quillen pairs are lifted to the categories of modules, for more detail see~\cite[\S 3.3]{SchwedeShipley03}. If $F: \C \rightleftarrows \D : U$ is a weak monoidal Quillen adjunction then $U$ preserves commutative monoids. Let $S$ be a commutative monoid in $\D$. The adjoint lifting theorem~\cite[4.5.6]{Borceux94} shows that the Quillen adjunction lifts to a weak monoidal Quillen adjunction 
\begin{center}\begin{tikzcd} 
	\mod{US}(\C) \arrow[r, "F^S", yshift=1mm] & \mod{S}(\D). \arrow[l, "U", yshift=-1mm]
	\end{tikzcd}\end{center}
If the original pair $(F,U)$ is a strong monoidal Quillen pair, then we have the formula 
$F^SM = S \otimes_{FUS} FM$ and the lifted Quillen pair $(F^S,U)$ is strong monoidal. For example, this arises in the case of the fixed points-inflation adjunction, see Section~\ref{sec:fixedpoints}.
 
In addition, when $(F,U)$ is a strong monoidal Quillen pair, $F$ preserves commutative monoids. It follows that for $S$ a commutative monoid in $\C$, the adjunction lifts to a strong monoidal Quillen adjunction
\begin{center}\begin{tikzcd} 
	\mod{S}(\C) \arrow[r, "F", yshift=1mm] & \mod{FS}(\D). \arrow[l, "U", yshift=-1mm]
\end{tikzcd}\end{center}

\section{Functors between categories of modules}\label{sec:functors}
\subsection{Quillen functors}\label{sec:quillen functors}
In this section we recall various adjunctions between categories of modules which play a crucial role in the remainder of this paper, and give conditions
under which these adjunctions are Quillen adjunctions. 
The results in this subsection are special cases of results of Lewis-Mandell~\cite{LewisMandell} (see Remark~\ref{rem:LewisMandell})
and we refer the reader there for more details and related results.

Let $\C$ be a monoidal model category and $\D$ be a model category whose underlying category is $\C$-enriched. Recall that $\D$ is a $\C$-enriched model category if it satisfies the analogue of Quillen's SM7 axiom, and a unit axiom. We refer the reader to~\cite[\S 1]{Barwick10} for detailed definitions and basic properties of enriched model categories.
\begin{lem}[{\cite[3.9]{LewisMandell}}]\label{lem:enriched}
Let $\C$ be a monoidal model category and let $S$ be a monoid in $\C$. Suppose that the right lifted model structure on $\mod{S}$ exists. Then $\mod{S}$ is a $\C$-enriched model category.
\end{lem}

Recall that an $R$-$S$-bimodule is an $R \otimes S^{\mathrm{op}}$-module.
\begin{prop}\label{prop:bimodule}
Let $\C$ be a monoidal model category, $R$, $S$ and $T$ be monoids in $\C$, and $M$ be an $R$-$S$-bimodule. Suppose that the right lifted model structures on $\mod{R}$, $\mod{R \otimes T^\mathrm{op}}$ and $\mod{S \otimes T^\mathrm{op}}$ exist. If $M$ is a cofibrant $R$-module then the tensor-hom adjunction
\[\begin{tikzcd}
\mod{S \otimes T^\mathrm{op}} \arrow[rrr, "M \otimes_S -", yshift=1mm] & & & \mod{R \otimes T^\mathrm{op}} \arrow[lll, "{\mathrm{Hom}_R(M,-)}", yshift=-1mm]
\end{tikzcd}\]
is a Quillen adjunction. 
\end{prop}
\begin{proof}
As $\mod{R}$ is a $\C$-enriched model category by Lemma~\ref{lem:enriched}, it follows by lifting properties that for $i\colon N \to N'$ a cofibration in $\mod{R}$ and $p\colon E \to B$ a fibration in $\mod{R}$ that $$\mathrm{Hom}_R(N',E) \to \mathrm{Hom}_R(N,E) \times_{\mathrm{Hom}_R(N,B)} \mathrm{Hom}_R(N',B)$$ is a fibration which is acyclic if either $i$ or $p$ is. Taking $i$ to be the map $\emptyset \to M$ shows that $\mathrm{Hom}_R(M,-)$ preserves fibrations and acyclic fibrations of $R$-modules if $M$ is cofibrant as an $R$-module. Since fibrations and weak equivalences of modules and bimodules are created in $\C$, this in particular shows that if $M$ is a cofibrant $R$-module, then $\mathrm{Hom}_R(M,-)$ sends fibrations (resp., acyclic fibrations) of $R$-$T$-bimodules to fibrations (resp., acyclic fibrations) of $S$-$T$-bimodules, and hence is right Quillen.
\end{proof}

We now give a criterion to force the cofibrancy condition in the previous result.
\begin{lem}\label{lem:cofibrantbimodule}
Let $\C$ be a monoidal model category, $R$ and $S$ be monoids in $\C$. Suppose that the right lifted model structures on $\mod{R}$ and $\mod{R \otimes S^{\mathrm{op}}}$ exist. If $S$ is cofibrant in $\C$, then any cofibrant $R$-$S$-bimodule is cofibrant as an $R$-module. 
\end{lem}
\begin{proof}
The forgetful functor $\mod{R \otimes S^{\mathrm{op}}} \to \mod{R}$ has a right adjoint given by $\mathrm{Hom}(S,-)$. This is a special case of the tensor-hom adjunction in Proposition~\ref{prop:bimodule} and the result follows from there.
\end{proof}

Given a map $\theta\colon S \to R$ of monoids in a monoidal category $\C$, there is an induced adjoint triple
\[\begin{tikzcd}
\mod{S} \arrow[rr, "\ext" description, yshift=3mm] \arrow[rr, "\coext" description, yshift=-3mm] & & \mod{R} \arrow[ll, "\res" description]
\end{tikzcd} \]
where the functors are defined as follows. For an $R$-module $M$ and an $S$-module $N$ we have:
\begin{itemize}
\item \emph{extension of scalars}: $\ext(N) = R \otimes_S N$;
\item \emph{restriction of scalars}: $\res(M)$ is the $S$-module with underlying object $M$ and action defined by the composite $S \otimes M \to R \otimes M \to M$;
\item \emph{coextension of scalars}: $\coext(N) = \text{Hom}_S(R, N)$.
\end{itemize}

\begin{prop}\label{prop:extrescoext}
Let $\C$ be a monoidal model category, and let $\theta\colon S \to R$ be a map of monoids in $\C$. Suppose that the right lifted model structures on $\mod{R}$ and $\mod{S}$ exist. 
\begin{enumerate}
\item The extension-restriction of scalars adjunction \[\begin{tikzcd}
\mod{S} \arrow[r, "\ext", yshift=1mm] & \mod{R} \arrow[l, "\res", yshift=-1mm]
\end{tikzcd} \] is a Quillen adjunction.
\item If $R$ is a cofibrant $S$-module then the restriction-coextension of scalars adjunction
\[\begin{tikzcd}\mod{S} \arrow[r, "\coext"', yshift=-1mm] & \mod{R} \arrow[l, "\res"', yshift=1mm] \end{tikzcd}\] is a Quillen adjunction. Moreover, if the unit of $\C$ is cofibrant, then the restriction-coextension of scalars adjunction is a Quillen adjunction if and only if $R$ is cofibrant as an $S$-module.
\end{enumerate}
\end{prop}
\begin{proof}
For (1), the restriction of scalars preserves fibrations and weak equivalences since the model structure on modules is right lifted from $\C$. The adjunction in (2) is a special case of the tensor-hom adjunction in Proposition~\ref{prop:bimodule} and the claim that it is Quillen if $R$ is cofibrant as an $S$-module follows immediately from there. Conversely, if $\mathds{1}$ is cofibrant in $\C$, then $R$ is a cofibrant $R$-module. Applying the left Quillen functor $\res$ then shows that $R$ is a cofibrant $S$-module.
\end{proof}

\begin{rem}\label{rem:LewisMandell}
Lewis-Mandell~\cite{LewisMandell} prove more general versions of the above statements by introducing a notion of semicofibrancy, see in particular~\cite[1.7(c)]{LewisMandell}. The relation between their results and those presented in this section can be seen from~\cite[6.1]{LewisMandell} where it is shown that any cofibrant object is semicofibrant, and moreover that if $\C$ has cofibrant unit, then semicofibrancy is equivalent to cofibrancy. 
\end{rem}

\subsection{The derived story}
In this section we use the results of the previous subsection to explore the situation at the derived level. For more details, we direct the reader to~\cite[\S 3.D]{Greenleeschangeofgroups} and~\cite{BDS}, but note that these references work purely in the derived categories.

If $\C$ is a stable, monoidal model category, its homotopy category is a tensor-triangulated category by~\cite[4.3.2, 6.6.4, \S 7.1]{Hovey99}. In other words, the homotopy category of $\C$ is a triangulated category with a compatible closed symmetric monoidal structure, see for instance~\cite[A.2.1]{HoveyPalmieriStrickland97}. It is often convenient to assume that the homotopy category of $\C$ is rigidly-compactly generated, meaning that the small (a.k.a compact) and dualizable objects coincide; see~\cite[2.7]{BDS} for more details. In the language of~\cite{HoveyPalmieriStrickland97}, rigidly-compactly generated tensor-triangulated categories are the unital algebraic stable homotopy categories. If the homotopy category of $\C$ is rigidly-compactly generated, it also follows that for a commutative monoid $S$ in $\C$ that the derived category $\Dsf(S)$ of $S$-modules (when it exists) is a rigidly-compactly generated tensor-triangulated category. 

\begin{rem}
The condition that the homotopy category of $\C$ is rigidly-compactly generated is implied by certain model categorical assumptions. Since this is not relevant for this paper, we refer the interested reader to~\cite[1.2.3.8]{ToenVezzosi08} and~\cite[\S 4]{BalchinGreenlees20} for further discussion.
\end{rem}

Let us now set a hypothesis which will form our standard setup. 
\begin{hyp}\label{hyp:cofibrancy}
Let $\C$ be a stable, monoidal model category whose homotopy category is rigidly-compactly generated,
and let $\theta\colon S \to R$ be a map of commutative monoids in $\C$. Suppose that:
\begin{enumerate}
\item the right lifted model structures on $\mod{R}$, $\mod{S}$ and $\mod{S \otimes R^\mathrm{op}}$ exist and are monoidal;
\item $R$ is a cofibrant $S$-module.
\end{enumerate}
\end{hyp}
In this subsection, we will in fact in addition require that $R$ and $S$ are cofibrant in $\C$. However, in the next subsection we explain why for the remainder of the paper this additional requirement will not be necessary. 

Given a map $\theta\colon S \to R$ of commutative monoids in a stable monoidal model category $\C$, we note that if $\C$ fits into a convenient pair $(\C,\widetilde{\C})$ as in Definition~\ref{defn:convenientmodelstructure} then there exists a weakly equivalent replacement of $\theta$ satisfying Hypothesis~\ref{hyp:cofibrancy} as described in Lemma~\ref{lem:convenient}. In a similar way, by first cofibrantly replacing $S$ in $\mathrm{CMon}(\widetilde{\C})$ one may additionally satisfy the hypothesis that $R$ and $S$ are cofibrant in $\C$. 

Fix a map $\theta\colon S \to R$ of commutative monoids in a stable monoidal model category $\C$ which satisfies Hypothesis~\ref{hyp:cofibrancy}. Since $R$ is a cofibrant $S$-module we have Quillen adjunctions
\[\begin{tikzcd}
\mod{S} \arrow[rr, "\ext" description, yshift=3mm] \arrow[rr, "\coext" description, yshift=-3mm] & & \mod{R} \arrow[ll, "\res" description]
\end{tikzcd} \]
by Proposition~\ref{prop:extrescoext}. We write $\dual R = \mathrm{Hom}_S(R,fS)$ for the relative dualizing complex where $fS$ is a fibrant replacement of $S$ as an $S$-module, and note that this is an $S$-$R$-bimodule. Now assume in addition that $R$ and $S$ are cofibrant in $\C$. We may cofibrantly replace $\dual R$ as an $S$-$R$-bimodule, and this is then both cofibrant as an $R$-module and as an $S$-module by Lemma~\ref{lem:cofibrantbimodule}. For ease of reading, we omit this replacement from the notation. Proposition~\ref{prop:bimodule} then gives us Quillen adjunctions
\[\begin{tikzcd}
\mod{S} \arrow[r, "\theta_{(*)}"', yshift=-1mm] & \mod{R} \arrow[l, "\twext"', yshift=1mm] 
\end{tikzcd} \quad \mathrm{and} \quad \begin{tikzcd}
\mod{S} \arrow[r, "\theta_{(!)}", yshift=1mm] & \mod{R} \arrow[l, "\theta^!", yshift=-1mm] 
\end{tikzcd}
\]
where the functors are defined as follows. For an $R$-module $M$ and an $S$-module $N$ we define
\begin{itemize}
\item $\twext(M) = \dual R \otimes_R M$;
\item $\theta_{(*)}(N) = \text{Hom}_S(\dual R,N)$;
\item $\theta_{(!)}(N) = \dual R \otimes_S N$;
\item $\theta^!(M) = \mathrm{Hom}_R(\dual R, M)$.
\end{itemize}
The functors $\twext$ and $\theta_{(*)}$ will be the most important of these for the remainder of this paper. We call $\twext$ the \emph{twisted extension of scalars} functor, and $\theta_{(*)}$ the \emph{twisted coextension of scalars} functor.

\begin{rem}
The assumption that $R$ is a cofibrant $S$-module is crucial in the consideration of the relative dualizing complex $\dual R$. Without this assumption, $\dual R$ will not represent the correct homotopy type. For example, consider the ring map $\theta\colon \mathbb{Q}[c] \to \mathbb{Q}$ where $|c|=-2$. This example is pertinent since it is the map $H^*BSO(2) \to H^*B1$. Then $\mathrm{Hom}_{\mathbb{Q}[c]}(\mathbb{Q},\mathbb{Q}[c]) = 0$ but $\R\mathrm{Hom}_{\mathbb{Q}[c]}(\mathbb{Q},\mathbb{Q}[c]) \simeq \Sigma\mathbb{Q}$. 
\end{rem}

Under an additional assumption we now explain how to relate these three sets of adjunctions. Recall that an $S$-module $M$ is said to be \emph{(derived) small} if the natural map \[\bigoplus [M, N_i]^S \to [M, \bigoplus N_i]^S\] is an isomorphism for any set $\{N_i\}$ of $S$-modules. In practice this can often be checked using the following result.
\begin{prop}[{\cite[10.2]{Greenlees18c}}]\label{prop:small fin gen} 
Let $S$ be a commutative ring spectrum such that $\pi_*S$ has finite global dimension. Then an $S$-module $M$ is small if and only if $\pi_*M$ is a finitely generated $\pi_*S$-module. 
\end{prop}
By Venkov's theorem~\cite{Venkov} (or alternatively see~\cite[3.10.1]{Bensonbook}), $H^*BH$ is a finitely generated $H^*BG$-module. Since $H^*BG$ is polynomial, it has finite global dimension. Therefore, any map $\theta\colon S \to R$ of commutative ring spectra with $\pi_*S = H^*BG$ and $\pi_*R = H^*BH$ has the property that $R$ is a small $S$-module. The majority of ring maps that we use in this paper fall into this category. 

Before we can summarise the salient points of this framework, we recall the following well-known lemma. Let $M$ be an $S$-$R$-bimodule which is cofibrant as an $S$-module. By Proposition~\ref{prop:bimodule}, the tensor-hom adjunction \[\begin{tikzcd}
\mod{S \otimes R^\mathrm{op}} \arrow[rrr, "M \otimes_S -", yshift=1mm] & & & \mod{S \otimes R^\mathrm{op}} \arrow[lll, "{\mathrm{Hom}_S(M,-)}", yshift=-1mm]
\end{tikzcd}\] is a Quillen adjunction, and therefore we obtain a derived adjunction. We denote the functional dual of $M$ by $\dual M = \R\mathrm{Hom}_S(M,S)$. For any $S$-module $N$, there is a natural map \[\psi_{M,N}\colon \dual M \otimes_S^\L N \to \R\mathrm{Hom}_S(M,N)\] of $S$-$R$-bimodules defined to be the transpose of the map $$\dual M \otimes_S^\L N \otimes_S^\L M \cong \dual M \otimes_S^\L M \otimes_S^\L N \xrightarrow{\mathrm{ev} \otimes 1} S \otimes_S^\L N \cong N.$$ There is also a natural map $\varphi_M\colon M \to \dual^2M$ of $S$-$R$-bimodules defined to be the transpose of the evaluation map $$\dual M \otimes_S^\L M \xrightarrow{\mathrm{ev}} S.$$
\begin{lem}\label{lem:dualizable}
Let $\theta\colon S \to R$ be a map of commutative monoids in $\C$ as in Hypothesis~\ref{hyp:cofibrancy}, and let $M$ be an $S$-$R$-bimodule which is cofibrant and small as an $S$-module.
\begin{enumerate}
\item For any $S$-module $N$, the natural map $\psi_{M,N}\colon \dual M \otimes_S^\L N \to \R\mathrm{Hom}_S(M,N)$ is an isomorphism in the homotopy category of $S$-$R$-bimodules.
\item The natural map $\phi_M\colon M \to \dual^2M$ is an isomorphism in the homotopy category of $S$-$R$-bimodules.
\end{enumerate}
\end{lem}
\begin{proof}
Since equivalences of bimodules and modules are underlying, it is sufficient to show that each map is an equivalence in $\Dsf(S)$. Since the homotopy category of $\C$ is assumed to be rigidly-compactly generated, it follows that $\Dsf(S)$ is rigidly-compactly generated. The result then follows, since in the rigidly-compactly generated setting an object is dualizable if and only if it is small.
\end{proof}

The following proposition summarises the key features of this setup.
\begin{prop}\label{prop:smallsetup}
Let $\theta\colon S \to R$ be a map of commutative monoids in $\C$ as in Hypothesis~\ref{hyp:cofibrancy} and suppose in addition that $R$ and $S$ are cofibrant in $\C$. If $R$ is a small $S$-module then:
\begin{enumerate} 
\item there are natural isomorphisms $\alpha\colon \L\ext \xrightarrow{\sim} \R\theta_{(*)}$ and $\beta\colon \L\theta_{(!)} \xrightarrow{\sim} \R\coext$ of derived functors;
\item there are adjunctions 
\[\begin{tikzcd}
\Dsf(S) \arrow[rrr, "\L\ext = \R\psext" description, yshift=4mm] \arrow[rrr, "\R\coext = \L\theta_{(!)}" description, yshift=-4mm] & & & \Dsf(R). \arrow[lll, "\L\res = \R\res" description] \arrow[lll, "\L\twext" description, yshift=8mm] \arrow[lll, "\R\theta^{!}" description, yshift=-8mm]
\end{tikzcd} \]
\end{enumerate}
\end{prop}
\begin{proof}
The hypotheses ensure that all of the functors are Quillen functors as explained above and hence that all of the derived functors exist. Note that since $R$ is a small $S$-module by assumption, we have that $\dual R$ is also a small $S$-module. We define $\alpha\colon \L\ext \Rightarrow \R\psext$ to be the composite $$\L\ext(N) = R \otimes_S^\L N \xrightarrow{\varphi_R \otimes 1} \dual^2R \otimes_S^\L N \xrightarrow{\psi_{\dual R,N}} \R\text{Hom}_S(\dual R,N) = \R\psext(N).$$ By Lemma~\ref{lem:dualizable} we have that $\alpha$ is an equivalence.
In a similar way, we define $\beta\colon \L\theta_{(!)} \Rightarrow \R\theta_!$ by 
\[\L\theta_{(!)}(N) = \dual R \otimes_S^\L N \xrightarrow{\psi_{R,N}} \R\mathrm{Hom}_S(R,N) = \R\coext(N)\] and this is also an equivalence by Lemma~\ref{lem:dualizable}. This proves (1), and the claim in (2) then follows by the tensor-hom adjunction, noting that $\L\res \cong \R\res$ since $\res$ is both left and right Quillen.
\end{proof}

\subsection{Relatively Gorenstein maps}
In this subsection we revisit the observations of the previous subsection under an additional hypothesis on the map $\theta\colon S \to R$. This additional hypothesis will allow us to show that the derived functors $\L\ext$ and $\R\coext$ are isomorphic up to a shift.
\begin{defn}
Let $\theta\colon S \to R$ be a map of commutative monoids in $\C$ as in Hypothesis~\ref{hyp:cofibrancy}. We say that $\theta$ is \emph{relatively Gorenstein of shift $a$} if $\R\coext(S) \simeq \Sigma^a R$, for some $a \in \mathbb{Z}$. We note that $\coext$ is right Quillen by Proposition~\ref{prop:extrescoext} since $R$ is a cofibrant $S$-module.
\end{defn}

\begin{rem}
The original definition of relatively Gorenstein maps given by Greenlees is made at the derived level, see~\cite{Greenlees18c}. Since we are interested in understanding the model categorical implications of such a statement, our definition is a modified version of that given in~\cite{Greenlees18c}. 
\end{rem}

All the relatively Gorenstein maps that we will use arise from one particular example. Whilst we do not state the following theorem of Greenlees~\cite{Greenlees20} in its full generality, we note that the fact that we work with connected compact Lie groups is vital. Working rationally, we write $DBG_+ = F(BG_+,S^0)$ and note that this is a special case of the cochain spectrum $C^*(BG,k) = F(BG_+,Hk)$ for $k=\mathbb{Q}$, since the rational sphere spectrum is equivalent to $H\mathbb{Q}$. In the following theorem, in order to satisfy Hypotheses~\ref{hyp:cofibrancy}, we should replace the map $\theta$ as in Lemma~\ref{lem:convenient}. For ease of reading we do not record this in the notation.
\begin{thm}[{\cite[6.1]{Greenlees20}}]\label{thm:relgor}
	Let $H$ be a subgroup of $G$ where both $H$ and $G$ are connected compact Lie groups. The ring map $\theta\colon DBG_+ \to DBH_+$ is relatively Gorenstein of shift $d$, where $d$ is the codimension of $H$ in $G$.
\end{thm}
In the next subsection, we will show that relatively Gorenstein maps are preserved by monoidal Quillen equivalences, see Propositions~\ref{prop:relGorstrong} and~\ref{prop:relGorweak}. This will make precise the above claim that all of the relatively Gorenstein maps which we require in this paper arise from the case of $DBG_+ \to DBH_+$.

The following result summarises the key points of this section. 
\begin{prop}\label{prop:functorsetup}
	Let $\theta\colon S \to R$ be a map of commutative monoids in $\C$ as in Hypothesis~\ref{hyp:cofibrancy}. If $R$ is small as an $S$-module and $\theta$ is relatively Gorenstein of shift $a$ then:
	\begin{enumerate}
	\item $\ext : \mod{S} \rightleftarrows \mod{R} : \res$ and $\res : \mod{R} \rightleftarrows \mod{S} : \coext$ are Quillen adjunctions;
	\item there is a natural isomorphism of derived functors $\L\ext \cong \Sigma^{-a}\R\coext$.	\end{enumerate}
\end{prop}
\begin{proof}
	The fact that $\ext \dashv \res$ and $\res \dashv \coext$ are Quillen adjunctions follows from Proposition~\ref{prop:extrescoext}. We have natural isomorphisms of derived functors
	\[\L\ext M = R \otimes_S^\L M \simeq \Sigma^{-a}\dual R \otimes_S^\L M \xrightarrow{\sim} \Sigma^{-a}\R\mathrm{Hom}_S(R, M) = \Sigma^{-a}\R\coext M\]
	where the first isomorphism comes from the fact that $\theta$ is relatively Gorenstein, and the second isomorphism from Lemma~\ref{lem:dualizable} as $R$ is small as an $S$-module.
\end{proof}

Therefore given a map $\theta\colon S \to R$ as in Hypothesis~\ref{hyp:cofibrancy} which is relatively Gorenstein of shift $a$, we have Quillen functors
\begin{center}
\begin{tikzcd}
	\quad\mod{S}\quad \arrow[dd, xshift=4mm, "\ext" description] \arrow[dd, xshift=-4mm, "\Sigma^{-a}\coext" description] \\
	\\
	\quad\mod{R}\quad \arrow[uu, xshift=12mm, "\res"' description] \arrow[uu, xshift=-12mm, "\Sigma^a\res" description] 
\end{tikzcd}
\end{center}
such that $\Sigma^{-a}\R\coext \cong \L\ext$. This motivates the general framework in which we will work in Section~\ref{sec:compare}.

\begin{rem}
It is worth noting the relation of the results in the section to the results of~\cite{BDS} and their trichotomy theorem. In their terminology, Proposition~\ref{prop:smallsetup} is a Grothendieck-Neeman duality context as in~\cite[1.7]{BDS}, and Proposition~\ref{prop:functorsetup} is an instance of a Wirthm\"uller isomorphism context as in~\cite[1.9]{BDS}. Note that there is a unfortunate notational clash though; what we denote by $\theta_*$ is their $f^*$, and similar clashes in notation occur for the other functors.
\end{rem}

\subsection{Relatively Gorenstein maps along Quillen equivalences}
In this section, we prove that relatively Gorenstein maps are preserved by monoidal Quillen equivalences. Firstly we recall the following lemma. 
\begin{lem}\label{lem:internalhom}
Let $F: \C \rightleftarrows \D :U$ be an adjunction between closed symmetric monoidal categories, and assume that $F$ is strong symmetric monoidal. For any $X \in \C$ and $Y \in \D$, there is a natural isomorphism $U\underline{\mathrm{Hom}}(FX,Y) \xrightarrow{\sim} \underline{\mathrm{Hom}}(X,UY)$ where $\underline{\mathrm{Hom}}$ denotes the internal hom.
\end{lem}
\begin{proof}
The natural map $U\underline{\mathrm{Hom}}(FX,Y) \to \underline{\mathrm{Hom}}(X,UY)$ is adjunct to the natural map $$F(X \otimes U\underline{\mathrm{Hom}}(FX,Y)) \to FX \otimes FU\underline{\mathrm{Hom}}(FX,Y) \xrightarrow{1 \otimes \varepsilon} FX \otimes \underline{\mathrm{Hom}}(FX,Y) \xrightarrow{\mathrm{ev}} Y$$ constructed using the oplax monoidal structure map of $F$. A standard adjunction argument using the Yoneda lemma allows one to conclude.
\end{proof}

\begin{prop}\label{prop:relGorstrong}
Let $F\colon \C \rightleftarrows \D : U$ be a strong monoidal Quillen equivalence between stable monoidal model categories whose homotopy categories are rigidly-compactly generated. Suppose that $\C$ and $\D$ both have cofibrant units, and both satisfy Quillen invariance. Let $\theta\colon S \to R$ be a relatively Gorenstein map of shift $a$ of commutative monoids in $\C$ and suppose that $R$ is a cofibrant $S$-module. Suppose that the right lifted model structures on $\mod{R}$, $\mod{S}$, $\mod{FR}$, $\mod{FS}$ exist and are monoidal. If $F$ preserves weak equivalences then $F\theta\colon FS \to FR$ is relatively Gorenstein of shift $a$.
\end{prop}
\begin{proof}
Since $F$ preserves all weak equivalences, by~\cite[7.1]{FlatnessandShipley'sTheorem} (see also~\cite[3.12]{SchwedeShipley03}) the induced Quillen adjunctions
\[\begin{tikzcd}
\mod{S} \arrow[r, "F", yshift=1mm] & \mod{FS} \arrow[l, "U", yshift=-1mm]
\end{tikzcd} \quad \mathrm{and} \quad \begin{tikzcd}
\mod{R} \arrow[r, "F", yshift=1mm] & \mod{FR} \arrow[l, "U", yshift=-1mm]
\end{tikzcd} \]
are strong monoidal Quillen equivalences. As such the derived adjunctions are strong monoidal equivalences. For ease of reading we omit notation for derived functors in the following display; all functors are implicitly derived. We then have
$$(F\theta)_! (FS) \simeq FU\mathrm{Hom}_{FS}(FR,FS) \simeq F\mathrm{Hom}_{S}(R,UFS) \simeq F\mathrm{Hom}_{S}(R,S) \simeq \Sigma^aFR$$ using Lemma~\ref{lem:internalhom}. This shows that $F\theta$ is relatively Gorenstein of shift $a$. 
\end{proof}

\begin{prop}\label{prop:relGorweak}
Let $F\colon \C \rightleftarrows \D : U$ be a strong monoidal Quillen equivalence between stable monoidal model categories whose homotopy categories are rigidly-compactly generated. Suppose that $\C$ and $\D$ both have cofibrant units, and both satisfy Quillen invariance. 
Suppose moreover that $U$ preserves colimits and that the homotopy category of $\C$ is compactly generated by its unit.
Let $\theta\colon S \to R$ be a map of commutative monoids in $\D$ such that $R$ is a cofibrant $S$-module and write $U\theta\colon US \to UR.$
 Suppose that there exists another model structure $\widetilde{\C}$ on the same underlying category as $\C$ so that $(\C, \widetilde{\C})$ is convenient (see Definition~\ref{defn:convenientmodelstructure}). 
 Let $q\colon QUR \to UR$ be a cofibrant replacement of $UR$ in $\calg{US}(\widetilde{\C})$ and write $\psi\colon US \to QUR$ for the unit map of the $US$-algebra structure on $QUR$. 
 Suppose that the right lifted model structures on $\mod{R}(\D)$, $\mod{S}(\D)$, $\mod{UR}(\C)$, $\mod{US}(\C)$ and $\mod{QUS}(\C)$ exist and are monoidal. If $\theta$ is relatively Gorenstein of shift $a$ and $U$ preserves weak equivalences, then $\psi\colon US \to QUR$ is relatively Gorenstein of shift $a$. 
\end{prop}
\begin{proof}
As $U$ preserves all weak equivalences, by~\cite[3.12(2)]{SchwedeShipley03} the induced Quillen adjunctions
\[\begin{tikzcd}
\mod{US} \arrow[r, "F^S", yshift=1mm] & \mod{S} \arrow[l, "U", yshift=-1mm]
\end{tikzcd} \quad \mathrm{and} \quad \begin{tikzcd}
\mod{UR} \arrow[r, "F^R", yshift=1mm] & \mod{R} \arrow[l, "U", yshift=-1mm]
\end{tikzcd} \]
are Quillen equivalences. In fact, these are moreover weak monoidal Quillen equivalences as we now argue; this was already observed in a special case in~\cite[Proof of 7.2]{FlatnessandShipley'sTheorem}. Since $U$ is lax monoidal and preserves colimits, for any $M,N \in \mod{S}$ we have a map
\begin{align*}
	UM \otimes_{US} UN &= \mathrm{coeq}(UM \otimes US \otimes UN \rightrightarrows UM \otimes UN) \\
	&\to \mathrm{coeq}(U(M \otimes S \otimes N) \rightrightarrows U(M \otimes N)) \\
	&\cong U(\mathrm{coeq}(M \otimes S \otimes N \rightrightarrows M \otimes N) \\
	&= U(M \otimes_S N)
	\end{align*}
which gives $U\colon \mod{S} \to \mod{US}$ a lax monoidal structure. It remains to argue that the Quillen adjunction $F^S \dashv U$ is a weak monoidal Quillen adjunction. For this we use the criterion given in~\cite[3.17]{SchwedeShipley03}. Since $US$ is a cofibrant $US$-module (as the unit of $\C$ is assumed to be cofibrant), the first condition is that $F^SUS \to S$ is a weak equivalence. Since $U$ preserves all weak equivalences and $US$ is cofibrant, this is the derived counit of the Quillen equivalence $F^S \dashv U$ and as such is a weak equivalence. The second condition we must verify is that $US$ is a generator for the homotopy category of $\mod{US}$. This follows from the fact the unit of $\C$ is assumed to be a generator for the homotopy category of $\C$. Therefore $F^S \dashv U$ is a weak monoidal Quillen equivalence, and similarly, $F^R \dashv U$ is also a weak monoidal Quillen equivalence.

These weak monoidal Quillen equivalences pass to give strong monoidal equivalences of the derived categories. Note that since the relative Gorenstein condition is a derived condition, it is enough to check that $\R\mathrm{Hom}_{US}(UR, US) \simeq \Sigma^a UR$. Omitting notation for derived functors in what follows we then have
\[\mathrm{Hom}_{US}(UR, US) \simeq U\mathrm{Hom}_{S}(F^SUR, S) \simeq U\mathrm{Hom}_S(R,S) \simeq \Sigma^a UR\]
using Lemma~\ref{lem:internalhom}, and hence $\psi$ is relatively Gorenstein of shift $a$ as required.
\end{proof}

We will use the previous results in conjunction with Theorem~\ref{thm:relgor} to show that all of the required maps in the second part of this paper are relatively Gorenstein.

\section{Comparing Quillen functors}\label{sec:compare}
In this section, we set up the general techniques for showing when Quillen functors correspond.

\subsection{The calculus of mates}
We firstly recap the calculus of mates, see~\cite[\S 2]{KellyStreet74} for a comprehensive account.

Consider the diagram
\begin{center}
	\begin{tikzcd}
	\C_G \arrow[r, yshift=1mm, "F_G"] \arrow[d, "L"'] & \D_G \arrow[l, yshift=-1mm, "U_G"] \arrow[d, "L'"] \\
	\C_H \arrow[r, yshift=1mm, "F_H"] & \D_H \arrow[l, yshift=-1mm, "U_H"] 
	\end{tikzcd}
\end{center}
in which $F_G \dashv U_G$ and $F_H \dashv U_H$ are adjunctions and $L$ and $L'$ are functors (not necessarily left adjoints). As stated in the introduction, in this general framework the notation is only suggestive and does not indicate the existence of any group actions.

Given a natural transformation $\alpha: F_HL \Rightarrow L'F_G$, one can define its \emph{mate} $\overline{\alpha}: LU_G \Rightarrow U_HL'$ to be the natural transformation $$LU_G \xRightarrow{\eta LU_G} U_HF_HLU_G \xRightarrow{U_H\alpha U_G} U_HL'F_GU_G \xRightarrow{U_HL'\epsilon} U_HL'.$$ Conversely, given $\beta: LU_G \Rightarrow U_HL'$ one defines its mate $\overline{\beta}: F_HL \Rightarrow L'F_G$ by the composite $$F_HL \xRightarrow{F_HL\eta} F_HLU_GF_G \xRightarrow{F_H\beta F_G} F_HU_HL'F_G \xRightarrow{\epsilon L'F_G} L'F_G.$$ These two operations are inverse to one another by the triangle identities and therefore give a bijection between the two kinds of natural transformation.

In general, the mate of a natural isomorphism need not be a natural isomorphism. However, this does hold in certain cases.
\begin{prop}\label{prop:mates}
	If $F_G \dashv U_G$ and $F_H \dashv U_H$ are adjoint equivalences, then a natural transformation $F_HL \Rightarrow L'F_G$ is a natural isomorphism if and only if its mate is.
\end{prop}
\begin{proof}
	Since $F_G \dashv U_G$ and $F_H \dashv U_H$ are adjoint equivalences, the units and counits are natural isomorphisms. The result then follows by the 2-out-of-3 property of isomorphisms.
\end{proof}

Unfortunately, there is no analogous result at the level of model categories, saying that a map is a natural weak equivalence if and only if its mate is. However, the mates correspondence does give us a way to attack questions about the commutativity of derived functors. 

Note that natural weak equivalences of Quillen functors of the same handedness pass to natural isomorphisms of the derived functors by virtue of the natural isomorphism $\L (F\circ F') \cong \L F \circ \L F'$. However, a natural weak equivalence between composites of left and right Quillen functors does not imply that the composites of derived functors are naturally isomorphic. An explicit counterexample can be found in~\cite[0.0.1]{MaySigurdsson06}. 

We give a short overview of how we will use the machinery of mates. Suppose we are given a square 
\begin{center}
	\begin{tikzcd}
	\C_G \arrow[r, yshift=1mm, "F_G"] \arrow[d, "L"'] & \D_G \arrow[l, yshift=-1mm, "U_G"] \arrow[d, "L'"] \\
	\C_H \arrow[r, yshift=1mm, "F_H"] & \D_H \arrow[l, yshift=-1mm, "U_H"] 
	\end{tikzcd}
\end{center}
in which $L$ and $L'$ are left Quillen functors and $F_G \dashv U_G$ and $F_H \dashv U_H$ are Quillen equivalences. Suppose we want to know that $\L L \circ \R U_G \cong \R U_H \circ \L L'$. Instead we can check that there is a natural weak equivalence $L'F_G \simeq F_HL$ on cofibrant objects, and then since they have the same handedness, we have a natural isomorphism $\L L' \circ \L F_G \cong \L F_H \circ \L L$. Since the Quillen equivalences will descend to adjoint equivalences at the level of derived categories, it follows from Proposition~\ref{prop:mates} that taking mates gives the desired natural isomorphism $\L L \circ \R U_G \cong \R U_H \circ \L L'$.

\begin{rem}
	Shulman~\cite{Shulman11} develops a method for comparing composites of left and right derived functors. Since we will be interested in the case where the horizontal adjunctions are Quillen equivalences, we can always compare composites of left and right derived functors by instead comparing their mates, as described above. This allows us to only ever have to consider Quillen functors of the same handedness. Therefore, Shulman's method is unnecessary for our purposes.
\end{rem}

\subsection{Proving commutation of derived functors}
Consider the diagram
\begin{center}
	\begin{tikzcd}
	\quad\C_G\quad \arrow[dd, xshift=2.5mm, "\ext" description] \arrow[dd, xshift=-2.5mm, "\psext" description] \arrow[rr, yshift=1.7mm, "F_G"' description] & & \quad\D_G\quad \arrow[dd, xshift=2.75mm, "\exta" description] \arrow[dd, xshift=-2.75mm, "\psexta" description] \arrow[ll, yshift=-1.7mm, "U_G"' description] \\
	& & \\
	\quad\C_H\quad \arrow[uu, xshift=7.5mm, "\res"' description] \arrow[uu, xshift=-7.5mm, "\twext" description] \arrow[rr, yshift=1.7mm, "F_H"' description] & & \quad\D_H\quad \arrow[uu, xshift=8mm, "\resa"' description] \arrow[uu, xshift=-8mm, "\twexta" description] \arrow[ll, yshift=-1.7mm, "U_H"' description]
	\end{tikzcd}
\end{center}
of model categories and Quillen functors, where $\R\psext \cong \L\ext$ and $\R\psext[\phi] \cong \L\ext[\phi]$, and $F_G \dashv U_G$ and $F_H \dashv U_H$ are Quillen equivalences. We note that we have chosen names in keeping with the example that we have in mind, where the top horizontal is a Quillen equivalence arising in the construction of an algebraic model for $G$-spectra. We are in interested in using a model categorical approach to show that all of the six squares of derived functors listed in Section~\ref{sec:summary of method} commute.
In this section we prove that in order to check that all six of the derived squares commute, it is sufficient to check that only two squares commute. 

\begin{lem}\label{lem:adjoints}
	Let \begin{center}\begin{tikzcd}
		\C_G \arrow[dd, xshift=-2mm, "\ext" description] \arrow[rr, yshift=1.7mm, "F_G"' description] & & \D_G \arrow[dd, xshift=-2mm, "\exta" description] \arrow[ll, yshift=-1.7mm, "U_G"' description] \\
		& & \\
		\C_H \arrow[uu, xshift=2mm, "\res"' description] \arrow[rr, yshift=1.7mm, "F_H"' description] & & \D_H \arrow[uu, xshift=2mm, "\resa"' description] \arrow[ll, yshift=-1.7mm, "U_H"' description]
		\end{tikzcd}\end{center}
	be a diagram of model categories and Quillen functors. 
	\begin{enumerate}
		\item There is a natural transformation $\exta F_G \Rightarrow F_H\ext$ whose restriction to cofibrant objects is a natural weak equivalence if and only if there is a natural transformation $\res U_H \Rightarrow U_G\resa$ whose restriction to fibrant objects is a natural weak equivalence.
		\item There is a natural transformation $F_H\ext \Rightarrow \exta F_G$ whose restriction to cofibrant objects is a natural weak equivalence if and only if there is a natural transformation $U_G\resa \Rightarrow \res U_H$ whose restriction to fibrant objects is a natural weak equivalence.
	\end{enumerate}
\end{lem}
\begin{proof}
	This follows from applying~\cite[1.4.4(b)]{Hovey99}.
\end{proof}

We now state the theorem which allows us to reduce the problem of checking that all six squares commute, to only checking that two squares commute. We emphasize that in the following result saying that ``there is a natural transformation between $F$ and $G$'' means that there is either a natural transformation $F \Rightarrow G$ or $G \Rightarrow F$. We do \emph{not} permit zig-zags.
\begin{thm}\label{thm:compare}
	Consider the square 
\begin{center}
	\begin{tikzcd}
	\quad\C_G\quad \arrow[dd, xshift=2.5mm, "\ext" description] \arrow[dd, xshift=-2.5mm, "\psext" description] \arrow[rr, yshift=1.7mm, "F_G"' description] & & \quad\D_G\quad \arrow[dd, xshift=2.75mm, "\exta" description] \arrow[dd, xshift=-2.75mm, "\psexta" description] \arrow[ll, yshift=-1.7mm, "U_G"' description] \\
	& & \\
	\quad\C_H\quad \arrow[uu, xshift=7.5mm, "\res"' description] \arrow[uu, xshift=-7.5mm, "\twext" description] \arrow[rr, yshift=1.7mm, "F_H"' description] & & \quad\D_H\quad \arrow[uu, xshift=8mm, "\resa"' description] \arrow[uu, xshift=-8mm, "\twexta" description] \arrow[ll, yshift=-1.7mm, "U_H"' description]
	\end{tikzcd}
\end{center} in which all the adjoint pairs are Quillen, $F_G \dashv U_G$ and $F_H \dashv U_H$ are Quillen equivalences, and suppose that there are natural isomorphisms $\L\ext \cong \R\psext$ and $\L\exta \cong \R\psexta$. Consider the statements:
	\begin{enumerate}[a)]
		\item There is a natural transformation between $\exta F_G$ and $F_H \ext$ which restricts to a natural weak equivalence on cofibrant objects.
		\item There is a natural transformation between $\res U_H$ and $U_G \resa$ which restricts to a natural weak equivalence on fibrant objects.
		\item There is a natural transformation between $F_G\twext$ and $\twexta F_H$ which restricts to a natural weak equivalence on cofibrant objects.
		\item There is a natural transformation between $\psext U_G$ and $U_H \psexta$ which restricts to a natural weak equivalence on fibrant objects.
	\end{enumerate}	
	If either (a) or (b) is true and either (c) or (d) is true, then each of the six derived squares listed in Section~\ref{sec:summary of method} commutes.
\end{thm}
\begin{proof}
	Statement $(a)$ is equivalent to statement $(b)$ by taking adjoints (see Lemma~\ref{lem:adjoints}), and similarly $(c)$ is equivalent to $(d)$. Since all of the functors in each statement have the same handedness, the natural weak equivalences descend to isomorphisms of derived functors.
	
	It just remains to show that these statements are sufficient to conclude that we have natural isomorphisms $\L\twext \circ \R U_H \cong\R U_G \circ \L\twexta$ and $\L F_G \circ \R \res \cong \R\resa \circ \L F_H$. By statement $(c)$ we have a natural isomorphism $\beta\colon \L F_G \circ \L \twext \cong \L \twexta \circ \L F_H$. Since $F_G \dashv U_G$ and $F_H \dashv U_H$ are Quillen equivalences and hence their derived adjunctions are adjoint equivalences, by Proposition~\ref{prop:mates} we find that the mate of $\beta$ is also a natural isomorphism $\overline\beta\colon \L\twext \circ \R U_H \cong \R U_G \circ \L\twexta$. Similarly, applying Proposition~\ref{prop:mates} to the natural isomorphism $\R U_G\circ\R\resa \cong \R\res\circ\R U_H$ from statement $(b)$ we obtain a natural isomorphism $\L F_G \circ \R \res \cong \R\resa \circ \L F_H$.
\end{proof}

\subsection{Checking commutativity}
In this section we give some lemmas which verify when the hypotheses of Theorem~\ref{thm:compare} hold. These fall into three types: strong monoidal Quillen pairs, weak monoidal Quillen pairs and Quillen equivalences which arise from Quillen invariance of modules. 


Firstly we deal with the case of strong monoidal Quillen pairs.
\begin{prop}\label{prop:comparestrong}
	Let 
	\[\begin{tikzcd}
	\C \arrow[r, "F", yshift=1mm] & \D \arrow[l, "U", yshift=-1mm]
	\end{tikzcd}\] be a strong monoidal adjunction and let $\theta\colon S \to R$ be a map of monoids in $\C$. Write $F\theta = \phi\colon FS \to FR$. There are natural isomorphisms $\res U \cong U\res[\phi]$ and $F\res \cong \res[\phi]F$, i.e., the diagrams 
	\[\begin{tikzcd}
	\mod{S}(\C) & \mod{FS}(\D) \arrow[l, "U"'] \\
	\mod{R}(\C) \arrow[u, "\res"] & \mod{FR}(\D) \arrow[l, "U"] \arrow[u, "{\res[\phi]}"']
	\end{tikzcd} \qquad 
	\begin{tikzcd}
	\mod{S}(\C) \arrow[r, "F"] & \mod{FS}(\D) \\
	\mod{R}(\C) \arrow[r, "F"'] \arrow[u, "\res"] & \mod{FR}(\D) \arrow[u, "{\res[\phi]}"']
	\end{tikzcd} \] commute up to natural isomorphism.
\end{prop}
\begin{proof}
	Take $N \in \mod{FR}$. The underlying objects of $\res UN$ and $U \resa N$ are the same so it remains to check that the module structures agree. We write $\Phi$ for the lax monoidal structure map on $U$, and $a\colon FR \wedge N \to N$ for the action map.
	\begin{center}
		\begin{tikzcd}
		S \wedge U\resa N \arrow[r, "\eta \wedge 1"] \arrow[dd, "1"'] \arrow[ddr, "\theta \wedge 1" description] & UFS \wedge U\resa N \arrow[r, "\Phi"] \arrow[ddr, "UF\theta \wedge 1" description] & U(FS \wedge \resa N) \arrow[r, "U(\phi \wedge 1)"] \arrow[ddr, "U(F\theta \wedge 1)" description] & U(FR \wedge N) \arrow[r, "Ua"] \arrow[dd, "1"] & UN \arrow[dd, "1"] \\
		& & & & \\
		S \wedge \res UN \arrow[r, "\theta \wedge 1"'] & R \wedge UN \arrow[r, "\eta \wedge 1"'] & UFR \wedge UN \arrow[r, "\Phi"'] & U(FR \wedge N) \arrow[r, "Ua"'] & UN 	\end{tikzcd}
	\end{center}
	The first triangle commutes by definition, the following square by naturality of $\eta$, the following square by naturality of $\Phi$, and the remaining triangle and square by definition. 
	
	The second natural isomorphism can be proved similarly.
\end{proof}

We now turn to the case of Quillen equivalences which arise from Quillen invariance of modules.
\begin{prop}\label{prop:ringsquare}
	Let $\C$ be a monoidal model category. Suppose that we have a commutative square of commutative monoids
	\begin{center}
		\begin{tikzcd}
		S \arrow[d, "\theta"'] & S' \arrow[l, "f"'] \arrow[d, "\psi"] \\ 
		R & R' \arrow[l, "g"] 
		\end{tikzcd}
	\end{center}
	in $\C$.
	\begin{enumerate}
		\item There is a natural isomorphism $\res[f]\res \cong \res[\psi]\res[g]$, i.e., the diagram \[\begin{tikzcd}
	\mod{S} \arrow[r, "{\res[f]}"] & \mod{S'} \\
	\mod{R} \arrow[u, "\res"] \arrow[r, "{\res[g]}"'] & \mod{R'} \arrow[u, "{\res[\psi]}"']
	\end{tikzcd}\] commutes up to natural isomorphism.
		\item Suppose that cofibrants are flat in $\mod{R'}$ (see Definition~\ref{defn:flatcofibrants}) and that there is a natural weak equivalence $S \otimes_{S'}R' \xrightarrow{\sim} R$ of commutative $R'$-algebras. Then there is a natural map $\ext[f]\res[\psi]M \to \res\ext[g]M$ which is a weak equivalence for cofibrant $R'$-modules $M$, i.e., the diagram
\[\begin{tikzcd}
	\mod{S} & \mod{S'} \arrow[l, "{\ext[f]}"'] \\
	\mod{R} \arrow[u, "\res"] & \mod{R'} \arrow[u, "{\res[\psi]}"'] \arrow[l, "{\ext[g]}"]
	\end{tikzcd}\]		
		commutes up to natural weak equivalence on cofibrant objects.
	\end{enumerate}
\end{prop}
\begin{proof}
	The identity map $\res[f]\res \Rightarrow \res[\psi]\res[g]$ is an $S'$-module map since $\theta f = g\psi$, which gives the claim in (1). The natural weak equivalence in (2) is obtained by applying $- \otimes_{R'} M$ to the natural weak equivalence $S \otimes_{S'}R' \xrightarrow{\sim} R$. We note that $- \otimes_{R'}M$ preserves weak equivalences since $M$ is cofibrant and cofibrants are flat in $\mod{R'}$ by hypothesis.	
\end{proof}

\begin{rem}\label{rem:pushout}
	We now give a few situations in which the hypothesis that there is a natural weak equivalence $S \otimes_{S'}R' \xrightarrow{\sim} R$ is satisfied. Note that in the category of commutative monoids, the pushout of a span $S \leftarrow S' \to R'$ is given by the tensor product $S \otimes_{S'} R'$. Therefore, the condition that there is a natural weak equivalence $S \otimes_{S'}R' \xrightarrow{\sim} R$ holds in the following cases.
	\begin{enumerate}
	\item The square is a pushout of commutative $S'$-algebras.
	\item The map $f$ is an acyclic cofibration of commutative $S'$-algebras and $g\colon R' \to R$ is a weak equivalence. In this case since pushouts preserve acyclic cofibrations, the map $S \otimes_{S'}R' \to R$ is a weak equivalence by the 2-out-of-3 property of weak equivalences.
	\item The maps $f$ and $g$ are weak equivalences, $\psi$ is a cofibration of commutative $S'$-algebras and the model category of commutative $S'$-algebras is left proper.
	\end{enumerate}
\end{rem}

Finally we treat the case of weak monoidal Quillen equivalences. In this case the argument is more involved since the left adjoint at the level of modules is different to the underlying left adjoint.
\begin{prop}\label{prop:compareweak}
	Let \begin{center}\begin{tikzcd}\C \arrow[r, yshift=-1mm, "U"'] &  \arrow[l, yshift=1mm, "F"'] \D\end{tikzcd}\end{center} be a weak monoidal Quillen pair and let $\theta\colon S \to R$ be a map of commutative monoids in $\C$. Write $\phi = U\theta\colon US \to UR.$ Suppose that there exists another model structure $\widetilde{\D}$ on the same underlying category as $\D$ so that $(\D, \widetilde{\D})$ is convenient (see Definition~\ref{defn:convenientmodelstructure}). Let $q\colon QUR \to UR$ be a cofibrant replacement of $UR$ in $\calg{US}(\widetilde{\D})$ and write $\psi\colon US \to QUR$ for the unit map of the $US$-algebra structure on $QUR$.
	\begin{enumerate}
		\item There is a natural isomorphism $\res[\psi]\res[q]U \cong U\res$, i.e., the diagram
		\[\begin{tikzcd}
		\mod{S}(\C) \arrow[rr, "U"] & & \mod{US}(\D) \\
		\mod{R}(\C) \arrow[r, "U"'] \arrow[u, "\res"] & \mod{UR}(\D) \arrow[r, "{\res[q]}"'] & \mod{QUR}(\D) \arrow[u, "{\res[\psi]}"']
		\end{tikzcd} \] commutes up to natural isomorphism.
		\item Suppose that $R$ is cofibrant as an $S$-module and that \[
			\begin{tikzcd}
			\mod{S}(\C) \arrow[r, yshift=-1mm, "U"'] & \arrow[l, yshift=1mm, "F^S"'] \mod{US}(\D)
			\end{tikzcd} \quad \mathrm{and} \quad \begin{tikzcd}
			\mod{R}(\C) \arrow[r, yshift=-1mm, "U"'] & \arrow[l, yshift=1mm, "F^R"'] \mod{UR}(\D)
			\end{tikzcd}
		\]
are Quillen equivalences. There is a natural map $F^S\res[\psi]N \to \res F^R\ext[q]N$ which is a weak equivalence for all cofibrant $QUR$-modules $N$, i.e., the diagram
			\[\begin{tikzcd}
		\mod{S}(\C) & & \mod{US}(\D) \arrow[ll, "F^S"'] \\
		\mod{R}(\C) \arrow[u, "\res"] & \mod{UR}(\D) \arrow[l, "F^R"] & \mod{QUR}(\D) \arrow[u, "{\res[\psi]}"'] \arrow[l, "{\ext[q]}"]
		\end{tikzcd} \] commutes up to natural weak equivalence on cofibrant objects. 
	\end{enumerate}
\end{prop}

\begin{proof}
Part (1) follows in a similar way as Proposition~\ref{prop:comparestrong}.
	
For (2), we first construct the natural map. Given any $QUR$-module $N$ we may construct a map $F^S\res[\psi]N \to \res F^R\ext[q]N$ via the composite
\begin{align*}
F^S\res[\psi]N 
&\xrightarrow{\mathmakebox[4em]{F^S\res[\psi]\eta_N^q}} F^S\res[\psi]\res[q]\ext[q]N \\ 
&\xrightarrow{\mathmakebox[4em]{\cong}} F^S\res[\phi]\ext[q]N \\
&\xrightarrow{\mathmakebox[4em]{F^S\res[\phi]\eta_{\ext[q]N}^R}} F^S\res[\phi]UF^R\ext[q]N \\
&\xrightarrow{\mathmakebox[4em]{\cong}} F^SU\res F^R\ext[q]N \\
&\xrightarrow{\mathmakebox[4em]{\epsilon^S_{\res F^R\ext[q]N}}} \res F^R\ext[q]N\end{align*}
where $\eta^q$ and $\eta^R$ are the units of the $\ext[q] \dashv \res[q]$ and $F^R \dashv U$ adjunctions respectively, and $\epsilon^S$ is the counit of the $F^S \dashv U$ adjunction. The two isomorphisms follow from similar arguments as in Propositions~\ref{prop:comparestrong} and~\ref{prop:ringsquare}. 

We now show that if $N$ is a cofibrant $QUR$-module then this natural map is a weak equivalence. Firstly consider the commutative diagram
\[\begin{tikzcd}
\res[\psi]N \arrow[r, "{\res[\psi]\eta^q_N}"] \arrow[r, "\sim"'] & \res[\psi]\res[q]\ext[q]N \arrow[r, "\cong"] & \res[\phi]\ext[q]N \arrow[rr, "{\res[\phi]\eta^R_{\ext[q]N}}"] \arrow[drr, "\sim"] \arrow[drr, "{\res[\phi]\overline{\eta}^R_{\ext[q]N}}"'] & & \res[\phi]UF^R\ext[q]N \arrow[r, "\cong"] \arrow[d, "{\res[\phi] U r}"] & U\res F^R\ext[q]N \arrow[d, "{U\res r}"] \\
& & & & \res[\phi]UfF^R\ext[q]N \arrow[r, "\cong"'] & U\res fF^R\ext[q]N 
\end{tikzcd} \]
where $r\colon F^R\ext[q]N \xrightarrow{\sim} fF^R\ext[q]N$ denotes a fibrant replacement and $\overline{\eta}$ denotes the derived unit. The map $\eta^q_N\colon N \to \res[q]\ext[q]N$ is a weak equivalence since $N$ is cofibrant and $\res[q]$ preserves all weak equivalences. Since $\res[\psi]$ also preserves all weak equivalences the leftmost map in the diagram is a weak equivalence. The diagonal map is a weak equivalence since the $F^R \dashv U$ adjunction was assumed to be a Quillen equivalence and $\res[\phi]$ preserves all weak equivalences.

We now may factor the composites $\res[\psi]N \to U\res F^R\ext[q]N$ and $\res[\psi]N \to U\res f F^R\ext[q]N$ into acyclic cofibrations followed by fibrations, apply $F^S$, and then postcompose with the counit $\epsilon^S$ of the $F^S \dashv U$ adjunction to obtain the commutative diagram
\[\begin{tikzcd} 
F^S\res[\psi]N \arrow[r, "\sim"] \arrow[dr, "\sim"'] & F^SQU\res F^R\ext[q]N \arrow[r] & F^SU\res F^R \ext[q]N \arrow[rr, "{\epsilon^S_{\res F^R \ext[q]N}}"]  \arrow[d, "{F^SU\res r}"] & & \res F^R\ext[q]N \arrow[d, "{\res r}"] \arrow[d, "\sim"'] \\
& F^SQU\res fF^R\res[q]N \arrow[r] & F^SU\res f F^R\ext[q]N \arrow[rr, "{\epsilon^S_{\res fF^R \ext[q]N}}"] & & \res f F^R \ext[q]N. 
\end{tikzcd}\]
in which the top row is the map we want to show is a weak equivalence. The right hand vertical map is a weak equivalence since $\res$ preserves all weak equivalences. The composite of the bottom row is the derived counit of the $F^S \dashv U$ adjunction on the fibrant object $\res f F^R\ext[q] N$ and as such is a weak equivalence. It follows from the 2-out-of-3 property of weak equivalences that the top row is a weak equivalence as required.
\end{proof}

\subsection{Quillen pairs post localization}
In this section we will give conditions under which the Quillen pairs of interest descend to Quillen pairs between Bousfield localizations.

We first must recall the projection formula.
\begin{defn}\label{defn:projectionformula} Let $F: \C \rightleftarrows \D :U$ be an adjunction between monoidal categories.
	\begin{enumerate}
	\item Suppose that $U$ is strong monoidal. We say that the \emph{left projection formula for $F \dashv U$ holds} if the natural map $p\colon F(U(X) \wedge Y) \to X \wedge F(Y)$ defined by \begin{center}\begin{tikzcd}F (U(X) \wedge Y) \arrow[r, "F (1 \wedge\eta)"] & F (U(X) \wedge UF (Y)) \arrow[r, "\cong"'] \arrow[r, "F\Phi"]& FU(X \wedge F (Y)) \arrow[r, "\varepsilon"] & X \wedge F (Y)\end{tikzcd}\end{center} is an isomorphism for all $X \in \D$ and $Y \in \C$, where $\Phi$ denotes the monoidal structure map of $U$.  
	\item Suppose that $F$ is strong monoidal. We say that the \emph{right projection formula for $F \dashv U$ holds} if the natural map  $p'\colon X \wedge U(Y) \to U(F (X) \wedge Y)$ defined by \begin{center}\begin{tikzcd}X \wedge U(Y) \arrow[r, "\eta"] & UF (X \wedge U(Y)) \arrow[r, "\cong"'] \arrow[r, "U\Phi^{-1}"]& U (F(X) \wedge FU(Y)) \arrow[r, "U(1 \wedge \varepsilon)"] & U(F(X) \wedge Y)\end{tikzcd}\end{center} is an isomorphism for all $X \in \C$ and $Y \in \D$, where $\Phi$ denotes the monoidal structure map of $F$.
	\end{enumerate}
\end{defn}


\begin{eg}
The right projection formula clearly holds for the extension-restriction of scalars adjunction along a map of commutative monoids in a symmetric monoidal category.
\end{eg}

\begin{eg}
Let $H$ be a closed subgroup of a compact Lie group $G$, and consider the induction-restriction-coinduction adjoint triple $\ind \dashv \forget \dashv \coind$ between $G$-spectra and $H$-spectra. The adjunction $\ind \dashv \forget$ satisfies the left projection formula both at the point-set level and in the homotopy category. On the other hand, the adjunction $\forget \dashv \coind$ does not satisfy the right projection formula at the point-set level, although it does hold in the homotopy category. See~\cite[V.2.3]{MandellMay02} for details on the point-set level, and~\cite[II.4.9]{LMS} for the derived version.
\end{eg}

We firstly deal with the case of (homological) left Bousfield localization. Let $\C$ be a stable, monoidal model category and $E$ be an object of $\C$. We say that a map $f\colon X \to Y$ in $\C$ is an $E$-equivalence if $E \otimes^\L f$ is an isomorphism in the homotopy category. The homological localization of $\C$ at $E$ (when it exists), denoted by $L_E\C$, is the model structure on $\C$ in which the weak equivalences are given by the $E$-equivalences, and the cofibrations are the same as the cofibrations in $\C$. One may also view this as a Bousfield localization in the usual sense by inverting the collection $S$ of $E$-equivalences. If the localizations exist, the universal property of $L_S\C$~\cite[3.3.18]{Hirschhorn03} shows that $L_E\C$ and $L_S\C$ are equal as model categories. We direct the reader to~\cite[6.1]{BalchinGreenlees20} for a general existence theorem for homological localizations. We write $\langle E \rangle$ for the Bousfield class of $E$, that is, for the class of objects $X$ for which $E \otimes^\L X \simeq 0$. In what follows, we implicitly use the fact that a functor which is both left and right Quillen preserves all weak equivalences and therefore is equivalent to its derived functor.
\begin{thm}\label{thm:quillenafterloc}
	Let 
	\begin{center}
		\begin{tikzcd}
			\C \arrow[rr, yshift=3.2mm, "\ext" description] \arrow[rr, yshift=-3.2mm, "\coext" description] & & \D \arrow[ll, "\res" description]
		\end{tikzcd}
	\end{center} 
	be a Quillen adjoint triple between stable monoidal model categories.
\begin{enumerate}
\item Suppose that $\ext \dashv \res$ is a weak monoidal Quillen pair and that $\L\ext \dashv \R\res$ satisfies the right projection formula in the homotopy category. Let $E \in \C$ and $E' \in \D$ and suppose that $L_E\C$ and $L_{E'}\D$ exist. Then
	\begin{center}
		\begin{tikzcd}
		L_E\C \arrow[rr, yshift=3.2mm, "\ext" description] \arrow[rr, yshift=-3.2mm, "\coext" description] & & L_{E'}\D \arrow[ll, "\res" description]
		\end{tikzcd}
	\end{center}
	is a Quillen adjoint triple if $\langle E' \rangle = \langle\L\ext (E) \rangle$.
\item Suppose that $\res \dashv \coext$ is a weak monoidal Quillen pair and that $\L\res \dashv \R\coext$ satisfies the left projection formula in the homotopy category. Let $E \in \C$ and $E' \in \D$ and suppose that $L_E\C$ and $L_{E'}\D$ exist. Then
	\begin{center}
		\begin{tikzcd}
		L_E\C \arrow[rr, yshift=3.2mm, "\ext" description] \arrow[rr, yshift=-3.2mm, "\coext" description] & & L_{E'}\D \arrow[ll, "\res" description]
		\end{tikzcd}
	\end{center}
	is a Quillen adjoint triple if $\langle E \rangle = \langle \res E' \rangle$.
\end{enumerate}
\end{thm}
\begin{proof}
	For (1), as $\langle E' \rangle = \langle\L\ext (E) \rangle$, the left Bousfield localizations $L_{E'}\D$ and $L_{\L\ext (E)}\D$ are \emph{equal} as model categories. Therefore it suffices to prove the result for the case when $E' = \L\ext E$. 
	By Hirschhorn~\cite[3.3.18]{Hirschhorn03}, to verify that $\ext \dashv \res$ is Quillen between the localizations it is enough to check that $\ext$ sends $E$-equivalences between cofibrant objects to $\L\ext (E)$-equivalenes. This follows from the fact that $\ext \dashv \res$ is a weak monoidal Quillen pair; alternatively see~\cite[3.1]{PolWilliamson}. For the $\res \dashv \coext$ adjunction, by Hirschhorn~\cite[3.3.18]{Hirschhorn03}, it suffices to check that $\res$ sends $\L\ext (E)$-equivalences between cofibrant objects to $E$-equivalences. As $\res (\L\ext (E) \otimes^\L X) \cong E \otimes^\L \res X$ for all $X \in \D$ by the projection formula, this follows.
	
	For (2), the argument is completely analogous. The $\res \dashv \coext$ adjunction follows immediately from the fact that $\res \dashv \coext$ is a weak monoidal Quillen pair, and the $\ext \dashv \res$ adjunction follows from the projection formula.
\end{proof}

The following lemma shows that the hypothesis on Bousfield classes in the previous statement transfers along monoidal Quillen equivalences.
\begin{lem}\label{lem:bousfieldclasses}
Let $F\colon \C \rightleftarrows \D :U$ be a weak monoidal Quillen equivalence between stable, monoidal model categories. 
\begin{enumerate}
\item Let $E, E' \in \C$ be such that $\langle E \rangle = \langle E' \rangle$. Then $\langle \L F(E) \rangle = \langle \L F(E') \rangle$.
\item Let $E, E' \in \D$ be such that $\langle E \rangle = \langle E' \rangle$. Then $\langle \R U(E) \rangle = \langle \R U(E') \rangle$.
\end{enumerate}
\end{lem}
\begin{proof}
Since $F \dashv U$ is a weak monoidal Quillen equivalence, it descends to an equivalence of homotopy categories in which the left adjoint $F$ is strong monoidal. For ease of reading, we omit notation for derived functors in this proof. We first prove (1), so suppose that $\langle E \rangle = \langle E' \rangle$. Then for any object $X$ of $\D$ we have $FE \otimes X \simeq FE \otimes FUX \simeq F(E \otimes UX)$. Therefore $FE \otimes X \simeq 0$ if and only if $E \otimes UX \simeq 0$ since $F$ is an equivalence. As $\langle E \rangle = \langle E' \rangle$, this is the case if and only if $E' \otimes UX \simeq 0$ if and only if $FE' \otimes X \simeq F(E' \otimes UX) \simeq 0$. Hence $\langle FE \rangle = \langle FE' \rangle$ as required. The proof for (2) is analogous.
\end{proof}

Now we turn to the case of cellularizations. Before we state the result let us recall some terminology. Let $\C$ be a stable model category and $K$ be an object of $\C$. A map $f$ of $\C$ is said to be a $K$-cellular equivalence if $[K,f]_*$ is an isomorphism, where $[-,-]_*$ denotes the graded abelian group of maps in the homotopy category. The cellularization of $\C$ at $X$ (when it exists), denoted $\cell_K\C$, is the model structure on $\C$ in which the weak equivalences are the $K$-cellular equivalences and the fibrations are the same as the fibrations in $\C$. The $K$-cellularization of $\C$ exists if $\C$ is right proper and cellular~\cite[5.1.1]{Hirschhorn03}. We recall from~\cite[2.5, 2.6]{GreenleesShipley14} that if $K$ is a small object of $\C$, then the homotopy category of $\cell_K\C$ is compactly generated by $K$. Finally, for $X,Y \in \C$ we say that $X$ \emph{builds} $Y$ if $Y$ is in the localizing subcategory generated by $X$, that is, if $Y$ is in the smallest replete, triangulated full subcategory of $h\C$ which is closed under arbitrary coproducts and contains $X$.
\begin{thm}\label{thm:quillenaftercell}
	Let 
	\begin{center}
		\begin{tikzcd}
		\C \arrow[rr, yshift=3.2mm, "\ext" description] \arrow[rr, yshift=-3.2mm, "\coext" description] & & \D \arrow[ll, "\res" description]
		\end{tikzcd}
	\end{center} be a Quillen adjoint triple between stable model categories.
	\begin{enumerate} 
	\item Let $K \in \D$ and suppose that $\cell_{\res K}\C$ and $\cell_{K}\D$ exist. If $K$ builds $\L\ext\res K$, then 
	\begin{center}
		\begin{tikzcd}
		\cell_{\res K}\C \arrow[rr, yshift=3mm, "\ext" description] \arrow[rr, yshift=-3mm, "\coext" description] & & \cell_{K}\D \arrow[ll, "\res" description]
		\end{tikzcd}
	\end{center}
	is a Quillen adjoint triple.
	\item Let $K \in \C$ and suppose that $\cell_K\C$ and $\cell_{\L\ext K}\D$ exist. If $K$ builds $\res\L\ext K$, then 
	\begin{center}
		\begin{tikzcd}
		\cell_{K}\C \arrow[rr, yshift=3mm, "\ext" description] \arrow[rr, yshift=-3mm, "\coext" description] & & \cell_{\L\ext K}\D \arrow[ll, "\res" description]
		\end{tikzcd}
	\end{center}
	is a Quillen adjoint triple.
	\end{enumerate}
\end{thm}
\begin{proof}
	We prove (1); the proof of (2) is completely analogous. By~\cite[2.7]{GreenleesShipley13}, the adjunction $\res \dashv \coext$ is Quillen between the cellularizations as the objects correspond. To check that the adjunction $\ext \dashv \res$ is Quillen between the cellularizations, we must check that $\res$ sends $K$-cellular equivalences between fibrant objects to $\res K$-cellular equivalences by Hirschhorn~\cite[3.3.18]{Hirschhorn03}.
	
	Suppose that $M \to N$ is a $K$-cellular equivalence between fibrant objects. By adjunction,  $\res M \to \res N$ is a $\res K$-cellular equivalence if and only if $M \to N$ is a $\L\ext\res K$-cellular equivalence. As $K$ builds $\L\ext\res K$, $M \to N$ is an $\L\ext\res K$-cellular equivalence as required.
\end{proof}

\begin{eg}\label{eg:smallbuild}
	If $\theta\colon S \to R$ is a map of commutative ring spectra, then for any $R$-module $K$, $K$ builds $\L\ext\res K$. To see this, note that $S$ builds $R$ and therefore applying $K \otimes_S^\L -$ gives the claim.
\end{eg}
	
\part{The correspondence of functors}
\section{The general strategy}

\subsection{Recap of the algebraic model}
In this section we recap the construction of the algebraic model for (co)free $G$-spectra for $G$ connected. Recall our convention that in this part of the paper, everything is rationalized without comment. 

The algebraic model for free $G$-spectra is due to Greenlees-Shipley~\cite{GreenleesShipley11, GreenleesShipley14}, and the cofree case is work of Pol and the author~\cite{PolWilliamson}. We note that in the free case, we follow the Eilenberg-Moore approach taken in~\cite{GreenleesShipley14}, rather than the approach via Koszul duality taken in~\cite{GreenleesShipley11}. It is important to note that the correspondence of functors we obtain by using the zig-zag of Quillen equivalences in~\cite{GreenleesShipley14} is \emph{different} from the correspondence obtained in~\cite{GreenleesShipley11} by using the Koszul duality approach. In the cofree case we follow the direct approach given in~\cite{PolWilliamson}, rather than the approach given by first passing to free $G$-spectra and then using the equivalence between derived torsion and derived complete modules.

Free $G$-spectra are modelled by the cellularization $\mathrm{Cell}_{G_+}\Sp_G$, and cofree $G$-spectra are modelled by the homological localization $L_{EG_+}\Sp_G$. We now recall the Quillen equivalences used in the construction of the algebraic model. Note that not all of the Quillen adjunctions below are Quillen equivalences, but they all become so after appropriate localization/cellularization.
\begin{enumerate}
	\item \emph{Change of rings:}
	Beginning in $G$-spectra, the first step is to change rings along the map $\kappa\colon S^0 \to DEG_+$ where $DEG_+ = F(EG_+,S^0)$. Therefore the first stage is the extension and restriction of scalars adjunction $$\begin{tikzcd}
	\Sp_G \arrow[rrr, "DEG_+ \wedge -" description, yshift=2mm] & & & \mod{DEG_+}(\Sp_G). \arrow[lll, "{\res[\kappa]}" description, yshift=-2mm]
	\end{tikzcd}$$
	
	\item\emph{Fixed points-inflation adjunction:} The next step is to use categorical fixed points to remove equivariance. More precisely, the next stage is the adjunction
	$$\begin{tikzcd}
	\mod{DEG_+}(\Sp_G) \arrow[rrrr, "(-)^G" description, yshift=-2.2mm] & & & & \mod{DBG_+}(\Sp) \arrow[llll, "DEG_+ \otimes_{DBG_+} -" description, yshift=2.2mm]
	\end{tikzcd}$$
	where we have suppressed notation for the inflation functors in the left adjoint.
	
	\item\emph{Shipley's algebraicization theorem:} The next stage is to use Shipley's algebraicization theorem~\cite{Shipley07} to pass from modules over the commutative ring spectrum $DBG_+$ to modules over a commutative DGA which we denote $\Theta DBG_+$. This is a zig-zag of Quillen equivalences. See Section~\ref{sec:shipley} for more details on the zig-zag of Quillen equivalences.
	
	\item\emph{Formality:} One can next use the fact that polynomial algebras are strongly intrinsically formal as \emph{commutative} DGAs to construct a quasiisomorphism $z\colon H^*BG \to \Theta DBG_+$. This gives a Quillen equivalence 
	$$\begin{tikzcd}
	\mod{\Theta DBG_+} \arrow[rrrr, "{\res[z]}" description, yshift=-2mm] & & & & \mod{H^*BG} \arrow[llll, "{\Theta DBG_+ \otimes_{H^*BG} -}" description, yshift=2mm]
	\end{tikzcd}$$
	via extension and restriction of scalars.
	
	\item\emph{Torsion and completion:} The final step is to identify the resulting localization or cellularization with an abelian model. In the localization case, this is the category of $L_0^I$-complete modules, and in the cellularization case this is the category of torsion modules. 
\end{enumerate}

\subsection{Proving the correspondence of functors}\label{sec:strategy}
In this section we explain the general process for proving that the functors correspond. 

The general setup is that we will have a square
\begin{center}
	\begin{tikzcd}
	\quad\C_G\quad \arrow[dd, xshift=2.5mm, "\ext" description] \arrow[dd, xshift=-2.5mm, "\psext" description] \arrow[rr, yshift=1.7mm, "F_G"' description] & & \quad\D_G\quad \arrow[dd, xshift=2.75mm, "\exta" description] \arrow[dd, xshift=-2.75mm, "\psexta" description] \arrow[ll, yshift=-1.7mm, "U_G"' description] \\
	& & \\
	\quad\C_H\quad \arrow[uu, xshift=7.5mm, "\res"' description] \arrow[uu, xshift=-7.5mm, "\twext" description] \arrow[rr, yshift=1.7mm, "F_H"' description] & & \quad\D_H\quad \arrow[uu, xshift=8mm, "\resa"' description] \arrow[uu, xshift=-8mm, "\twexta" description] \arrow[ll, yshift=-1.7mm, "U_H"' description]
	\end{tikzcd}
\end{center}
of model categories and Quillen functors, where $\R\psext \cong \L\ext$ and $\R\psext[\phi] \cong \L\ext[\phi]$, and $F_G \dashv U_G$ and $F_H \dashv U_H$ are Quillen equivalences. Depending on the type of square we can then apply Theorem~\ref{thm:compare} in conjunction with Propositions~\ref{prop:comparestrong},~\ref{prop:ringsquare} and~\ref{prop:compareweak} to conclude that the vertical Quillen functors correspond, and hence that all six squares of derived functors commute.

In general the vertical functors will arise from the extension-restriction-coextension of scalars functors along a map of commutative ring spectra $\theta\colon S \to R$. In general, $R$ need not be cofibrant as an $S$-module, so that $\res \dashv \coext$ will not be a Quillen adjunction by Proposition~\ref{prop:extrescoext}. To rectify this, we must cofibrantly replace $R$ as a commutative $S$-algebra to obtain $S \to QR$. In a convenient model structure, this implies that $QR$ is also cofibrant as an $S$-module, see Section~\ref{sec:flat}. In all of our examples of interest, Quillen invariance of modules holds and therefore shows that extension and restriction of scalars gives a Quillen equivalence $\mod{R} \simeq_Q \mod{QR}$. 

In summary, each step will consist of:
\begin{enumerate}
	\item Construct a square of Quillen functors. This may involve taking cofibrant replacements of commutative algebras in a convenient model structure. As such, unless stated otherwise all cofibrant replacements of commutative algebras in what follows are performed in a convenient model structure.
	\item Check that under the relevant localizations/cellularizations which make the horizontals Quillen equivalences, the vertical functors are still Quillen. In general, this can be achieved by using Theorem~\ref{thm:quillenafterloc}, Lemma~\ref{lem:bousfieldclasses} and Theorem~\ref{thm:quillenaftercell}.
	\item Use Propositions~\ref{prop:comparestrong},~\ref{prop:ringsquare} and~\ref{prop:compareweak} to verify that the hypotheses of Theorem~\ref{thm:compare} are satisfied and hence deduce that the Quillen functors correspond.
\end{enumerate}

\section{Change of rings}\label{sec:changeofrings}
Let $G$ be a connected compact Lie group and $H$ be a connected closed subgroup of $G$.
\subsection{The setup}
The first step in the series of Quillen equivalences is a change of rings along the maps $S^0 \to DEG_+$ and $S^0 \to DEH_+$. The construction of adjoints between $\mod{DEG_+}(\Sp_G)$ and $\mod{DEH_+}(\Sp_H)$ requires some explanation, because we not only have to change the underlying category between $\Sp_G$ and $\Sp_H$, but also the rings that we take modules over. 

\begin{lem}\label{lem:modulelemma}
	Let $\C$ and $\D$ be monoidal categories and $R$ be a monoid in $\C$. Suppose that we have adjunctions 
	\begin{center}
		\begin{tikzcd}
		\C \arrow[rr, yshift=-3.5mm, "\coind" description] \arrow[rr, yshift=3.5mm, "\ind" description] & & \arrow[ll, "\forget" description] \D
		\end{tikzcd}
	\end{center}
	where $\forget$ is strong monoidal. If the left projection formula for $\ind \dashv \forget$ holds (see Definition~\ref{defn:projectionformula}) then we have adjunctions
	\begin{center}
		\begin{tikzcd}
		\mod{\forget R}(\C) \arrow[rr, yshift=-3.5mm, "\coind" description] \arrow[rr, yshift=3.5mm, "\ind" description] & & \arrow[ll, "\forget" description] \mod{R}(\D)
		\end{tikzcd}
	\end{center}
	between the categories of modules.
\end{lem}
\begin{proof}
	Since $\forget$ is strong monoidal it sends monoids in $\C$ to monoids in $\D$. Moreover, it follows that $\ind$ is oplax monoidal and $\coind$ is lax monoidal. Let $M$ be an $\forget R$-module.
	
	To give $\ind M$ an $R$-module structure we define the action map by $$R \wedge \ind M \xrightarrow[\cong]{p^{-1}} \ind (\forget R \wedge M) \xrightarrow{\ind (a)} \ind M$$ where $p$ is the projection formula map and $a\colon \forget R \wedge M \to M$ is the module structure map for $M$. Similarly, we define an $R$-module structure on $\coind M$ by $$R \wedge \coind M \xrightarrow{p'} \coind (\forget R \wedge M) \xrightarrow{\coind (a)} \coind M$$ where $p'$ is the natural map in the right projection formula for $\coind$, see Definition~\ref{defn:projectionformula}. Note that we do not require that the right projection formula for $\forget\dashv \coind$ holds; we only require the existence of the map $p'$, not that it is an isomorphism.
	
	It remains to check that the action maps defined above are associative and unital. So as to avoid interrupting the flow, we defer the remainder of the proof to Appendix~\ref{sec:appendix}.
\end{proof}

\begin{prop}\label{prop:indrescoextmods}
	There is a Quillen adjoint triple of functors
	\begin{center}
		\begin{tikzcd}
		\mod{DEH_+}(\Sp_H) \arrow[rr, yshift=-3.5mm, "\coind" description] \arrow[rr, yshift=3.5mm, "\ind" description] & & \arrow[ll, "\forget" description] \mod{DEG_+}(\Sp_G)
		\end{tikzcd}
	\end{center}
	where $\forget$ is the restriction from $G$-spectra to $H$-spectra, $\ind = G_+ \wedge_H -$ and $\coind = F_H(G_+,-).$
\end{prop}
\begin{proof}
	Recall that the left projection formula holds for $\ind \dashv \forget$, see~\cite[V.2.3]{MandellMay02}. By Lemma~\ref{lem:modulelemma}, it follows that we have adjunctions as described, so it only remains to check that they are Quillen. Since the weak equivalences and fibrations in the categories of modules are created by the forgetful functors to the underlying equivariant spectra, $\forget$ and $\coind$ are right Quillen, since they are right Quillen when viewed as functors between $G$-spectra and $H$-spectra.
\end{proof}

\subsection{Quillen functors}
In this section we show that all the functors of interest are Quillen after the relevant localizations and cellularizations.
\begin{prop}\label{prop:changeofringsquillen1}
	There are Quillen adjoint triples of functors
	\begin{center}
		\begin{tikzcd}
		L_{EH_+}\Sp_H \arrow[rr, yshift=-3.5mm, "\coind" description] \arrow[rr, yshift=3.5mm, "\ind" description] & & \arrow[ll, "\forget" description] L_{EG_+}\Sp_G
		\end{tikzcd}
	\end{center}
	and
	\begin{center}
		\begin{tikzcd}
		L_{EH_+}\mod{DEH_+}(\Sp_H) \arrow[rr, yshift=-3.5mm, "\coind" description] \arrow[rr, yshift=3.5mm, "\ind" description] & & \arrow[ll, "\forget" description] L_{EG_+}\mod{DEG_+}(\Sp_G).
		\end{tikzcd}
	\end{center}
\end{prop}
\begin{proof}
Since $\forget EG_+ \simeq EH_+$, this follows from Theorem~\ref{thm:quillenafterloc}(2) together with Proposition~\ref{prop:indrescoextmods}.
\end{proof}

\begin{prop}\label{prop:changeofringsquillen2}
	There are Quillen adjoint triples of functors
	\begin{center}
		\begin{tikzcd}
		\cell_{H_+}\Sp_H \arrow[rr, yshift=-3.5mm, "\coind" description] \arrow[rr, yshift=3.5mm, "\ind" description] & & \arrow[ll, "\forget" description] \cell_{G_+}\Sp_G
		\end{tikzcd}
	\end{center}
	and
	\begin{center}
		\begin{tikzcd}
		\cell_{H_+}\mod{DEH_+}(\Sp_H) \arrow[rr, yshift=-3.5mm, "\coind" description] \arrow[rr, yshift=3.5mm, "\ind" description] & & \arrow[ll, "\forget" description] \cell_{G_+}\mod{DEG_+}(\Sp_G).
		\end{tikzcd}
	\end{center}
\end{prop}
\begin{proof}
	Note that $\ind H_+ \simeq G_+$. Since $EH_+ \wedge \forget G_+ \simeq \forget G_+$, we see that $\forget G_+$ is a free $H$-spectrum and hence is built from $H_+$. The claim then follows from Theorem~\ref{thm:quillenaftercell}(2) together with Proposition~\ref{prop:indrescoextmods}.
\end{proof}

\subsection{The corresponding adjoints}
We consider the squares
\begin{equation}\label{eq:changeofrings}	
	\begin{tikzcd}
	L_{EG_+}\mod{S^0}(\Sp_G) \arrow[dd, "\forget" description] \arrow[rr, yshift=2mm, "\extkg" description] 
	& & L_{EG_+}\mod{DEG_+}(\Sp_G) \arrow[dd, "\forget" description] \arrow[ll, yshift=-2mm, "\reskg" description] \\
	& & \\
	L_{EH_+}\mod{S^0}(\Sp_H) \arrow[uu, xshift=-5mm, "\ind" description] \arrow[uu, xshift=5mm, "\coind" description] \arrow[rr, yshift=2mm, "\extkh" description] & & L_{EH_+}\mod{DEH_+}(\Sp_H) \arrow[ll, yshift=-2mm, "\reskh" description] \arrow[uu, xshift=-5mm, "\ind" description] \arrow[uu, xshift=5mm, "\coind" description] 
	\end{tikzcd}
	\end{equation}
and
\begin{equation}\label{eq:changeofrings2}	
 \begin{tikzcd}
	\cell_{G_+}\mod{S^0}(\Sp_G) \arrow[dd, "\forget" description] \arrow[rr, yshift=2mm, "\extkg" description] 
	& & \cell_{G_+}\mod{DEG_+}(\Sp_G) \arrow[dd, "\forget" description] \arrow[ll, yshift=-2mm, "\reskg" description] \\
	& & \\
	\cell_{H_+}\mod{S^0}(\Sp_H) \arrow[uu, xshift=-5mm, "\ind" description] \arrow[uu, xshift=5mm, "\coind" description] \arrow[rr, yshift=2mm, "\extkh" description] & & \cell_{H_+}\mod{DEH_+}(\Sp_H) \arrow[ll, yshift=-2mm, "\reskh" description] \arrow[uu, xshift=-5mm, "\ind" description] \arrow[uu, xshift=5mm, "\coind" description] 
	\end{tikzcd}
\end{equation}
where the horizontals are given by extension and restriction of scalars along the ring maps $\kappa_G\colon S^0 \cong F(S^0, S^0) \to F(EG_+, S^0) = DEG_+$ and $\kappa_H\colon S^0 \to DEH_+$ and the verticals are the change of groups adjunctions of equivariant spectra. 

\begin{prop}\label{prop:changeofrings}
The vertical Quillen functors in~(\ref{eq:changeofrings}) and~(\ref{eq:changeofrings2}) correspond.
\end{prop}
\begin{proof}
The vertical functors in~(\ref{eq:changeofrings}) and~(\ref{eq:changeofrings2}) are Quillen by Propositions~\ref{prop:changeofringsquillen1} and~\ref{prop:changeofringsquillen2}, and the horizontals are Quillen equivalences by~\cite[\S 8 Step 1]{PolWilliamson} and~\cite[\S 3]{GreenleesShipley14} respectively. This puts us in the setup of Theorem~\ref{thm:compare}. 

We now argue that the forgetful functor $\forget$ commutes with the restrictions $\res[\kappa]$. None of the functors change the underlying objects so it is enough to show that the module structures agree. Let $M$ be an object of $\mod{DEG_+}(\Sp_G)$ with module action map $a\colon DEG_+ \wedge M \to M$. We write $\Phi$ for the monoidal structure map of $\forget$. Consider the diagram
\begin{center}
	\begin{tikzcd}
	S^0 \wedge \reskh\forget M \arrow[d, "1"'] \arrow[r, "\kappa_H \wedge 1"] & DEH_+ \wedge \forget M \arrow[r, "1"] & \forget DEG_+ \wedge \forget M \arrow[r, "\Phi"] \arrow[r, "\cong"'] & \forget(DEG_+ \wedge M) \arrow[d, "1"]  \arrow[r, "\forget(a)"] & \forget M \arrow[d, "1"] \\
	S^0 \wedge \forget\reskg M \arrow[r, "1"'] \arrow[ur, "\kappa_H \wedge 1" description] & \forget S^0 \wedge \forget\reskg M \arrow[r, "\Phi"'] \arrow[r, "\cong"] & \forget (S^0 \wedge \reskg M) \arrow[ru, "\forget (\kappa_G \wedge 1)" description] \arrow[r, "\forget(\kappa_G \wedge 1)"'] & \forget (DEG_+ \wedge M) \arrow[r, "i^*(a)"'] & i^*M
	\end{tikzcd}
\end{center}
in which the top row is the module structure defined by $\reskh\forget M$ and the bottom row is the module structure defined by $\forget\reskg M$. This diagram commutes by naturality of $\Phi$ using that $\forget\kappa_G = \kappa_H$, and hence the module structures agree (i.e., the identity is a module map). Therefore we have a natural isomorphism $\reskh\forget \cong \forget\reskg.$ 

We now check that the identity gives a natural isomorphism $\coind\reskh \cong \reskg\coind$. Let $M$ be a $DEH_+$-module with action map $a\colon DEH_+ \wedge M \to M$. Then the underlying objects of $\coind\reskh M$ and $\reskg\coind M$ are the same so it remains to check the module structures agree. Consider the diagram
\begin{center}
	\begin{tikzcd}
	S^0 \wedge \coind\reskh M \arrow[d, "1"'] \arrow[r, "p^{-1}"] \arrow[dr, "\kappa_G \wedge 1" description] & \coind(S^0 \wedge \reskh M) \arrow[r, "\coind(\kappa_H \wedge 1)"] \arrow[rd, "\coind(\kappa_H \wedge 1)" description] & \coind (DEH_+ \wedge M) \arrow[r, "\coind(a)"] & \coind M \arrow[d, "1"] \\
	S^0 \wedge \reskg\coind M \arrow[r, "\kappa_G \wedge 1"'] & DEG_+ \wedge \coind M \arrow[r, "p^{-1}"'] & \coind (DEH_+ \wedge M) \arrow[r, "\coind (a)"'] & \coind M
	\end{tikzcd}
\end{center}
in which the top row is the module structure on $\coind\reskh M$ and the bottom row is the module structure on $\reskg\coind M$. The diagram commutes by the naturality of $p$, and hence we have a natural isomorphism $\coind\reskh \cong \reskg\coind$. It follows from the special case of Theorem~\ref{thm:compare} in which the two downward pointing arrows are the same, that the Quillen functors correspond. 
\end{proof}

\section{The fixed points-inflation adjunction}\label{sec:fixedpoints}
In this section we describe how the passage from equivariant module spectra to non-equivariant module spectra interacts with change of groups functors. Since the categorical fixed points functor is not right Quillen between the flat model structures, extra care is required to ensure that all functors in play are Quillen.

\subsection{Quillen pairs}
Stolz~\cite{Stolz11} constructs a flat model structure on orthogonal $G$-spectra which has the same weak equivalences as the stable model structure, and for which the identity functor is a left Quillen equivalence from the stable model structure to the flat model structure. The flat model structure has all the necessary properties to apply our general results from Sections~\ref{sec:functors} and~\ref{sec:compare}. In particular, the flat model structure on orthogonal $G$-spectra is a monoidal model structure which satisfies the monoid axiom (so that the categories of modules inherit a right lifted model structure), fits into a convenient pair of model structures, and has flat cofibrant objects, see Section~\ref{sec:background} for more details. Before we can show a correspondence of functors we need to prove that certain functors are Quillen in the flat model structure on equivariant spectra. 

We now briefly sketch the construction of the flat model structure on $G$-spectra. For more details see~\cite{Stolz11} and~\cite{BrunDundasStolz}. Given a family $\mc{F}$ of subgroups of a group $G$, the $\mc{F}$-model structure on $G$-spaces is the model structure in which a map $f$ is a weak equivalence (resp., fibration) if and only if $f^H$ is a weak homotopy equivalence (resp., Serre fibration) for each $H \in \mc{F}$. We refer to the weak equivalences (resp., fibrations, resp., cofibrations) in this model structure as the $\mc{F}$-equivalences (resp., $\mc{F}$-fibrations, resp., $\mc{F}$-cofibrations). In the case when the family $\mc{F}$ consists of all the subgroups of $G$, we call these the genuine weak equivalences, genuine fibrations and genuine cofibrations.

In the level flat model structure on $G$-spectra, a map $f\colon X \to Y$ is a weak equivalence (resp., fibration) if and only if in each level $V$, the map $f\colon X(V) \to Y(V)$ is an $\mc{F}$-equivalence (resp., $\mc{F}$-fibration) for a certain family of subgroups of $O(V) \rtimes G$, see~\cite[2.3.3]{Stolz11} for more details. Left Bousfield localizing the level flat model structure in such a way that forces the weak equivalences to be the usual $\underline{\pi}_*$-isomorphisms of $G$-spectra then gives the flat model structure. We say that a map of $G$-spectra $f\colon X \to Y$ is a \emph{flat cofibration} if and only if the latching map $\nu_n(f)\colon L_n(f) \to Y_n$ is a genuine $G\times O(n)$-cofibration for all levels $n$, see~\cite[3.5.6]{Schwedeglobal} for more details on latching maps. The flat cofibrations are the cofibrations in both the level flat and flat model structure on $G$-spectra. It is also important to note that the flat cofibrations are independent of the choice of universe on $G$-spectra, see~\cite[2.3.15]{Stolz11}.

\begin{lem}
	There is a Quillen adjoint triple 
	\begin{center}
		\begin{tikzcd}
		\Sp_G^\mathrm{flat} \arrow[rr, "\forget" description] & & \Sp_H^\mathrm{flat} \arrow[ll, "\ind" description, yshift=3mm] \arrow[ll, "\coind" description, yshift=-3mm]
		\end{tikzcd}
	\end{center}
	where both categories are equipped with the flat model structure.
\end{lem}
\begin{proof}
	By~\cite[2.6.11]{BrunDundasStolz} the adjunction $\ind \dashv \forget$ is Quillen with respect to the level flat model structures, and therefore $\ind$ preserves flat cofibrations. Since $\ind$ also preserves all weak equivalences, it follows that it is left Quillen in the flat model structure. The restriction functor $\forget\colon \Sp_G \to \Sp_H$ preserves all weak equivalences and flat cofibrations. Therefore $\forget \dashv \coind$ is a Quillen adjunction in the flat model structure.
\end{proof}

\begin{lem}\label{lem:flatfixed}
	There is a Quillen adjunction
	\begin{center}
		\begin{tikzcd}
		\Sp_G^\mathrm{flat, triv} \arrow[rr, "(-)^G" description, yshift=-2mm] & & \Sp^\mathrm{flat} \arrow[ll, yshift=2mm, "\mathrm{inf}" description]
		\end{tikzcd}
	\end{center}
	where the categories are equipped with the flat model structure, and $\Sp_G$ is indexed on a trivial universe.
\end{lem}
\begin{proof}
	The inflation functor $\mathrm{inf}\colon \Sp \to \Sp_G$ preserves flat cofibrations since the inflation functor on spaces preserves genuine cofibrations. The weak equivalences on both sides are the non-equivariant equivalences, and so $\mathrm{inf}$ preserves weak equivalences too.
\end{proof}

\begin{rem}\label{rem:notQuillen}
If $\Sp_G$ is instead indexed on a complete universe, then 
	\begin{center}
		\begin{tikzcd}
		\Sp_G^\mathrm{flat} \arrow[rr, "(-)^G" description, yshift=-2mm] & & \Sp^\mathrm{flat} \arrow[ll, yshift=2mm, "\mathrm{inf}" description]
		\end{tikzcd}
	\end{center}
is not a Quillen adjunction in general. This follows from the fact that the change of universe functors do not form a Quillen pair in general in the flat model structure, see~\cite[5.1]{Hausmann17}.
\end{rem}

\subsection{A coinduction Quillen equivalence}
The coinduction functor $\coind\colon \Sp_H \to \Sp_G$ is lax symmetric monoidal since the forgetful functor $\forget\colon \Sp_G \to \Sp_H$ is strong symmetric monoidal. Therefore the coinduction functor $\coind$ takes (commutative) ring $H$-spectra to (commutative) ring $G$-spectra. Recall that given a stable model category $\C$, we say that $X \in \C$ \emph{generates} $\C$ if $X$ builds every object of $\C$, that is, $\mathrm{Loc}(X) = h\C$.
\begin{prop}
	Let $H \leq G$ and $R$ be a ring $H$-spectrum. If $R$ generates $\mod{R}(\Sp_H)$, then $\coind R$ generates $\mod{\coind R}(\Sp_G)$.
\end{prop}
\begin{proof}
	Any $\coind R$-module $M$ is a retract of $\coind R \wedge M$ so it is enough to show that $\coind R$ builds $\coind R \wedge M$. By the right projection formula for $\forget \dashv \coind$, we have $\coind R \wedge M \simeq \coind(R \wedge \forget M)$. Since $R$ builds $R \wedge \forget M$ as $R$ generates $\mod{R}(\Sp_H)$ by assumption, the result follows.
\end{proof}

As $H$ is connected (and we work rationally), the commutative ring $H$-spectrum $DEH_+$ generates its category of modules~\cite[3.1]{Greenlees20}. Therefore we obtain the following corollary.
\begin{cor}\label{cor:coindgenerates}
	The commutative ring $G$-spectrum $F_H(G_+,DEH_+)$ generates its category of modules.
\end{cor}

\begin{prop}\label{prop:descentequivalence}
	The adjunction
	\begin{center}
		\begin{tikzcd}
		\mod{DEH_+}({\Sp}_H) \arrow[rrrrr, yshift=-0.8mm, "\coind"'] & & & & & \mod{F_H(G_+,DEH_+)}({\Sp}_G) \arrow[lllll, yshift=0.8mm, "DEH_+ \otimes_{\forget F_H(G_+,DEH_+)} \forget(-)"']
		\end{tikzcd}
	\end{center}
	is a Quillen equivalence.
\end{prop}
\begin{proof}
	The existence of the Quillen adjunction follows from~\cite[\S 3]{SchwedeShipley03}, also see Section~\ref{sec:liftadjunctionstomodules}. The derived counit is an equivalence on $DEH_+$ and therefore as $DEH_+$ and $F_H(G_+,DEH_+)$ generate their categories of modules (see~\cite[3.1]{Greenlees20} and Corollary~\ref{cor:coindgenerates} respectively), by the Cellularization Principle~\cite[2.7]{GreenleesShipley13} the adjunction is a Quillen equivalence.
\end{proof}

\begin{rem}
	Since the adjunction 	
	\begin{center}
		\begin{tikzcd}
		\mod{DEH_+}({\Sp}_H) \arrow[rrrrr, yshift=-0.8mm, "\coind"'] & & & & & \mod{F_H(G_+,DEH_+)}({\Sp}_G) \arrow[lllll, yshift=0.8mm, "DEH_+ \otimes_{\forget F_H(G_+,DEH_+)} \forget(-)"']
		\end{tikzcd}
	\end{center} 
	is Quillen in both the stable and flat model structures, the previous proposition holds in both the stable and flat model structures.
\end{rem}

\begin{rem}
	Proposition~\ref{prop:descentequivalence} could be viewed as an analogue of the results of Balmer-Dell'Ambrogio-Sanders~\cite[1.1]{BalmerDellAmbrogioSanders15} on \'etale extensions, also see~\cite[5.32]{MathewNaumannNoel17}.
\end{rem}

\subsection{Commutativity: step 1}
Let $\omega\colon DEG_+ \to F_H(G_+,DEH_+)$ be the natural map which is adjoint to the identity map on $DEH_+$ (i.e., the unit map $DEG_+ \to \coind\forget DEG_+$). By taking a cofibrant replacement of $F_H(G_+,DEH_+)$ as a commutative $DEG_+$-algebra, we obtain a commutative diagram
\begin{center}
	\begin{tikzcd}
	DEG_+ \arrow[d, "\omega"'] \arrow[dr, hookrightarrow, "\psi"] & \\
	F_H(G_+,DEH_+) & QF_H(G_+,DEH_+). \arrow[l, twoheadrightarrow, "q", "\sim"']
	\end{tikzcd}
\end{center}

We can now turn to showing that the functors correspond. Firstly let us set the following notation.
\begin{nota}
We set the notation $a = \mathrm{dim}(G/H)$. Note that this is the shift that arises in the Wirthm\"uller isomorphism and in the Gorenstein condition for the map $DBG_+ \to DBH_+$, see Theorem~\ref{thm:relgor}.
\end{nota} 

Consider the diagram
\begin{equation}\label{eq:fixedpoints1}
	\begin{tikzcd}
	\mod{DEG_+}(\Sp_G) \arrow[rrrrr, yshift=-2.2mm, "1"' description] \arrow[ddd, "\forget" description] & & & & & \mod{DEG_+}(\Sp_G) \arrow[lllll, yshift=2.2mm, "1" description] \arrow[ddd, "{\ext[\psi]}" description, xshift=5mm] \arrow[ddd, "{\Sigma^{-a}\coext[\psi]}" description, xshift=-5mm] \\
	& & & & & \\
	& & & & & \\
	\mod{DEH_+}(\Sp_H) \arrow[uuu, xshift=-5mm, "\ind" description] \arrow[uuu, xshift=5mm, "\coind" description] \arrow[rrrrr, yshift=-2.2mm, "{\res[q]}\coind"'] & & & & & \mod{QF_H(G_+,DEH_+)}(\Sp_G). \arrow[lllll, yshift=2.2mm, "DEH_+ \otimes_{\forget F_H(G_+,DEH_+)} \forget{\ext[q]}(-)"'] \arrow[uuu, xshift=-15mm, "{\Sigma^a\res[\psi]}" description] \arrow[uuu, xshift=15mm, "{\res[\psi]}" description]
	\end{tikzcd}
\end{equation}

We note that for a $DEH_+$-module $M$, the coinduction $\coind M$ is both a $DEG_+$-module and a $F_H(G_+,DEH_+)$-module: the $DEG_+$-module structure arises from the projection formula map as in Lemma~\ref{lem:modulelemma} and the $F_H(G_+,DEH_+)$-module structure by using the lax monoidal structure on $\coind$.
\begin{prop}\label{prop:fixedpoints1}
	Let $M \in \mod{DEH_+}$. Then $\coind M \cong \res[\psi]\res[q]\coind M$ as $DEG_+$-modules.
\end{prop}
\begin{proof}
	The underlying objects are the same so it is sufficient to check that the module structures agree. Note that $\res[\psi]\res[q] \cong \res[(q\psi)] = \res[\omega]$. We write $R = DEG_+$ to ease the notation. We must check that 
	\begin{center}
		\begin{tikzcd}
		R \wedge \coind M \arrow[dd, "1"'] \arrow[rrdd, "\eta \wedge 1" description] \arrow[rr, "p'"] &  & \coind (\forget R \wedge M) \arrow[rr, "\coind(a)"] \arrow[rrdd, "1" description] &  & \coind M \arrow[rrdd, "1"]                           &  &          \\
		&  &                                                                                   &  &                                                      &  &          \\
		{R \wedge \res[\omega]\coind M} \arrow[rr, "\omega \wedge 1"']                                    &  & \coind\forget R \wedge \coind M \arrow[rr, "l"']                                  &  & \coind (\forget R \wedge M) \arrow[rr, "\coind(a)"'] &  & \coind M
		\end{tikzcd}
	\end{center}
	commutes, where $p'$ is the projection formula map and $l$ is the lax monoidal structure on $\coind$. By definition $\omega$ is the adjoint of the identity map on $\forget R$, or in other words it is the unit $\eta\colon R \to \coind\forget R$ and hence the first triangle commutes. From this one sees that it is enough to just check that the middle square commutes. Spelling out the definitions of $p'$ and $l$, this amounts to checking that
	\begin{center}
		\begin{tikzcd}
		R \wedge \coind M \arrow[r, "\eta"] \arrow[dd, "\eta \wedge 1"'] & \coind\forget (R \wedge \coind M) \arrow[rr, "\coind\Phi^{-1}"] \arrow[dd, "\coind\forget(\eta \wedge 1)" description] & & \coind (\forget R \wedge \forget \coind M) \arrow[ddl, "\coind(\eta \wedge 1)" description] \arrow[dd, "\coind(1 \wedge \epsilon)"]\\
		& & & \\
		\coind\forget R \wedge \coind M \arrow[r, "\eta"'] & \coind\forget(\coind\forget R \wedge \coind M) \arrow[r, "\coind\Phi^{-1}"'] & \coind(\forget\coind\forget R \wedge \forget\coind M) \arrow[r, "\coind(\epsilon \wedge \epsilon)"'] & \coind(\forget R \wedge M) 
		\end{tikzcd}
	\end{center}
	commutes, where going right and down is the map $p'$ and going along the bottom row is the lax monoidal structure map $l$. This diagram commutes using (from left to right) naturality of $\eta$, naturality of $\Phi^{-1}$ and the triangle identities.
\end{proof}

\begin{prop}\label{prop:fixedpoints2}
	For any fibrant $DEG_+$-module $M$, there is a natural weak equivalence $\Sigma^{-a}\coext[\psi]M \xrightarrow{\sim} \res[q]\coind\forget M$.
\end{prop}
\begin{proof}
	The Wirthm\"uller isomorphism gives a natural weak equivalence $G_+ \wedge_H DEH_+ \to \Sigma^a F_H(G_+,DEH_+)$. Since $\Sigma^aQF_H(G_+,DEH_+) \to \Sigma^aF_H(G_+,DEH_+)$ is an acyclic fibration and $G_+ \wedge_H DEH_+$ is cofibrant, by lifting properties we have a natural weak equivalence $G_+ \wedge_H DEH_+ \to \Sigma^aQF_H(G_+,DEH_+)$ of $DEG_+$-modules between cofibrant objects. We then have the string of natural weak equivalences \begin{align*}\Sigma^{-a}\coext[\psi]M &= \Sigma^{-a}\mathrm{Hom}_{DEG_+}(QF_H(G_+,DEH_+),M) \\ &\xrightarrow{\sim} \res[q]\mathrm{Hom}_{DEG_+}((G_+ \wedge_H DEH_+),M)  \\ &\cong \res[q]\mathrm{Hom}_{DEG_+}(G/H_+ \wedge DEG_+, M)\\ &\cong \res[q]F(G/H_+,M) \\ &\cong \res[q]\coind\forget M\end{align*} 
	where the first equivalence follows from Ken Brown's lemma.
\end{proof}

\begin{prop}\label{prop:fixedpointscommute}
The vertical Quillen functors in~(\ref{eq:fixedpoints1}) correspond.
\end{prop}
\begin{proof}
The horizontal adjunctions are Quillen equivalences by Proposition~\ref{prop:descentequivalence}. The derived functors of $\Sigma^{-a}\coext[\psi]$ and $\ext[\psi]$ are isomorphic by the Wirthm\"uller isomorphism. Therefore we are in the context of Theorem~\ref{thm:compare} and the result follows from applying this together with Propositions~\ref{prop:fixedpoints1} and~\ref{prop:fixedpoints2}.
\end{proof}

In order to apply this to (co)free equivariant spectra, we now consider the localized and cellularized versions 
\begin{equation}\label{eq:fixedpoints1cofree}
	\begin{tikzcd}
	L_{EG_+}\mod{DEG_+}(\Sp_G) \arrow[rrrrr, yshift=-2.2mm, "1"' description] \arrow[ddd, "\forget" description] & & & & & L_{EG_+}\mod{DEG_+}(\Sp_G) \arrow[lllll, yshift=2.2mm, "1" description] \arrow[ddd, "{\ext[\psi]}" description, xshift=5mm] \arrow[ddd, "{\Sigma^{-a}\coext[\psi]}" description, xshift=-5mm] \\
	& & & & & \\
	& & & & & \\
	L_{EH_+}\mod{DEH_+}(\Sp_H) \arrow[uuu, xshift=-5mm, "\ind" description] \arrow[uuu, xshift=5mm, "\coind" description] \arrow[rrrrr, yshift=-2.2mm, "{\res[q]}\coind"'] & & & & & L_{q^*\coind EH_+}\mod{QF_H(G_+,DEH_+)}(\Sp_G) \arrow[lllll, yshift=2.2mm, "DEH_+ \otimes_{\forget F_H(G_+,DEH_+)} \forget{\ext[q]}(-)"'] \arrow[uuu, xshift=-15mm, "{\Sigma^a\res[\psi]}" description] \arrow[uuu, xshift=15mm, "{\res[\psi]}" description]
	\end{tikzcd}
\end{equation}
and
\begin{equation}\label{eq:fixedpoints1free}
	\begin{tikzcd}
	\cell_{G_+}\mod{DEG_+}(\Sp_G) \arrow[rrrrr, yshift=-2.2mm, "1"' description] \arrow[ddd, "\forget" description] & & & & & \cell_{G_+}\mod{DEG_+}(\Sp_G) \arrow[lllll, yshift=2.2mm, "1" description] \arrow[ddd, "{\ext[\psi]}" description, xshift=5mm] \arrow[ddd, "{\Sigma^{-a}\coext[\psi]}" description, xshift=-5mm] \\
	& & & & & \\
	& & & & & \\
	\cell_{H_+}\mod{DEH_+}(\Sp_H) \arrow[uuu, xshift=-5mm, "\ind" description] \arrow[uuu, xshift=5mm, "\coind" description] \arrow[rrrrr, yshift=-2.2mm, "{\res[q]}\coind"'] & & & & & \cell_{q^*\coind H_+}\mod{QF_H(G_+,DEH_+)}(\Sp_G). \arrow[lllll, yshift=2.2mm, "DEH_+ \otimes_{\forget F_H(G_+,DEH_+)} \forget{\ext[q]}(-)"'] \arrow[uuu, xshift=-15mm, "{\Sigma^a\res[\psi]}" description] \arrow[uuu, xshift=15mm, "{\res[\psi]}" description]
	\end{tikzcd}
\end{equation}
We must firstly show that the functors pass to the localizations and cellularizations.
\begin{prop}\label{prop:fixedpoints1quillencofree}
	There is a Quillen adjoint triple of functors
	\begin{center}
		\begin{tikzcd}
		L_{\res[q]\coind EH_+}\mod{{QF_H(G_+,DEH_+)}}(\Sp_G) \arrow[rr, "{\res[\psi]}" description] & &  L_{EG_+}\mod{DEG_+}(\Sp_G). \arrow[ll, yshift=-3.5mm, "{\coext[\psi]}" description] \arrow[ll, yshift=3.5mm, "{\ext[\psi]}" description]
		\end{tikzcd}
	\end{center}
\end{prop}
\begin{proof}
Since the diagram
\begin{center}
	\begin{tikzcd}
	DEG_+ \arrow[d, "\omega"'] \arrow[dr, hookrightarrow, "\psi"] & \\
	F_H(G_+,DEH_+) & QF_H(G_+,DEH_+) \arrow[l, twoheadrightarrow, "q", "\sim"']
	\end{tikzcd}
\end{center}
commutes, we have an isomorphism of functors $\ext[\omega] \cong \ext[q]\ext[\psi]$. Taking mates, we have an equivalence $\ext[\psi] \simeq \res[q]\ext[\omega]$ on cofibrant objects. Note that $\ext[\omega] EG_+ \simeq F_H(G_+,EH_+)$ since $$F_H(G_+,DEH_+) \otimes_{DEG_+} EG_+ \simeq F(G/H_+,S^0) \wedge EG_+ \simeq F_H(G_+,EH_+)$$ by smallness of $G/H_+$. Therefore $\ext[\psi]EG_+ \simeq \res[q]F_H(G_+,EH_+)$ and the result follows from Theorem~\ref{thm:quillenafterloc}(1).
\end{proof}

By the Wirthm\"uller isomorphism, $\coind H_+$ is equivalent to a shift of $G_+$ and so Theorem~\ref{thm:quillenaftercell}(1) gives the following result.
\begin{prop}\label{prop:fixedpoints1quillenfree}
	There is a Quillen adjoint triple of functors
	\begin{center}
		\begin{tikzcd}
		\cell_{\res[q]\coind H_+}\mod{{QF_H(G_+,DEH_+)}}(\Sp_G) \arrow[rr, "{\res[\psi]}" description] & &  \cell_{G_+}\mod{DEG_+}(\Sp_G). \arrow[ll, yshift=-3.5mm, "{\coext[\psi]}" description] \arrow[ll, yshift=3.5mm, "{\ext[\psi]}" description]
		\end{tikzcd}
	\end{center}
\end{prop}

\begin{cor}\label{cor:fixedpoints1commute}
The vertical Quillen functors in~(\ref{eq:fixedpoints1cofree}) and~(\ref{eq:fixedpoints1free}) correspond.
\end{cor}
\begin{proof}
The vertical functors in~(\ref{eq:fixedpoints1cofree}) and~(\ref{eq:fixedpoints1free}) are Quillen by Propositions~\ref{prop:changeofringsquillen1},~\ref{prop:changeofringsquillen2},~\ref{prop:fixedpoints1quillencofree} and~\ref{prop:fixedpoints1quillenfree}, so this follows from Proposition~\ref{prop:fixedpointscommute}.
\end{proof}

\subsection{Commutativity: step 2}
We now want to take categorical $G$-fixed points, but in light of Lemma~\ref{lem:flatfixed} and Remark~\ref{rem:notQuillen}, we must take extra care in this step. The key point is that since we are working with (co)free $G$-spectra, we can up to Quillen equivalence restrict to a trivial universe. 

We have a map $\psi\colon DEG_+ \to QF_H(G_+, DEH_+)$ of commutative ring $G$-spectra, and taking categorical $G$-fixed points we obtain a map $\psi^G\colon DBG_+ \to DBH_+$ of commutative ring spectra. 
\begin{rem}
We note that here we abuse notation and write $DBH_+ = (QF_H(G_+, DEH_+))^G$ to simplify notation. The reader may take this as our definition of $DBH_+$. On the other hand, the notation is chosen to be suggestive since $(QF_H(G_+, DEH_+))^G$ is weakly equivalent to the functional dual $F(BH_+, S^0)$ by using the natural isomorphism $(\coind X)^G \cong X^H$ which exists for any $H$-spectrum $X$.
\end{rem}
By taking a cofibrant replacement of $DBH_+$ as a commutative $DBG_+$-algebra, we obtain a commutative diagram
\begin{center}
	\begin{tikzcd}
	DBG_+ \arrow[d, "\psi^G"'] \arrow[dr, hookrightarrow, "\phi"] & \\
	DBH_+ & QDBH_+. \arrow[l, twoheadrightarrow, "\sim"', "c"]
	\end{tikzcd}
\end{center}

The inflation functor is strong monoidal, so as in Section~\ref{sec:liftadjunctionstomodules}, given a commutative ring $G$-spectrum $S$, we obtain an induced adjunction
\[\begin{tikzcd}
\mod{S}(\Sp_G) \arrow[rr, "(-)^G"', yshift=-1mm] & & \mod{S^G}(\Sp) \arrow[ll, yshift=1mm, "\mathrm{inf}^S"']
\end{tikzcd}\]
and moreover, we have the description $\mathrm{inf}^S = S \otimes_{\mathrm{inf}S^G} \mathrm{inf}(-)$. In what follows we will drop the inflation notation for brevity, instead writing $S \otimes_{S^G} -$. We refer the reader to~\cite{GreenleesShipley14c} for further discussion on this adjunction.

\begin{prop}\label{prop:QEfixedpoints}
There are Quillen equivalences
\[\begin{tikzcd}
		L_{EG_+}\mod{DEG_+}^\mathrm{flat} \arrow[rrrr, "(-)^G" description, yshift=-2.1mm] & & & &\cell_{S^0}\mod{DBG_+}^\mathrm{flat} \arrow[llll, yshift=2.1mm, "DEG_+ \otimes_{DBG_+} -" description]
		\end{tikzcd}
\]
and 
\[\begin{tikzcd}
		L_{\res[q]{F_H(G_+,EH_+)}}\mod{QF_H(G_+,DEG_+)}^\mathrm{flat} \arrow[rrrrr, "{\res[c](-)^G}" description, yshift=-2.1mm] & & & & &\cell_{S^0}\mod{QDBH_+}^\mathrm{flat}. \arrow[lllll, yshift=2.1mm, "{QF_H(G_+,DEH_+)} \otimes_{QDBH_+} -" description]
		\end{tikzcd}
\]
\end{prop}
\begin{proof}
The proofs of both statements are similar, so we prove only the first. Consider the composite adjunction
\[\begin{tikzcd}
L_{EG_+}\mod{DEG_+}^\mathrm{flat} \arrow[rr, "1" description, yshift=-2.1mm] & & \mod{DEG_+}^\mathrm{flat, triv} \arrow[ll, "1" description, yshift=2.1mm] \arrow[rrrr, "(-)^G" description, yshift=-2.1mm] & & & & \cell_{S^0}\mod{DBG_+}^\mathrm{flat}.\arrow[llll, yshift=2.1mm, "DEG_+ \otimes_{DBG_+} -" description]
\end{tikzcd} \]
The right hand adjunction is Quillen by Lemma~\ref{lem:flatfixed} since the identity functor \[\mod{DBG_+} \to \cell_{S^0}\mod{DBG_+}\] is right Quillen. For the left hand adjunction, note that in both model categories the weak equivalences are the non-equivariant equivalences and the cofibrations are the flat cofibrations since these are independent of the universe. Therefore the composite adjunction is Quillen. 

To see that it is a Quillen equivalence, consider the following commutative diagram of Quillen adjunctions:
\[
\begin{tikzcd}
L_{EG_+}\mod{DEG_+}^\mathrm{flat} \arrow[rr, "1" description, yshift=-2.1mm] \arrow[dd, "1", xshift=2mm] & & \mod{DEG_+}^\mathrm{flat, triv} \arrow[ll, "1" description, yshift=2.1mm] \arrow[rrrr, "(-)^G" description, yshift=-2.1mm] & & & & \cell_{S^0}\mod{DBG_+}^\mathrm{flat}\arrow[llll, yshift=2.1mm, "DEG_+ \otimes_{DBG_+} -" description] \arrow[dd, "1", xshift=2mm]  \\
& & & & & & \\
\cell_{G_+}\mod{DEG_+}^\mathrm{stable} \arrow[uu, "1", xshift=-2mm] \arrow[rrrrrr, "(-)^G" description, yshift=-2.1mm] & & & & & & \cell_{S^0}\mod{DBG_+}^\mathrm{stable}.\arrow[llllll, yshift=2.1mm, "DEG_+ \otimes_{DBG_+} -" description] \arrow[uu, "1", xshift=-2mm]
\end{tikzcd}
 \]
 The left hand vertical is a Quillen equivalence since in both cases the weak equivalences are the non-equivariant equivalences, and the right hand vertical is also a Quillen equivalence. The bottom horizontal adjunction is a Quillen equivalence by~\cite[8.1]{GreenleesShipley14c}, and therefore by the 2-out-of-3 property of Quillen equivalences, one sees that the desired composite adjunction is a Quillen equivalence.
\end{proof}

\begin{rem}\label{rem:freevscofree}
The localization on the left and the cellularization on the right in Proposition~\ref{prop:QEfixedpoints} are vital to ensure that the functors are Quillen, see Remark~\ref{rem:notQuillen}. Therefore, in this step there are not two different approaches: one for the free case and one for the cofree case; instead we must use the Quillen equivalence 
\[\begin{tikzcd}
L_{EG_+}\mod{DEG_+} \arrow[rr, "1" description, yshift=-2mm] & & \cell_{G_+}\mod{DEG_+} \arrow[ll, "1" description, yshift=2mm]
\end{tikzcd}\]
to pass between the two worlds. More precisely, if one wants to work with the free case, one must first use the above Quillen equivalence to pass to the cofree case. Then applying Proposition~\ref{prop:QEfixedpoints} returns one to the free case. In Section~\ref{sec:backtocofree} we will explain how to pass back to the cofree case.
\end{rem}

The next square we consider is
\begin{equation}\label{eq:fixedpoints2}
	\begin{tikzcd}
	\quad L_{EG_+}\mod{DEG_+}(\Sp_G)\quad \arrow[ddd, "{\ext[\psi]}" description, xshift=5mm] \arrow[ddd, "{\Sigma^{-a}\coext[\psi]}" description, xshift=-5mm] \arrow[rrrrr, yshift=-2.2mm, "(-)^G"' description] & & & & &\quad\cell_{S^0}\mod{DBG_+}(\Sp)\quad \arrow[ddd, "{\ext[\phi]}" description, xshift=5mm] \arrow[ddd, "{\Sigma^{-a}\coext[\phi]}" description, xshift=-5mm] \arrow[lllll, yshift=2.2mm, "DEG_+ \otimes_{DBG_+} -" description] \\
	& & & & & \\
	& & & & & \\
	L_{\res[q]{F_H(G_+,EH_+)}}\mod{QF_H(G_+,DEH_+)}(\Sp_G) \arrow[uuu, xshift=-15mm, "{\Sigma^a\res[\psi]}" description] \arrow[uuu, xshift=15mm, "{\res[\psi]}" description] \arrow[rrrrr, yshift=-2.2mm, "{\res[c]}((-)^G)"' description] & & & & & \quad\cell_{S^0}\mod{QDBH_+}(\Sp).\quad \arrow[uuu, xshift=-15mm, "{\Sigma^a\res[\phi]}" description] \arrow[uuu, xshift=15mm, "{\res[\phi]}" description] \arrow[lllll, yshift=2.2mm, "{QF_H(G_+,DEH_+)} \otimes_{QDBH_+} -" description]
	\end{tikzcd}
\end{equation}
This square is an instance of the framework of Proposition~\ref{prop:compareweak}, but before we can apply this to deduce the correspondence of functors we need to verify that the vertical functors are Quillen.

\begin{prop}\label{prop:changeofrings2quillen}
	There is a Quillen adjoint triple of functors
	\begin{center}
		\begin{tikzcd}
		\cell_{S^0}\mod{QDBH_+}(\Sp) \arrow[rr, "{\res[\phi]}" description] & &  \cell_{S^0}\mod{DBG_+}(\Sp). \arrow[ll, yshift=-3.5mm, "{\coext[\phi]}" description] \arrow[ll, yshift=3.5mm, "{\ext[\phi]}" description]
		\end{tikzcd}
	\end{center}
\end{prop}
\begin{proof}
This follows from Theorem~\ref{thm:quillenaftercell}(1) together with Example~\ref{eg:smallbuild}.
\end{proof}

\begin{prop}\label{prop:fixedpointscommute2}
The vertical Quillen functors in~(\ref{eq:fixedpoints2}) correspond.
\end{prop}
\begin{proof}
The vertical functors in~(\ref{eq:fixedpoints2}) are Quillen by Propositions~\ref{prop:fixedpoints1quillencofree} and~\ref{prop:changeofrings2quillen} and the horizontals are Quillen equivalences by Proposition~\ref{prop:QEfixedpoints}. The derived functors of $\Sigma^{-a}\coext[\phi]$ and $\res[\phi]$ are isomorphic since $\phi$ is relatively Gorenstein by Theorem~\ref{thm:relgor}. The result then follows from Proposition~\ref{prop:compareweak} and Theorem~\ref{thm:compare}.
\end{proof}

\subsection{Returning to the cofree case}\label{sec:backtocofree}
In Remark~\ref{rem:freevscofree} we explained that we have a single approach rather than two separate approaches for treating commutativity in taking categorical fixed points. The previous section explained how this works when landing in the free case. If one wishes to then pass back to the cofree case, one must use the Quillen equivalence
\[\begin{tikzcd}
L_{BG_+}\mod{DBG_+} \arrow[rr, "1" description, yshift=-2mm] & & \cell_{S^0}\mod{DBG_+} \arrow[ll, "1" description, yshift=2mm]
\end{tikzcd} \]
as we now describe. This amounts to considering the square
\begin{equation}\label{eq:fixedpoints3}
	\begin{tikzcd}
	\cell_{S^0}\mod{DBG_+}(\Sp) \arrow[ddd, "{\ext[\phi]}" description, xshift=5mm] \arrow[ddd, "{\Sigma^{-a}\coext[\phi]}" description, xshift=-5mm] \arrow[rrrr, yshift=2.2mm, "1"' description] & & & & L_{BG_+}\mod{DBG_+}(\Sp) \arrow[ddd, "{\ext[\phi]}" description, xshift=5mm] \arrow[ddd, "{\Sigma^{-a}\coext[\phi]}" description, xshift=-5mm] \arrow[llll, yshift=-2.2mm, "1" description] \\
	& & & & \\
	& & & & \\
	\cell_{S^0}\mod{QDBH_+}(\Sp) \arrow[uuu, xshift=-15mm, "{\Sigma^a\res[\phi]}" description] \arrow[uuu, xshift=15mm, "{\res[\phi]}" description] \arrow[rrrr, yshift=2.2mm, "1"' description] & & & & L_{BH_+}\mod{QDBH_+}(\Sp). \arrow[uuu, xshift=-15mm, "{\Sigma^a\res[\phi]}" description] \arrow[uuu, xshift=15mm, "{\res[\phi]}" description] \arrow[llll, yshift=-2.2mm, "1" description]
	\end{tikzcd}
\end{equation}

Firstly we need to justify that the extension, restriction and coextension of scalars along $\phi\colon DBG_+ \to QDBH_+$ are Quillen functors between the localizations. In order to prove this we require a couple of preparatory results.
We write $\Gamma_k$ for the $k$-cellularization functor, i.e., the right adjoint to the inclusion of the localizing subcategory generated by $k$.
\begin{lem}\label{lem:cellularization and extension of scalars}
	Let $\theta\colon S \to R$ be a map of ring spectra such that $R$ is small over $S$, and let $k$ be an $R$-algebra. Then for $M$ an $S$-module, we have $R \otimes_S^\L \Gamma_kM \simeq \Gamma_k(R \otimes_S^\L M)$.
\end{lem}
\begin{proof}
	Throughout the proof all functors are implicitly derived. We must check that $R \otimes_S \Gamma_kM$ satisfies the universal properties of the cellularization. Firstly, there is a natural map $R \otimes_S \Gamma_kM \to R \otimes_S M$ as there is a natural map $\Gamma_kM \to M$. Since $S$ finitely builds $R$, $k$ finitely builds $R \otimes_S k$, so $R \otimes_S \Gamma_kM$ is $k$-cellular. It remains to check that $R \otimes_S \Gamma_kM \to R \otimes_S M$ is a $k$-cellular equivalence. 
	
	Note that there are natural equivalences $$\mathrm{Hom}_R(k,R \otimes_S -) \simeq \mathrm{Hom}_R(k,\mathrm{Hom}_S(DR,-)) \simeq \mathrm{Hom}_S(DR,\mathrm{Hom}_S(k,-))$$ since $R$ is a small $S$-module. It follows that $R \otimes_S \Gamma_kM \to R \otimes_S M$ is a $k$-cellular equivalence as $\Gamma_kM \to M$ is a $k$-cellular equivalence.
\end{proof}

\begin{cor}\label{cor:extension of BG}
	There is an equivalence \[\L\ext[\phi](BG_+) = QDBH_+ \otimes^\L_{DBG_+} \Sigma^{\mathrm{dim}(G)}BG_+ \simeq \Sigma^{\mathrm{dim}(H)}BH_+.\]
\end{cor}
\begin{proof}
	This follows from Lemma~\ref{lem:cellularization and extension of scalars} using that $QDBH_+$ is a small $DBG_+$-module by Venkov's theorem (see after Proposition~\ref{prop:small fin gen}) and the fact that $\Gamma_{S^0}DBG_+ \simeq \Sigma^{\mathrm{dim}(G)}BG_+$ by Gorenstein duality~\cite{DwyerGreenleesIyengar06}. 
\end{proof}

\begin{prop}\label{prop:changeofrings2quillenloc}
	There is a Quillen adjoint triple of functors
	\begin{center}
		\begin{tikzcd}
		L_{BH_+}\mod{QDBH_+}(\Sp) \arrow[rr, "{\res[\phi]}" description] & &  L_{BG_+}\mod{DBG_+}(\Sp). \arrow[ll, yshift=-3.5mm, "{\coext[\phi]}" description] \arrow[ll, yshift=3.5mm, "{\ext[\phi]}" description]
		\end{tikzcd}
	\end{center}
\end{prop}
\begin{proof}
	We have $\langle BH_+ \rangle = \langle \L\ext[\phi](BG_+) \rangle$ by Corollary~\ref{cor:extension of BG} so the claim follows from Theorem~\ref{thm:quillenafterloc}(1).
\end{proof}

In order to see that we can pass back to the cofree case from the free case, it remains to check that the horizontals in~(\ref{eq:fixedpoints3}) are Quillen equivalences.
\begin{prop}\label{prop:fixedpoints3}
The horizontal adjunctions in~(\ref{eq:fixedpoints3}) are Quillen equivalences. Therefore the vertical functors in~(\ref{eq:fixedpoints3}) correspond.
\end{prop}
\begin{proof}
We prove the claim for the top horizontal adjunction; the bottom horizontal adjunction follows similarly. Consider the diagram
\[
\begin{tikzcd}
L_{EG_+}\mod{DEG_+} \arrow[dd, "1", xshift=2mm] \arrow[rrrr, "(-)^G" description, yshift=-2.1mm] & & & & L_{BG_+}\mod{DBG_+} \arrow[llll, yshift=2.1mm, "DEG_+ \otimes_{DBG_+} -" description] \arrow[dd, "1", xshift=2mm]  \\
& & & & \\
\cell_{G_+}\mod{DEG_+} \arrow[uu, "1", xshift=-2mm] \arrow[rrrr, "(-)^G" description, yshift=-2.1mm] & & & & \cell_{S^0}\mod{DBG_+}. \arrow[llll, yshift=2.1mm, "DEG_+ \otimes_{DBG_+} -" description] \arrow[uu, "1", xshift=-2mm]
\end{tikzcd}
 \]
 The left hand vertical is a Quillen equivalence since in both model categories the weak equivalences are the non-equivariant equivalences. The horizontal
 adjunctions are Quillen equivalences since $DEG_+$ generates its category of modules by~\cite[3.1]{Greenlees20}; see~\cite[\S 8 Step 2]{PolWilliamson} and~\cite[8.1]{GreenleesShipley14c} for more details. Therefore the right hand vertical is a Quillen equivalence as required.
\end{proof}

\section{Change of model}\label{sec:model}
In this section we pass from orthogonal spectra to symmetric spectra since the results of~\cite{Shipley07} use symmetric spectra as their model, whereas we have been using orthogonal spectra so far in order to discuss equivariant spectra. 

We will write $\Sp^\mathscr{O}$ and $\Sp^\Sigma$ for the categories of orthogonal spectra and symmetric spectra respectively. The forgetful functor $\mathbb{U}\colon \Sp^\mathscr{O} \to \Sp^\Sigma$ has a left adjoint $\mathbb{P}$ and this forms a strong monoidal Quillen equivalence, see~\cite{MandellMaySchwedeShipley01} for details. (The argument given there is for the stable model structure, but a similar argument directly applies to the flat model structure also.) Therefore, given any commutative orthogonal ring spectrum $S$, we obtain a Quillen equivalence
\[\begin{tikzcd}
\mod{S}(\Sp^{\mathscr{O}}) \arrow[rr, "\mathbb{U}" description, yshift=-2mm] & & \mod{\mathbb{U}S}(\Sp^\Sigma) \arrow[ll, "\mathbb{P}^S" description, yshift=2mm]
\end{tikzcd}\]
as in Section~\ref{sec:liftadjunctionstomodules}. We can now specialize to our case of interest.

We have a map $\phi\colon DBG_+ \to QDBH_+$ of commutative orthogonal ring spectra as discussed in the previous section. Applying $\mathbb{U}$ gives a map $\mathbb{U}\phi\colon \mathbb{U}DBG_+ \to \mathbb{U}QDBH_+$ of commutative symmetric ring spectra. Cofibrantly replacing $\mathbb{U}QDBH_+$ as a commutative $\mathbb{U}DBG_+$-algebra gives a commutative diagram 
\begin{center}
	\begin{tikzcd}
	\mathbb{U}DBG_+ \arrow[d, "\mathbb{U}\phi"'] \arrow[dr, hookrightarrow, "\psi"] & \\
	\mathbb{U}QDBH_+ & Q\mathbb{U}QDBH_+. \arrow[l, twoheadrightarrow, "\sim"', "q"]
	\end{tikzcd}
\end{center}
This puts us in the context of Proposition~\ref{prop:compareweak}, and the next square we consider is 
\begin{equation}\label{eq:changeofmodel}
	\begin{tikzcd}
	\mod{DBG_+}(\Sp^\mathscr{O}) \arrow[rrr, yshift=-2mm,"{\mathbb{U}}" description] \arrow[dd, "{\ext[\phi]}" description, xshift=4.2mm] \arrow[dd, "{\Sigma^{-a}\coext[\phi]}" description, xshift=-4.2mm] & & & \mod{\mathbb{U}DBG_+}(\Sp^\Sigma)  \arrow[lll, yshift=2mm,"\mathbb{P}^{DBG_+}" description] \arrow[dd, "{\ext[\psi]}" description, xshift=4.2mm] \arrow[dd, "{\Sigma^{-a}\coext[\psi]}" description, xshift=-4.2mm]
	\\
	& & \\
	\mod{QDBH_+}(\Sp^\mathscr{O}) \arrow[rr, yshift=-2mm,"\mathbb{U}" description] \arrow[uu, xshift=-12.6mm, "{\Sigma^a\res[\phi]}" description] \arrow[uu, xshift=12.6mm, "{\res[\phi]}" description] & & \mod{\mathbb{U}QDBH_+}(\Sp^\Sigma) \arrow[ll, yshift=2mm,"\mathbb{P}^{QDBH_+}" description] \arrow[r, yshift=-2mm,"{\res[q]}" description] & \mod{Q\mathbb{U}QDBH_+}(\Sp^\Sigma). \arrow[uu, xshift=-13mm, "{\Sigma^a\res[\psi]}" description] \arrow[uu, xshift=12.6mm, "{\res[\psi]}" description]  \arrow[l, yshift=2mm,"{\ext[q]}" description]
	\end{tikzcd}
\end{equation}

\begin{prop}\label{prop:changeofmodelcorrespond}
The vertical Quillen functors in~(\ref{eq:changeofmodel}) correspond.
\end{prop}
\begin{proof}
The square satisfies the hypotheses of Proposition~\ref{prop:compareweak}. By combining Theorem~\ref{thm:relgor} with Proposition~\ref{prop:relGorweak} one sees that each of the vertical ring maps are relatively Gorenstein of shift $a$. Therefore it follows from Theorem~\ref{thm:compare} that the vertical Quillen functors correspond.
\end{proof}

In order for this to be applicable to the case of (co)free equivariant spectra, we need to verify that the previous proposition is also true after the relevant localization and cellularization. To this end we consider the following diagrams. We note that in first diagram below we write $S^0$ for both the orthogonal and symmetric sphere spectrum; since $\mathbb{U}(S^0) \simeq S^0$ this is justified. We make a similar abuse of notation in the second diagram also.
\begin{equation}\label{eq:changeofmodelcell}
	\begin{tikzcd}
	\cell_{S^0}\mod{DBG_+}(\Sp^\mathscr{O}) \arrow[rrr, yshift=-2mm,"{\mathbb{U}}" description] \arrow[dd, "{\ext[\phi]}" description, xshift=4.2mm] \arrow[dd, "{\Sigma^{-a}\coext[\phi]}" description, xshift=-4.2mm] & & & \cell_{S^0}\mod{\mathbb{U}DBG_+}(\Sp^\Sigma)  \arrow[lll, yshift=2mm,"\mathbb{P}^{DBG_+}" description] \arrow[dd, "{\ext[\psi]}" description, xshift=4.2mm] \arrow[dd, "{\Sigma^{-a}\coext[\psi]}" description, xshift=-4.2mm]
	\\
	& & \\
	\cell_{S^0}\mod{QDBH_+}(\Sp^\mathscr{O}) \arrow[rr, yshift=-2mm,"\mathbb{U}" description] \arrow[uu, xshift=-12.6mm, "{\Sigma^a\res[\phi]}" description] \arrow[uu, xshift=12.6mm, "{\res[\phi]}" description] & & \cell_{S^0}\mod{\mathbb{U}QDBH_+}(\Sp^\Sigma) \arrow[ll, yshift=2mm,"\mathbb{P}^{QDBH_+}" description] \arrow[r, yshift=-2mm,"{\res[q]}" description] & \cell_{S^0}\mod{Q\mathbb{U}QDBH_+}(\Sp^\Sigma) \arrow[uu, xshift=-13mm, "{\Sigma^a\res[\psi]}" description] \arrow[uu, xshift=12.6mm, "{\res[\psi]}" description]  \arrow[l, yshift=2mm,"{\ext[q]}" description]
	\end{tikzcd}
\end{equation}

\begin{equation}\label{eq:changeofmodelloc}
	\begin{tikzcd}
	L_{BG_+}\mod{DBG_+}(\Sp^\mathscr{O}) \arrow[rrr, yshift=-2mm,"{\mathbb{U}}" description] \arrow[dd, "{\ext[\phi]}" description, xshift=4.2mm] \arrow[dd, "{\Sigma^{-a}\coext[\phi]}" description, xshift=-4.2mm] & & & L_{BG_+}\mod{\mathbb{U}DBG_+}(\Sp^\Sigma)  \arrow[lll, yshift=2mm,"\mathbb{P}^{DBG_+}" description] \arrow[dd, "{\ext[\psi]}" description, xshift=4.2mm] \arrow[dd, "{\Sigma^{-a}\coext[\psi]}" description, xshift=-4.2mm]
	\\
	& & \\
	L_{BH_+}\mod{QDBH_+}(\Sp^\mathscr{O}) \arrow[rr, yshift=-2mm,"\mathbb{U}" description] \arrow[uu, xshift=-12.6mm, "{\Sigma^a\res[\phi]}" description] \arrow[uu, xshift=12.6mm, "{\res[\phi]}" description] & & L_{BH_+}\mod{\mathbb{U}QDBH_+}(\Sp^\Sigma) \arrow[ll, yshift=2mm,"\mathbb{P}^{QDBH_+}" description] \arrow[r, yshift=-2mm,"{\res[q]}" description] & L_{BH_+}\mod{Q\mathbb{U}QDBH_+}(\Sp^\Sigma) \arrow[uu, xshift=-13mm, "{\Sigma^a\res[\psi]}" description] \arrow[uu, xshift=12.6mm, "{\res[\psi]}" description]  \arrow[l, yshift=2mm,"{\ext[q]}" description]
	\end{tikzcd}
\end{equation}

\begin{cor}\label{cor:changeofmodel}
The vertical Quillen functors in~(\ref{eq:changeofmodelcell}) and~(\ref{eq:changeofmodelloc}) correspond.
\end{cor}
\begin{proof}
It is enough to show that the vertical functors in~(\ref{eq:changeofmodelcell}) and~(\ref{eq:changeofmodelloc}) are Quillen, and then the result is a consequence of Proposition~\ref{prop:changeofmodelcorrespond}. The case of~(\ref{eq:changeofmodelcell}) follows from Theorem~\ref{thm:quillenaftercell}(1) together with Example~\ref{eg:smallbuild}. For the case of~(\ref{eq:changeofmodelloc}), we note that the left hand vertical functors are Quillen by Proposition~\ref{prop:changeofrings2quillenloc}, so it follows from Theorem~\ref{thm:quillenafterloc}(1) and Lemma~\ref{lem:bousfieldclasses} that the right hand vertical functors are also Quillen. 
\end{proof}

\section{Shipley's algebraicization theorem}\label{sec:shipley}
In this section we use the results of~\cite{Shipley07, FlatnessandShipley'sTheorem} to pass between modules over the commutative symmetric ring spectra $\mathbb{U}DBG_+$ and $Q\mathbb{U}QDBH_+$ to modules over commutative DGAs, whilst keeping track of the change of rings adjunctions. This consists of three steps which we will deal with separately. The first step involves passage from symmetric spectra in simplicial sets to symmetric spectra in simplicial $\mathbb{Q}$-modules, the second step is a stabilized version of the Dold-Kan equivalence and the final step is the passage to algebra. Shipley~\cite{Shipley07} proved that these Quillen equivalences hold in the stable model structure. For our purposes, the flat model structure is crucial, and the Quillen equivalences are proven to hold in the flat model structures in~\cite{FlatnessandShipley'sTheorem}. 

\subsection{Recap of the equivalences}
We provide a brief recap of the Quillen equivalences used in Shipley's algebraicization theorem and use this as an opportunity to set notation. For more details see~\cite{Shipley07} and~\cite{FlatnessandShipley'sTheorem}. 

Let $\C$ be a bicomplete, closed symmetric monoidal category. Let $\Sigma$ be the category whose objects are the finite sets $\underline{n} = \{1,\ldots ,n\}$ for $n \geq 0$ where $\underline{0} = \emptyset$, and whose morphisms are the bijections of $\underline{n}$. The category of \emph{symmetric sequences} in $\C$ is the functor category $\C^\Sigma$. For an object $K \in \C$, the category of \emph{symmetric spectra} $\text{Sp}^\Sigma(\C,K)$ is the category of modules over $\text{Sym}(K)$ in $\C^\Sigma$, where $\text{Sym}(K) = (\mathds{1}, K, K^{\otimes 2}, \cdots)$ is the free commutative monoid on $K$.

We write $\widetilde{\mathbb{Q}}\colon \sset_*^\Sigma \to \sq^\Sigma$ for the functor which takes the free simplicial $\mathbb{Q}$-module levelwise on the non-basepoint simplices. Recall that we define $\Sp^\Sigma(\sq)$ to be the category of modules over $\mathrm{Sym}(\widetilde{\mathbb{Q}}S^1)$ in $\sq^\Sigma$. The object $\mathrm{Sym}(\widetilde{\mathbb{Q}}S^1)$ is equivalent to $H\mathbb{Q}$. There is a ring map $\alpha\colon H\mathbb{Q} \to \widetilde{\mathbb{Q}}H\mathbb{Q}$, and the composite $\res[\alpha]\widetilde{\mathbb{Q}}$ gives a zig-zag of strong monoidal Quillen equivalences between $\mod{H\mathbb{Q}}$ and $\Sp^\Sigma(\sq)$.

We write $\Sp^\Sigma(\chq)$ for the category of modules over $\mathrm{Sym}(\mathbb{Q}[1])$ in $(\chq)^\Sigma$ where $\mathbb{Q}[1]$ is the chain complex which contains a single copy of $\mathbb{Q}$ in degree 1. Applying the normalized chains functor $N\colon \sq \to \mathrm{Ch}_\mathbb{Q}^+$ levelwise gives a functor $\Sp^\Sigma(\sq) \to \mod{\N}((\chq)^\Sigma)$ where $\N = N\mathrm{Sym}(\widetilde{\mathbb{Q}}S^1)$. There is a ring map $\phi\colon \mathrm{Sym}(\mathbb{Q}[1]) \to \N$, and the composite $$\res[\phi]N\colon \Sp^\Sigma(\sq) \to \Sp^\Sigma(\chq)$$ is a right Quillen equivalence. This forms a weak monoidal Quillen equivalence. 

Finally, there is a functor $R\colon \mathrm{Ch}_\mathbb{Q} \to \Sp^\Sigma(\chq)$ which is defined by $RY_n = C_0(Y \otimes \mathbb{Q}[n])$ where $C_0$ denotes the connective cover, and this functor is a right Quillen equivalence. Together with its left adjoint $D$, this forms a strong monoidal Quillen equivalence.

Since all of these functors are appropriately monoidal, they lift to give Quillen equivalences between the categories of modules over monoids as described in Section~\ref{sec:liftadjunctionstomodules}.

\subsection{To simplicial \texorpdfstring{$\mathbb{Q}$}{Q}-modules}
Although we are interested in the particular case which corresponds to the case of (co)free $G$-spectra, it is convenient to work in a more general setting for clarity and to ease notation. 

We will work in the general setting of a map of commutative (symmetric) ring spectra $\phi\colon S \to R$ and suppose that $R$ is a cofibrant $S$-module. We will moreover assume that $\phi\colon S \to R$ is relatively Gorenstein of shift $a$. Let $K$ be an $S$-module and $E$ be an $R$-module; these will play the role of the objects which we cellularize and localize at respectively.

We write $\kappa = \widetilde{\mathbb{Q}}\phi\colon \widetilde{\mathbb{Q}}S \to \widetilde{\mathbb{Q}}R$. Since $\widetilde{\mathbb{Q}}$ is left Quillen, $\widetilde{\mathbb{Q}}R$ is cofibrant as a $\widetilde{\mathbb{Q}}S$-module. Cofibrantly replacing $\res[\alpha]\widetilde{\mathbb{Q}}R$ as a commutative $\res[\alpha]\widetilde{\mathbb{Q}}S$-algebra gives a commutative diagram
\begin{center}
	\begin{tikzcd}
	\res[\alpha]\widetilde{\mathbb{Q}}S \arrow[d, "{\res[\alpha]}\kappa"'] \arrow[dr, hookrightarrow, "\delta"]& \\
	\res[\alpha]\widetilde{\mathbb{Q}}R & Q\res[\alpha]\widetilde{\mathbb{Q}}R. \arrow[l, twoheadrightarrow, "q"] \arrow[l, "\sim"']
	\end{tikzcd}
\end{center}

The passage to simplicial $\mathbb{Q}$-modules can be seen as the two squares
\begin{equation}\label{eq:shipley1}
	\begin{tikzcd}
	\qquad\mod{S}\qquad \arrow[rrr, yshift=2mm, "\widetilde{\mathbb{Q}}" description] \arrow[dd, "{\ext[\psi]}" description, xshift=4.2mm] \arrow[dd, "{\Sigma^{-a}\coext[\psi]}" description, xshift=-4.2mm] & & &  \qquad\mod{\widetilde{\mathbb{Q}}S}\qquad \arrow[rrr, "{\res[\alpha]}" description, yshift=-2mm] \arrow[lll, "U" description, yshift=-2mm] \arrow[dd, "{\ext[\kappa]}" description, xshift=4.2mm] \arrow[dd, "{\Sigma^{-a}\coext[\kappa]}" description, xshift=-4.2mm] & & & \qquad\mod{\res[\alpha]\widetilde{\mathbb{Q}}S}\qquad \arrow[lll, "{\ext[\alpha]^{\widetilde{\mathbb{Q}}S}}" description, yshift=2mm] \arrow[dd, "{\ext[\delta]}" description, xshift=4.2mm] \arrow[dd, "{\Sigma^{-a}\coext[\delta]}" description, xshift=-4.2mm] \\
	& & & & & & \\
	\qquad\mod{R}\qquad \arrow[rrr, yshift=2mm, "\widetilde{\mathbb{Q}}" description] \arrow[uu, xshift=-12.8mm, "{\Sigma^a\res[\psi]}" description] \arrow[uu, xshift=12.6mm, "{\res[\psi]}" description] & & &  \qquad\mod{\widetilde{\mathbb{Q}}R}\qquad \arrow[rrr, "{\res[q]\res[\alpha]}" description, yshift=-2.2mm] \arrow[lll, "U" description, yshift=-2mm] \arrow[uu, xshift=-12.6mm, "{\Sigma^a\res[\kappa]}" description] \arrow[uu, xshift=12.6mm, "{\res[\kappa]}" description] & & &  \qquad\mod{Q\res[\alpha]\widetilde{\mathbb{Q}}R}.\qquad \arrow[lll, "{\ext[\alpha]^{\widetilde{\mathbb{Q}}R}\ext[q]}" description, yshift=2.2mm] \arrow[uu, xshift=-12.6mm, "{\Sigma^a\res[\delta]}" description] \arrow[uu, xshift=12.6mm, "{\res[\delta]}" description]
	\end{tikzcd}
\end{equation}

\begin{prop}\label{prop:shipley1correspond}
The vertical Quillen functors in~(\ref{eq:shipley1}) correspond.
\end{prop}
\begin{proof}
The first square satisfies the hypotheses of Proposition~\ref{prop:comparestrong} and the second square satisfies the hypotheses of Proposition~\ref{prop:compareweak}. Since $\phi\colon S \to R$ is relatively Gorenstein of shift $a$ by assumption, applying Proposition~\ref{prop:relGorstrong} and~\ref{prop:relGorweak} one sees that $\kappa$ and $\delta$ are also relatively Gorenstein of shift $a$. Therefore it follows from Theorem~\ref{thm:compare} that the vertical Quillen functors correspond.
\end{proof}

In order to apply the results to the case of free and (co)free $G$-spectra we need to check that the previous proposition remains true after the relevant localizations and cellularizations. Therefore we consider the diagrams
\begin{equation}\label{eq:shipley1cell}
	\begin{tikzcd}
	\cell_{\res[\psi]K}\mod{S} \arrow[rrr, yshift=2mm, "\widetilde{\mathbb{Q}}" description] \arrow[dd, "{\ext[\psi]}" description, xshift=4.2mm] \arrow[dd, "{\Sigma^{-a}\coext[\psi]}" description, xshift=-4.2mm] & & &  \cell_{\widetilde{\mathbb{Q}}\res[\psi]K}\mod{\widetilde{\mathbb{Q}}S} \arrow[rrr, "{\res[\alpha]}" description, yshift=-2mm] \arrow[lll, "U" description, yshift=-2mm] \arrow[dd, "{\ext[\kappa]}" description, xshift=4.2mm] \arrow[dd, "{\Sigma^{-a}\coext[\kappa]}" description, xshift=-4.2mm] & & & \cell_{\res[\alpha]\widetilde{\mathbb{Q}}\res[\psi]K}\mod{\res[\alpha]\widetilde{\mathbb{Q}}S} \arrow[lll, "{\ext[\alpha]^{\widetilde{\mathbb{Q}}S}}" description, yshift=2mm] \arrow[dd, "{\ext[\delta]}" description, xshift=4.2mm] \arrow[dd, "{\Sigma^{-a}\coext[\delta]}" description, xshift=-4.2mm] \\
	& & & & & & \\
	\cell_{K}\mod{R} \arrow[rrr, yshift=2mm, "\widetilde{\mathbb{Q}}" description] \arrow[uu, xshift=-12.8mm, "{\Sigma^a\res[\psi]}" description] \arrow[uu, xshift=12.6mm, "{\res[\psi]}" description] & & &  \cell_{\widetilde{\mathbb{Q}}K}\mod{\widetilde{\mathbb{Q}}R} \arrow[rrr, "{\res[q]\res[\alpha]}" description, yshift=-2.2mm] \arrow[lll, "U" description, yshift=-2mm] \arrow[uu, xshift=-12.6mm, "{\Sigma^a\res[\kappa]}" description] \arrow[uu, xshift=12.6mm, "{\res[\kappa]}" description] & & & \cell_{\res[q]\res[\alpha]\widetilde{\mathbb{Q}}K}\mod{Q\res[\alpha]\widetilde{\mathbb{Q}}R} \arrow[lll, "{\ext[\alpha]^{\widetilde{\mathbb{Q}}R}\ext[q]}" description, yshift=2.2mm] \arrow[uu, xshift=-12.6mm, "{\Sigma^a\res[\delta]}" description] \arrow[uu, xshift=12.6mm, "{\res[\delta]}" description]
	\end{tikzcd}
\end{equation}
and
\begin{equation}\label{eq:shipley1L}
	\begin{tikzcd}
	\quad L_{E}\mod{S}\quad \arrow[rrr, yshift=2mm, "\widetilde{\mathbb{Q}}" description] \arrow[dd, "{\ext[\psi]}" description, xshift=4.2mm] \arrow[dd, "{\Sigma^{-a}\coext[\psi]}" description, xshift=-4.2mm] & & & \quad L_{\widetilde{\mathbb{Q}}E}\mod{\widetilde{\mathbb{Q}}S}\quad\arrow[rrr, "{\res[\alpha]}" description, yshift=-2mm] \arrow[lll, "U" description, yshift=-2mm] \arrow[dd, "{\ext[\kappa]}" description, xshift=4.2mm] \arrow[dd, "{\Sigma^{-a}\coext[\kappa]}" description, xshift=-4.2mm] & & & L_{\res[q]\res[\alpha]\widetilde{\mathbb{Q}}E}\mod{\res[\alpha]\widetilde{\mathbb{Q}}S} \arrow[lll, "{\ext[\alpha]^{\widetilde{\mathbb{Q}}S}}" description, yshift=2mm] \arrow[dd, "{\ext[\delta]}" description, xshift=4.2mm] \arrow[dd, "{\Sigma^{-a}\coext[\delta]}" description, xshift=-4.2mm] \\
	& & & & & & \\
	\quad L_{\ext[\psi]E}\mod{R}\quad \arrow[rrr, yshift=2mm, "\widetilde{\mathbb{Q}}" description] \arrow[uu, xshift=-12.8mm, "{\Sigma^a\res[\psi]}" description] \arrow[uu, xshift=12.6mm, "{\res[\psi]}" description] & & & \quad L_{\widetilde{\mathbb{Q}}\ext[\psi]E}\mod{\widetilde{\mathbb{Q}}R} \quad \arrow[rrr, "{\res[q]\res[\alpha]}" description, yshift=-2.2mm] \arrow[lll, "U" description, yshift=-2mm] \arrow[uu, xshift=-12.6mm, "{\Sigma^a\res[\kappa]}" description] \arrow[uu, xshift=12.6mm, "{\res[\kappa]}" description] & & &  L_{\res[q]\res[\alpha]\widetilde{\mathbb{Q}}\ext[\psi]E}\mod{Q\res[\alpha]\widetilde{\mathbb{Q}}R}. \arrow[lll, "{\ext[\alpha]^{\widetilde{\mathbb{Q}}R}\ext[q]}" description, yshift=2.2mm] \arrow[uu, xshift=-12.6mm, "{\Sigma^a\res[\delta]}" description] \arrow[uu, xshift=12.6mm, "{\res[\delta]}" description]
	\end{tikzcd}
\end{equation}

\begin{cor}\label{cor:shipley1}
The vertical Quillen functors in~(\ref{eq:shipley1cell}) and~(\ref{eq:shipley1L}) correspond.
\end{cor}
\begin{proof}
This is a consequence of Proposition~\ref{prop:shipley1correspond} once we show that the vertical functors in~(\ref{eq:shipley1cell}) and~(\ref{eq:shipley1L}) are Quillen functors. The case of~(\ref{eq:shipley1L}) follows immediately from Theorem~\ref{thm:quillenafterloc}(1) and Lemma~\ref{lem:bousfieldclasses}. For the case of the cellularizations we note that by Theorem~\ref{thm:quillenaftercell}(1) and Example~\ref{eg:smallbuild} it is enough to show that each cell on the top row is obtained from the corresponding cell on the bottom row by restriction of scalars. For the left most functors this was part of the setup, and for the other cases, this follows since lax monoidal functors commute with restriction of scalars as in Proposition~\ref{prop:comparestrong}. 
\end{proof}

\subsection{The Dold-Kan type equivalence}
Since $\res[\phi]N$ is lax symmetric monoidal, we obtain a map of commutative monoids $\res[\phi]N\delta$, and by cofibrant replacement in commutative algebras, we obtain a commutative diagram
\begin{center}
	\begin{tikzcd}
	\res[\phi]N\res[\alpha]\widetilde{\mathbb{Q}}S \arrow[d, "{\res[\phi]N\delta}"'] \arrow[dr, hookrightarrow, "\gamma"]& \\
	\res[\phi]NQ\res[\alpha]\widetilde{\mathbb{Q}}R & Q\res[\phi]NQ\res[\alpha]\widetilde{\mathbb{Q}}R. \arrow[l, twoheadrightarrow, "q"] \arrow[l, "\sim"']
	\end{tikzcd}
\end{center}

From these we get the square which gives the passage from symmetric spectra in simplicial $\mathbb{Q}$-modules to symmetric spectra in non-negatively graded chain complexes of $\mathbb{Q}$-modules, as shown in the following diagram. Note that since $(L,\res[\phi]N)$ is only a weak monoidal Quillen pair as an adjunction between $\symsp(\sq)$ and $\symsp(\chq)$, the left adjoint to $\res[\phi]N$ at the level of modules is not $L$. 

\begin{equation}\label{eq:shipley2}
	\begin{tikzcd}
	\qquad\mod{{\res[\alpha]}\widetilde{\mathbb{Q}}S}\qquad \arrow[rrr, yshift=-2mm,"{\res[\phi]N}" description] \arrow[dd, "{\ext[\delta]}" description, xshift=4.2mm] \arrow[dd, "{\Sigma^{-a}\coext[\delta]}" description, xshift=-4.2mm] & & & \quad\mod{{\res[\phi]}N{\res[\alpha]}\widetilde{\mathbb{Q}}S}\quad  \arrow[lll, yshift=2mm,"L^{{\res[\alpha]}\widetilde{\mathbb{Q}}S}" description] \arrow[dd, "{\ext[\gamma]}" description, xshift=4.2mm] \arrow[dd, "{\Sigma^{-a}\coext[\gamma]}" description, xshift=-4.2mm]
	\\
	& & \\
	\qquad\mod{Q{\res[\alpha]}\widetilde{\mathbb{Q}}R}\qquad \arrow[rr, yshift=-2mm,"{\res[\phi]N}" description] \arrow[uu, xshift=-12.6mm, "{\Sigma^a\res[\delta]}" description] \arrow[uu, xshift=12.6mm, "{\res[\delta]}" description] & & \mod{{\res[\phi]}NQ{\res[\alpha]}\widetilde{\mathbb{Q}}R} \arrow[ll, yshift=2mm,"L^{Q{\res[\alpha]}\widetilde{\mathbb{Q}}R}" description] \arrow[r, yshift=-2mm,"{\res[q]}" description] & \quad\mod{Q{\res[\phi]}NQ{\res[\alpha]}\widetilde{\mathbb{Q}}R}\quad \arrow[uu, xshift=-13mm, "{\Sigma^a\res[\gamma]}" description] \arrow[uu, xshift=12.6mm, "{\res[\gamma]}" description]  \arrow[l, yshift=2mm,"{\ext[q]}" description]
	\end{tikzcd}
\end{equation}

\begin{prop}\label{prop:shipley2}
The vertical Quillen functors in~(\ref{eq:shipley2}) correspond.
\end{prop}
\begin{proof}
Firstly note that $\gamma$ is relatively Gorenstein of shift $a$ by Proposition~\ref{prop:relGorweak} since $\delta$ is as explained in the proof of Proposition~\ref{prop:shipley1correspond}. Then apply Theorem~\ref{thm:compare} together with Proposition~\ref{prop:compareweak}.
\end{proof}

We now consider the corresponding localized and cellularized versions of~(\ref{eq:shipley2}). In order to not overcomplicate the notation we will not write down these localized and cellularized versions explicitly. They are obtained from~(\ref{eq:shipley2}) by cellularizing/localizing at the images of the cells/localizing objects just as~(\ref{eq:shipley1cell}) and~(\ref{eq:shipley1L}) were obtained from~(\ref{eq:shipley1}).

\begin{cor}\label{cor:shipley2}
The vertical Quillen functors in~(\ref{eq:shipley2}) correspond after the appropriate localizations and cellularizations.
\end{cor}
\begin{proof}
It is enough to check that the vertical functors are Quillen after the appropriate localizations and cellularizations. This follows from Theorems~\ref{thm:quillenafterloc} and~\ref{thm:quillenaftercell} together with Lemma~\ref{lem:bousfieldclasses} as in Corollary~\ref{cor:shipley1}.
\end{proof}

\subsection{To algebra}
The final stage in the process is the passage from symmetric spectra in non-negatively graded chain complexes to chain complexes. This can be realized by a single square
\begin{equation}\label{eq:shipley3}
	\begin{tikzcd}
	\quad\mod{{\res[\phi]}N{\res[\alpha]}\widetilde{\mathbb{Q}}S}\quad \arrow[rr, yshift=2mm,"D" description]\arrow[dd, "{\ext[\gamma]}" description, xshift=4.2mm] \arrow[dd, "{\Sigma^{-a}\coext[\gamma]}" description, xshift=-4.2mm] & & \quad\mod{D{\res[\phi]}N{\res[\alpha]}\widetilde{\mathbb{Q}}S}\quad \arrow[ll, yshift=-2mm,"R" description] \arrow[dd, "{\ext[\lambda]}" description, xshift=4.2mm] \arrow[dd, "{\Sigma^{-a}\coext[\lambda]}" description, xshift=-4.2mm] \\
	& & \\
	\quad\mod{Q{\res[\phi]}NQ{\res[\alpha]}\widetilde{\mathbb{Q}}R}\quad \arrow[uu, xshift=-13mm, "{\Sigma^a\res[\gamma]}" description] \arrow[uu, xshift=12.6mm, "{\res[\gamma]}" description]  \arrow[rr, yshift=2mm,"D" description] & & \quad\mod{DQ{\res[\phi]}NQ{\res[\alpha]}\widetilde{\mathbb{Q}}R}\quad \arrow[ll, yshift=-2mm,"R" description] \arrow[uu, xshift=-13mm, "{\Sigma^a\res[\lambda]}" description] \arrow[uu, xshift=12.6mm, "{\res[\lambda]}" description] 
	\end{tikzcd}
\end{equation}
in which $\lambda = D\gamma$. Since $D$ is left Quillen, $DQ{\res[\phi]}NQ{\res[\alpha]}\widetilde{\mathbb{Q}}R$ is a cofibrant $D{\res[\phi]}N{\res[\alpha]}\widetilde{\mathbb{Q}}S$-module. 

\begin{prop}\label{prop:shipley3}
The vertical Quillen functors in~(\ref{eq:shipley3}) correspond.
\end{prop}
\begin{proof}
Since $\gamma$ is relatively Gorenstein of shift $a$ as explained in the proof of Proposition~\ref{prop:shipley2}, the ring map $\lambda$ is also relatively Gorenstein of shift $a$ by Proposition~\ref{prop:relGorstrong}. The result then follows from Theorem~\ref{thm:compare} together with Proposition~\ref{prop:comparestrong}.
\end{proof}

\begin{cor}\label{cor:shipley3}
The vertical Quillen functors in~(\ref{eq:shipley3}) correspond after the appropriate localizations and cellularizations.
\end{cor}
\begin{proof}
This is analogous to the proof of Corollary~\ref{cor:shipley1}.
\end{proof}

Taking $S=\mathbb{U}DBG_+$, $R=Q\mathbb{U}QDBH_+$, $K = S^0$ and $E = BH_+$ shows that the functors correspond through Shipley's algebraicization theorem in our setting. We note that despite starting in the flat model structure on symmetric spectra, we are now in the usual projective model structure on chain complexes. 

\section{The formality square}
Recall from Section~\ref{sec:model} that we have a map of commutative symmetric ring spectra $\mathbb{U}DBG_+ \to Q\mathbb{U}QDBH_+$. To ease notation, we write $S_a$ and $R_a$ for the commutative DGAs which are the images of $\mathbb{U}DBG_+$ and $Q\mathbb{U}QDBH_+$ respectively under the Quillen equivalences in Shipley's algebraicization theorem described in the previous section. They are commutative models for the cochains on $BG$ and $BH$ respectively but they are not good models for us since it is not clear how build a commutative square 
\begin{center}
	\begin{tikzcd}
	S_a \arrow[d] & H^*BG \arrow[l] \arrow[d] \\
	R_a & H^*BH. \arrow[l] 
	\end{tikzcd}
\end{center}
We now describe how to alter these models so that a commutative square can be built with the desired properties. 

Firstly factor the map $\lambda\colon S_a \to R_a$ as an acyclic cofibration of commutative DGAs followed by a fibration $$S_a \xrightarrow{j_G} S_a' \xrightarrow{\nu} R_a.$$ 

Since $\nu\colon S_a' \to R_a$ is a fibration (i.e., a surjection) we can build a commutative square 
\begin{equation}\label{eq:surjection}
	\begin{tikzcd}
	S_a \arrow[r, hookrightarrow, "j_G"] \arrow[d, "\lambda"'] & S_a' \arrow[dl, twoheadrightarrow, "\nu"] & H^*BG \arrow[l, "w"'] \arrow[l, "\sim"] \arrow[d, "\theta"] \\
	R_a & & H^*BH \arrow[ll, "w'"] \arrow[ll, "\sim"'] 
	\end{tikzcd}
\end{equation}
as we now describe. This method was also used in~\cite[Proof of 9.1]{GreenleesShipley11}.

We first define $w'\colon H^*BH \to R_a$ by choosing cocycle representatives $\tilde y_i$ of the polynomial generators $y_i$ of $H^*BH$. Write $H^*BG = \mathbb{Q}[x_1,...,x_r]$. Choosing cocycle representatives $\tilde x_i'$ for the polynomial generators will not yield a commutative square in general. However, the cohomology classes $\nu(\tilde x_i')$ and $w'\theta (x_i)$ are cohomologous since the map $DBG_+ \to DBH_+$ which gives rise to $\lambda\colon S_a \to R_a$ via Shipley's algebraicization theorem represents the map $\theta$ in homotopy. Therefore the differences $\nu(\tilde x_i') - w'\theta (x_i)$ are coboundaries $d(b_i)$. As $\nu\colon S_a'\to R_a$ is a surjection, we can lift the coboundary $d(b_i)$ to give a coboundary $a_i$ in $S_a'$ such that $\nu a_i = d(b_i)$. Define $w(x_i) = \tilde x_i' - a_i$. Then $$\nu w(x_i) = \nu \tilde{x}_i' - \nu a_i = \nu\tilde{x}_i' - d(b_i) = w'\theta (x_i)$$ so the square commutes. Note also that $w$ and $w'$ are quasiisomorphisms by construction.

Next factor the map $\theta\colon H^*BG \to H^*BH$ into a cofibration followed by an acyclic fibration $$H^*BG \xrightarrow{\phi} QH^*BH \xrightarrow{q} H^*BH$$ of commutative DGAs. Taking the pushout of the span $S_a' \xleftarrow{w} H^*BG \xrightarrow{\phi} QH^*BH$ gives the commutative diagram
\begin{equation}\label{eq:pushout}
	\begin{tikzcd}
	& S_a' \arrow[ddl, bend right, "\nu"'] \arrow[d, hookrightarrow, "\delta"'] & H^*BG \arrow[l, "w"'] \arrow[l, "\sim"] \arrow[d, hookrightarrow, "\phi"] \\
	& R_a' \arrow[dl, "\alpha" description] & QH^*BH \arrow[dll, bend left, "w'q"] \arrow[l, "z"] \arrow[l, "\sim"'] \\
	R_a & & \\
	\end{tikzcd}
\end{equation}
where $z$ is a weak equivalence by left properness. 

\begin{lem}\label{lem:formalitydiagram}
Using the notation above, the diagram
$$
\begin{tikzcd}
S_a \arrow[d, hookrightarrow, "\lambda"'] & & S_a \arrow[ll, "1"', "\sim"] \arrow[d, hookrightarrow, "\delta j_G"'] \arrow[rr, "j_G", "\sim"', hookrightarrow] & & S_a' \arrow[d, hookrightarrow, "\delta"] & & H^*BG \arrow[ll, "w"', "\sim"] \arrow[d, hookrightarrow, "\phi"] \\
R_a & & R_a' \arrow[ll, "\alpha", "\sim"'] \arrow[rr, "1"', "\sim"] & & R_a'  & & QH^*BH \arrow[ll, "z", "\sim"']
\end{tikzcd}
$$
is commutative.
\end{lem}
\begin{proof}
The left hand square is commutative since $\alpha\delta j_G = \nu j_G = \lambda$ using~(\ref{eq:pushout}) and~(\ref{eq:surjection}) respectively. The middle square commutes by definition, and the right hand square commutes by~(\ref{eq:pushout}).
\end{proof}

The formality step now consists of the following three squares induced by the diagram in Lemma~\ref{lem:formalitydiagram}. We write $\beta = \delta j_G$ to ease notation.
\begin{equation}\label{eq:formality}
\resizebox{\hsize}{!}{\begin{tikzcd}[ampersand replacement=\&]
	\qquad\mod{S_a}\qquad \arrow[rr, yshift=-2mm, "1" description] \arrow[dd, xshift=-4.2mm, "{\Sigma^{-a}\coext[\lambda]}" description] \arrow[dd, xshift=4.2mm, "{\ext[\lambda]}" description] \& \& \qquad\mod{S_a}\qquad \arrow[ll, yshift=2mm, "1" description] \arrow[rr, yshift=2mm, "{\ext[j]^G}" description] \arrow[dd, xshift=-4.2mm, "{\Sigma^{-a}\coext[\beta]}" description] \arrow[dd, xshift=4.2mm, "{\ext[\beta]}" description] \& \& \qquad\mod{S_a'}\qquad \arrow[rr, yshift=-2mm, "{\res[w]}" description]\arrow[ll, yshift=-2mm, "{\res[j]_G}" description]\arrow[dd, xshift=-4.2mm, "{\Sigma^{-a}\coext[\delta]}" description] \arrow[dd, xshift=4.2mm, "{\ext[\delta]}" description]\& \&\quad\mod{H^*BG}\quad \arrow[ll, yshift=2mm, "{\ext[w]}" description] \arrow[dd, xshift=-4.2mm, "{\Sigma^{-a}\coext[\phi]}" description] \arrow[dd, xshift=4.2mm, "{\ext[\phi]}" description] \\
	\& \& \& \& \& \& \\
	\qquad\mod{R_a}\qquad \arrow[rr, yshift=-2mm, "{\res[\alpha]}" description] \arrow[uu, xshift=-13mm, "{\Sigma^a\res[\lambda]}" description] \arrow[uu, xshift=13mm, "{\res[\lambda]}" description]\& \& \qquad\mod{R_a'}\qquad \arrow[ll, yshift=2mm, "{\ext[\alpha]}" description] \arrow[rr, yshift=2mm, "1" description] \arrow[uu, xshift=-13mm, "{\Sigma^a\res[\beta]}" description] \arrow[uu, xshift=13mm, "{\res[\beta]}" description]\& \& \qquad\mod{R_a'}\qquad \arrow[rr, yshift=-2mm, "{\res[z]}" description]\arrow[ll, yshift=-2mm, "1" description] \arrow[uu, xshift=-13mm, "{\Sigma^a\res[\delta]}" description] \arrow[uu, xshift=13mm, "{\res[\delta]}" description]\& \&\quad\mod{QH^*BH}\quad \arrow[ll, yshift=2mm, "{\ext[z]}" description] \arrow[uu, xshift=-13mm, "{\Sigma^a\res[\phi]}" description] \arrow[uu, xshift=13mm, "{\res[\phi]}" description]
	\end{tikzcd}}
\end{equation}

\begin{prop}\label{prop:formality}
The vertical Quillen functors in~(\ref{eq:formality}) correspond.
\end{prop}
\begin{proof}
All three squares satisfy the hypotheses of Proposition~\ref{prop:ringsquare} using Remark~\ref{rem:pushout}, so the result follows from Theorem~\ref{thm:compare}.
\end{proof}

\begin{rem}
	The splitting of this formality step into three pieces is required to ensure that the vertical functors are Quillen. In order to guarantee this, the vertical maps of commutative DGAs need to be cofibrations, see Proposition~\ref{prop:extrescoext}.
\end{rem}

We now need to check that the correspondence still holds after localization/cellularization at the appropriate objects. Firstly we need to identify the images of the objects to localize/cellularize at. We write $\Theta$ for the derived functor $\mod{DBG_+} \to \mod{S_a}$ which realizes the Quillen equivalence of Section~\ref{sec:shipley}. We recall that $\Theta$ has the property that $H_*(\Theta M) = \pi_*M$.
\begin{lem}
\leavevmode
\begin{enumerate}
\item We have $\Theta(S^0) \simeq \mathbb{Q}$.
\item We have $\Theta(BG_+) \simeq H_*BG$, and moreover, there is an equality of Bousfield classes $\langle H_*BG \rangle = \langle \mathbb{Q} \rangle$.
\end{enumerate}
\end{lem}
\begin{proof}
Since $\pi_*S^0 = \mathbb{Q}$ which is intrinsically formal as it is concentrated in degree zero, the first claim follows; see~\cite[8.10]{GreenleesShipley11} for more details. We have $\Theta(BG_+) \simeq H_*BG$ by~\cite[8.2]{PolWilliamson}. Fix an $S_a$-module $M$, and note that the collection of $S_a$-modules $Z$ for which $Z \otimes _{S_a}^{\mathbb{L}} M \simeq 0$ is localizing. Since $\mathbb{Q}$ and $H_*BG$ build each other by~\cite[8.3]{PolWilliamson} it follows that $\langle H_*BG \rangle = \langle \mathbb{Q} \rangle$.
\end{proof}

The previous result tells us that it is sufficient to localize/cellularize at $\mathbb{Q}$.
\begin{cor}\label{cor:formality}
The vertical functors in~(\ref{eq:formality}) correspond after the appropriate localization and cellularization.
\end{cor}
\begin{proof}
The vertical functors are Quillen after the localization/cellularization by Theorems~\ref{thm:quillenafterloc} and~\ref{thm:quillenaftercell}, so the result follows from Proposition~\ref{prop:formality}.
\end{proof}

Combining the results of the previous sections gives a proof of Theorem~\ref{thm:main} but we will make this explicit in the next section
once we have more language to explain it. 

\section{Torsion and completion}
There are two options for the final step: either to pass to torsion $H^*BG$-modules, or $L$-complete $H^*BG$-modules. We show that unless the ranks of $G$ and $H$ are equal, there are obstructions to forming the correct setup in this step. 

\subsection{Koszul complexes and local (co)homology}
In this section we recall some key definitions which we require for the rest of the section. For more detail see~\cite[\S 6]{DwyerGreenlees02}. 

For $I=(x_1,\cdots,x_n)$ a finitely generated homogeneous ideal in a graded commutative ring $S$, the \emph{stable Koszul complex} denoted $K_\infty(I)$ is defined by $$K_\infty(I) = K_\infty(x_1) \otimes_S \cdots \otimes_S K_\infty(x_n)$$ where $K_\infty(x_i)$ is the complex $S \to S[1/x_i]$ in degrees 0 and -1. Note that up to quasiisomorphism, the stable Koszul complex $K_\infty(I)$ does not depend on the chosen set of generators for $I$, and in fact only depends on $I$ up to radical~\cite[1.2]{GreenleesMay95b}.

The \emph{local homology} of an $S$-module $N$ is defined by $$H_n^IN = H_n\mathrm{Hom}_S(PK_\infty(I),N)$$ where $PK_\infty(I)$ is a projective replacement of $K_\infty(I)$, see~\cite[\S 1]{GreenleesMay95b} for instance. We write $\Lambda_I = \mathbb{R}\mathrm{Hom}_S(K_\infty(I),-)$ for the derived completion functor and say that $N$ is \emph{derived complete} if the natural map $N \to \Lambda_IN$ is a quasiisomorphism. 

The \emph{local cohomology} of an $S$-module $N$ is defined by $$H^n_IN = H_{-n}(K_\infty(I) \otimes_S N).$$ We write $\Gamma_I = K_\infty(I) \otimes_R^\mathbb{L} -$ for the derived torsion functor, and say that $N$ is \emph{derived torsion} if the natural map $\Gamma_IN \to N$ is a quasiisomorphism.

We note that when $S$ is Noetherian the 0th local homology is the $L$-completion functor $$L_0^I = L_0((-)_I^\wedge))$$ which is the 0th left derived functor of $I$-adic completion~\cite[2.5]{GreenleesMay92}. We say that an $S$-module $M$ is $L_0^I$-complete if the natural map $M \to L_0^IM$ is an isomorphism, and write $\mod{S}^\wedge$ for the full subcategory of $L_0^I$-complete $S$-modules. 

Similarly, for $S$ Noetherian, the 0th local cohomology is the $I$-power torsion functor $$T_IM = \{m \in M\mid I^nm = 0 \text{ for some $n$}\}$$ by a result of Grothendieck~\cite{Hartshorne67}. We say that an $S$-module $M$ is ($I$-power) torsion if the natural map $T_IM \to M$ is an isomorphism, and write $\mod{S}^\mathrm{torsion}$ for the full subcategory of torsion $S$-modules.

We now turn to the interaction of local homology with change of rings. Given a map of rings $\theta\colon S \to R$ and an ideal $I = (x_1,\ldots ,x_n)$ of $S$, we write $IR$ for the ideal $(\theta(x_1), \ldots, \theta(x_n))$ of $R$.
\begin{lem}\label{lem:local homology base change}
	Let $\theta\colon S \to R$ be a map of rings and $I$ be an ideal in $S$. There is an isomorphism $H^I_*(\res M) \cong H^{IR}_*(M)$ for any $R$-module $M$. Furthermore, if $R$ is a projective $S$-module, there is an isomorphism $\mathrm{Hom}_S(R,H^I_*(N)) \cong H^{IR}_*(\mathrm{Hom}_S(R,N))$ for any $S$-module $N$.
\end{lem}
\begin{proof}
	Throughout this proof we neglect to indicate that homs are derived.
	Recall that there is an isomorphism $R \otimes_S K_\infty(I) \cong K_\infty(IR)$. Therefore $$\Lambda_I(\res M) = \mathrm{Hom}_S(K_\infty(I),\res M) \cong \mathrm{Hom}_R(R \otimes_S K_\infty(I),M) \cong \mathrm{Hom}_R(K_\infty(IR),M) = \Lambda_{IR}M$$ which gives the first desired isomorphism by taking homology. 
	
	For the other isomorphism, note that \begin{align*}\coext\Lambda_IN &= \mathrm{Hom}_S(R, \mathrm{Hom}_S(K_\infty(I),N)) \\ &\cong \mathrm{Hom}_S(R \otimes_S K_\infty(I),N) \\ &\cong \mathrm{Hom}_S(K_\infty(I), \mathrm{Hom}_S(R,N))\\ &\cong \mathrm{Hom}_R(R \otimes_S K_\infty(I), \mathrm{Hom}_S(R,N))\\ &= \Lambda_{IR}(\coext N).\end{align*} By taking homology, the result then follows from the fact that $\coext$ is exact since $R$ is projective as an $S$-module.
\end{proof}

We now recall the change of base theorem for local cohomology. We omit the proof as it is similar to Lemma~\ref{lem:local homology base change}.
\begin{lem}\label{lem:local cohomology base change}
	Let $\theta\colon S \to R$ be a map of rings and $I$ be an ideal in $S$. There is an isomorphism $H_I^*(\res M) \cong H_{IR}^*(M)$ for any $R$-module $M$. Furthermore, if $R$ is a flat $S$-module, there is an isomorphism $R \otimes_S H_I^*(N) \cong H_{IR}^*(R \otimes_S N)$ for any $S$-module $N$.
\end{lem}

\subsection{Cofibrancy}
We have a natural map $\theta\colon H^*BG \to H^*BH$ induced by the inclusion of $H$ into $G$. Throughout the remainder of this section, we write $I$ for the augmentation ideal of $H^*BG$ and $J$ for the augmentation ideal of $H^*BH$. We note that $\sqrt{I \cdot H^*BH} = J$ by Venkov's theorem, using the general fact that if $(S,\mathfrak{n}) \to (R, \mathfrak{m})$ is a map of local rings with $R$ a finitely generated $S$-module, then $\sqrt{\mathfrak{n}R} = \mathfrak{m}$ by the going-up theorem, see~\cite[5.3]{BensonGreenlees97} for instance. Up to quasiisomorphism, the stable Koszul complex only depends on the radical of the ideal, so we can apply the base change results from the previous section in this example.

Recall from Proposition~\ref{prop:extrescoext} that the restriction-coextenstion of scalars adjunction along $\theta$ is Quillen (in the projective model structure) if and only if $H^*BH$ is cofibrant as a $H^*BG$-module. A dg-module $M$ over a dga $S$ is called \emph{semi-projective} if $\mathrm{Hom}_S(M,-)$ preserves surjective quasiisomorphisms. The semi-projective dg-$S$-modules are precisely the cofibrant objects in the projective model structure on $\mod{S}$~\cite[3.15]{Abbasirad}.  
\begin{prop}\label{prop:cofibrantrank}
	Let $H \to G$ be the inclusion of a connected compact Lie group into a connected compact Lie group. The following are equivalent:
	\begin{enumerate}
		\item $\mathrm{rk}G = \mathrm{rk}H$;
		\item $H^*BH$ is a cofibrant $H^*BG$-module (in the projective model structure);
		\item $H^*BH$ is a free $H^*BG$-module.
	\end{enumerate}
\end{prop}
\begin{proof}\leavevmode
	$(1) \Rightarrow (3)$: If $\mathrm{rk}G=\mathrm{rk}H$ then $H^*BH \cong H^*BG \otimes H^*(G/H)$ as a $H^*BG$-module~(see~\cite[8.3]{McCleary01}) so $H^*BH$ is a free $H^*BG$-module.
	
	$(3) \Rightarrow (2)$: Both $H^*BH$ and $H^*BG$ have trivial differential so by~\cite[9.8.1]{AvramovFoxbyHalperin03}, since $H^*BH$ is underlying projective as a $H^*BG$-module (as it is free) it is semi-projective and hence cofibrant.
	
	$(2) \Rightarrow (3)$: Recall that semi-projective implies underlying projective~\cite[9.6.1,~9.4.1]{AvramovFoxbyHalperin03}. By Venkov's theorem $H^*BH$ is a finitely generated $H^*BG$-module, so $H^*BH$ is free over $H^*BG$ as finitely generated projective modules over local rings are free.
	
	$(3) \Rightarrow (1)$: Firstly, note that the rank of $G$ is the Krull dimension of $H^*BG$. As $H^*BH$ is a free $H^*BG$-module, we have that $\mathrm{dim}_\mathbb{Q}H^sBH \geq \mathrm{dim}_\mathbb{Q}H^sBG$ for all $s$. It follows from the characterization of Krull dimension in terms of Hilbert series that $\mathrm{rk}H \geq \mathrm{rk}G$. As $H$ is a subgroup of $G$, we also have $\mathrm{rk}H \leq \mathrm{rk}G$.
\end{proof}

\begin{cor}\label{cor:extension isomorphic to coextension}
	If any of the equivalent conditions of Proposition~\ref{prop:cofibrantrank} hold, then $\ext M \cong \coext M$.
\end{cor}
\begin{proof}
	If any of the equivalent conditions of Proposition~\ref{prop:cofibrantrank} hold, then $H^*BH$ is finitely generated and free over $H^*BG$. Therefore $H^*BH$ is a finite sum of shifted copies of $H^*BG$ and the result follows.
\end{proof}

\begin{rem}\label{rem:equalranks}
	One may expect that if $H^*BH$ is not cofibrant as a $H^*BG$-module that we may cofibrantly replace it in the category of commutative $H^*BG$-algebras. Note that $H^*BG$ and $H^*BH$ have zero differential, but the cofibrant replacement $QH^*BH$ will not have zero differential. This means that it is not possible to talk about $L$-complete or torsion modules over $QH^*BH$. This is why Theorem~\ref{thm:main} has a stronger statement in the case that the ranks of $G$ and $H$ are equal. 
\end{rem}

\subsection{Completion}
Firstly we must show that the functors pass to the category of $L$-complete modules.
\begin{lem}
	Suppose that $\mathrm{rk}G=\mathrm{rk}H$. There are Quillen adjunctions 
	\begin{center}
		\begin{tikzcd}
		\mod{H^*BG}^\wedge \arrow[rr, "\ext" description, yshift=3.5mm] \arrow[rr, "\coext" description, yshift=-3.5mm] & & \mod{H^*BH}^\wedge. \arrow[ll, "\res" description]
		\end{tikzcd}
	\end{center}
\end{lem}
\begin{proof}
	By Lemma~\ref{lem:local homology base change} and Proposition~\ref{prop:cofibrantrank} the restriction functor $\res$ and the coextension functor $\coext$ send $L$-complete modules to $L$-complete modules. Combining this with Corollary~\ref{cor:extension isomorphic to coextension} shows that $\ext$ also sends $L$-complete modules to $L$-complete modules. Since the weak equivalences and fibrations of $L$-complete modules are created by the inclusion to all modules, it is immediate that $\res$ and $\coext$ are still both right Quillen.
\end{proof}

It remains to consider the square
\begin{equation}\label{eq:Lcomplete}
	\begin{tikzcd}
	L_\mathbb{Q}\mod{H^*BG} \arrow[dd, "\ext" description] \arrow[rr, yshift=2mm, "L_0^I" description] & & \mod{H^*BG}^\wedge \arrow[dd, "\ext" description] \arrow[ll, yshift=-2mm, "i" description] \\
	& & \\
	L_\mathbb{Q}\mod{H^*BH} \arrow[uu, xshift=5mm, "\res" description] \arrow[uu, xshift=-5mm, "\res" description] \arrow[rr, yshift=2mm, "L_0^J" description] & & \mod{H^*BH}^\wedge \arrow[uu, xshift=5mm, "\res" description] \arrow[uu, xshift=-5mm, "\res" description] \arrow[ll, yshift=-2mm, "j" description] 
	\end{tikzcd}
\end{equation}
where $i$ and $j$ denote the inclusions and $\mathrm{rk}G=\mathrm{rk}H$. 

\begin{prop}\label{prop:Lcompletecorrespondence}
The vertical Quillen functors in~(\ref{eq:Lcomplete}) correspond.
\end{prop}
\begin{proof}
The horizontal adjunctions are Quillen equivalences by~\cite[6.10]{PolWilliamson}. The inclusions commute with all vertical functors, and therefore applying Theorem~\ref{thm:compare} gives the result.
\end{proof}

\subsection{Torsion}
Another possible model instead of $L$-complete modules is that of torsion modules. The category of torsion modules does not have enough projectives, but it has does have enough injectives, so it supports an injective model structure~\cite[8.6]{GreenleesShipley11}. Since the injective model structure is not right lifted from the underlying category of chain complexes, the results of Section~\ref{sec:quillen functors} do not apply. Therefore, we must next recall when the extension-restriction-coextension functors are Quillen in the injective model structure. A dg-$S$-module $M$ is said to be \emph{semi-flat} if $M \otimes_S -$ preserves injective quasiisomorphisms.
\begin{prop}
	Let $\theta\colon S \to R$ be a map of DGAs. The adjunction $\res \dashv \coext$ is Quillen in the injective model structure. The adjunction $\ext \dashv \res$ is Quillen in the injective model structure if and only if $R$ is semi-flat as an $S$-module.
\end{prop}
\begin{proof}
	Since the weak equivalences and cofibrations are underlying, the restriction of scalars functor $\res$ preserves them. Therefore, $\res \dashv \coext$ is Quillen.
	
	If $R$ is a semi-flat $S$-module, $R \otimes_S -$ preserves injective quasiisomorphisms. It follows from~\cite[11.2.1]{AvramovFoxbyHalperin03} that $R$ is linearly flat, meaning that $R \otimes_S -$ preserves injections (cofibrations). Therefore $R \otimes_S -$ preserves cofibrations and acyclic cofibrations so is a left Quillen functor. Conversely, if the adjunction is Quillen, then $R \otimes_S -$ preserves injective quasiisomorphisms, so $R$ is a semi-flat $S$-module.
\end{proof}

Note that this proposition shows that if the ranks of $G$ and $H$ are the same, then the extension-restriction and restriction-coextension adjunctions along the ring map $H^*BG \to H^*BH$ are Quillen in the injective model structure.

\begin{lem}
	Suppose that $\mathrm{rk}G=\mathrm{rk}H$. There are Quillen adjunctions 
	\begin{center}
		\begin{tikzcd}
		\mod{H^*BG}^\mathrm{torsion}
		\arrow[rr, "\ext" description, yshift=3.5mm] \arrow[rr, "\coext" description, yshift=-3.5mm] & & \mod{H^*BH}^\mathrm{torsion}. \arrow[ll, "\res" description]
		\end{tikzcd}
	\end{center}
\end{lem}
\begin{proof}
	By Lemma~\ref{lem:local cohomology base change}, Proposition~\ref{prop:cofibrantrank} and Corollary~\ref{cor:extension isomorphic to coextension} we have that extension, restriction and coextension preserve torsion modules. Since the weak equivalences and cofibrations of torsion modules are created by the inclusion to all modules, it is immediate that $\ext$ and $\res$ are still both left Quillen.
\end{proof}

It remains to consider the square
\begin{equation}\label{eq:torsion}
	\begin{tikzcd}
	\cell_\mathbb{Q}\mod{H^*BG} \arrow[dd, "\ext" description] \arrow[rr, yshift=-2mm, "T_I" description] & & \mod{H^*BG}^{\mathrm{torsion}} \arrow[dd, "\ext" description] \arrow[ll, yshift=2mm, "i" description] \\
	& & \\
	\cell_\mathbb{Q}\mod{H^*BH} \arrow[uu, xshift=5mm, "\res" description] \arrow[uu, xshift=-5mm, "\res" description] \arrow[rr, yshift=-2mm, "T_J" description] & & \mod{H^*BH}^{\mathrm{torsion}} \arrow[uu, xshift=5mm, "\res" description] \arrow[uu, xshift=-5mm, "\res" description] \arrow[ll, yshift=2mm, "j" description] 
	\end{tikzcd}
\end{equation}
where $i$ and $j$ denote the inclusions and $\mathrm{rk}G=\mathrm{rk}H$. 

\begin{prop}\label{prop:torsioncorrespondence}
The vertical Quillen functors in~(\ref{eq:torsion}) correspond.
\end{prop}
\begin{proof}
The horizontals are Quillen equivalences by~\cite[8.7]{GreenleesShipley11}. The inclusions commute with all the vertical functors, and therefore applying Theorem~\ref{thm:compare} gives the result. 
\end{proof}

\subsection{Proofs of the main results}
We now combine the results of the previous sections to prove Theorems~\ref{thm:main} and~\ref{thm:main2}.
\begin{proof}[Proof of Theorem~\ref{thm:main}]
Firstly, note that the cellularization at $\mathbb{Q}$ is a model for derived torsion modules and that the localization at $\mathbb{Q}$ is a model for derived complete modules; see~\cite[\S 5]{Greenlees01Tate} and~\cite{GreenleesMay95b}. This justifies the terminology used in the statement of the theorem. The result then follows from applying Proposition~\ref{prop:changeofrings}, Corollary~\ref{cor:fixedpoints1commute}, Propositions~\ref{prop:fixedpointscommute2} and~\ref{prop:fixedpoints3} with Corollaries~\ref{cor:changeofmodel},~\ref{cor:shipley1},~\ref{cor:shipley2},~\ref{cor:shipley3} and~\ref{cor:formality} in turn.
\end{proof}

\begin{proof}[Proof of Theorem~\ref{thm:main2}]
This follows from combining Propositions~\ref{prop:Lcompletecorrespondence} and~\ref{prop:torsioncorrespondence} with Theorem~\ref{thm:main}.
\end{proof}

\appendix
\section{Proof of Lemma~\ref{lem:modulelemma}}\label{sec:appendix}
We conclude the proof of Lemma~\ref{lem:modulelemma}.
\begin{proof}
	We must check that the module structures defined in the proof of Lemma~\ref{lem:modulelemma} satisfy the associativity and unit axiom. We prove this for $\ind$. The proof for $\coind$ is similar and therefore we omit it. Associativity amounts to the outer square commuting in the following diagram.
	\begin{center}
		\begin{tikzcd}
		R \wedge R \wedge \ind M \arrow[dddd, "\mu \wedge 1"'] \arrow[r, "1 \wedge p^{-1}"] \arrow[rdd, "p^{-1}"'] & R \wedge \ind (\forget R \wedge M) \arrow[dd, "p^{-1}"] \arrow[r, "1 \wedge \ind (a)"] & R \wedge \ind M \arrow[dd, "p^{-1}"] \\
		& {} & {} \\
		& \ind (\forget R \wedge \forget R \wedge M) \arrow[dd, "\ind (\mu \wedge 1)"] \arrow[r, "\ind (1 \wedge a)"] & \ind (\forget R \wedge M) \arrow[dd, "\ind (a)"] \\
		& & \\
		R \wedge \ind M \arrow[r, "p^{-1}"'] & \ind (\forget R \wedge M) \arrow[r, "\ind (a)"'] & \ind M 
		\end{tikzcd}
	\end{center}
	The top right square and the bottom left square commute by naturality of $p$, and the bottom right square commutes by the associativity axiom for $M$. The triangle in the top left requires us to use the construction of $p$, by considering the following diagram.
	\begin{center}
		\begin{tikzcd}
		\ind (\forget R \wedge \forget R \wedge M) \arrow[rrr, "1"] \arrow[d, "\ind (1 \wedge \eta)"'] & & & \ind (\forget R \wedge \forget R \wedge M) \arrow[d, "\ind (1 \wedge 1 \wedge \eta)"] \\
		\ind (\forget R \wedge \forget\ind (\forget R \wedge M)) \arrow[d, "\ind \Phi"'] \arrow[rr, "\ind (1 \wedge \forget p)"] & & \ind (\forget R \wedge \forget (R \wedge \ind M)) \arrow[dr, "\ind \Phi" description] \arrow[r, "\ind(1 \wedge \Phi)"] & \ind (\forget R \wedge \forget R \wedge \forget \ind M) \arrow[d, "\ind \Phi"] \\
		\ind\forget (R \wedge \ind (\forget R \wedge M)) \arrow[rrr, "\ind\forget (1 \wedge p)"] \arrow[d, "\varepsilon"'] & & & \ind\forget (R \wedge R \wedge \ind M) \arrow[d, "\varepsilon"] \\
		R \wedge \ind (\forget R \wedge M) \arrow[rrr, "1 \wedge p"'] & & & R \wedge R \wedge \ind M 
		\end{tikzcd}
	\end{center}
	The bottom square commutes by naturality of $\varepsilon$, the middle square commutes by naturality of $\Phi$ and the triangle on the right commutes since $i^*$ is strong monoidal (with structure map $\Phi$). It remains to check that the top square commutes, for which is it sufficient to check that the following diagram commutes.
	\begin{center}
		\begin{tikzcd}
		\forget R \wedge M \arrow[r, "1 \wedge \eta"] \arrow[dd, "\eta"'] & \forget R \wedge \forget \ind M \arrow[rr, "\Phi"] \arrow[dd, "\eta"] & & \forget (R \wedge \ind M) \arrow[dl, "1"] \\
		& & \forget (R \wedge \ind M) \arrow[dr, "\eta"] & \\
		\forget\ind (\forget R \wedge M) \arrow[r, "\forget\ind (1 \wedge \eta)"'] & \forget \ind (\forget R \wedge \forget\ind M) \arrow[rr, "\forget\ind \Phi"']& & \forget \ind \forget (R \wedge \ind M) \arrow[uu, "\forget\varepsilon"'] \\
		\end{tikzcd}
	\end{center} 
	The left hand square commutes by naturality of $\eta$, the right hand square commutes by naturality of $\eta$, and the triangle commutes by the triangle identities. This completes the proof of the associativity axiom.
	
	For the unit axiom, we require that the outer square in the following diagram commutes, where $\mathbb{S}$ denotes the monoidal unit of $\D$.
	\begin{center}
		\begin{tikzcd}
		\mathbb{S} \wedge \ind M \arrow[rr, "\eta \wedge 1"] \arrow[dd, "\lambda_\C"'] \arrow[dr,"p^{-1}" description] & & R \wedge \ind M \arrow[dd, "p^{-1}"] \\
		& \ind (\forget \mathbb{S} \wedge M) \arrow[dl, "\ind \lambda_\D" description] \arrow[dr, "\ind (\forget\eta \wedge 1)" description]  & \\
		\ind M & & \ind (\forget R \wedge M) \arrow[ll, "\ind (a)"] 	
		\end{tikzcd}
	\end{center}
	The top right square commutes by naturality of $p$ and the bottom triangle commutes by the unit axiom for $M$. For the left triangle we show that it commutes with $p$ instead since we have an explicit construction for this using the unit and counit of the adjunction. Consider the diagram 
	\begin{center}
		\begin{tikzcd}
		\ind (\forget\mathbb{S} \wedge M) \arrow[rrr, "\ind (1 \wedge \eta)"] \arrow[ddd, "\ind (\lambda_\D)"'] \arrow[dr, "\ind (1 \wedge \eta)" description] & & & \ind (\forget\mathbb{S} \wedge \forget\ind M) \arrow[dll, "1" description] \arrow[d, "\Phi"] \\
		& \ind (\forget\mathbb{S} \wedge \forget\ind M) \arrow[dr, "\ind \lambda_\D" description] & & \ind \forget(\mathbb{S} \wedge \ind M) \arrow[d, "\varepsilon"] \arrow[dl, "\ind \forget\lambda_\C" description] \\
		& & \ind \forget\ind M \arrow[dr, "\varepsilon \ind " description] & \mathbb{S} \wedge \ind M \arrow[d, "\lambda_\C"] \\
		\ind M \arrow[rrr, "1"'] \arrow[rru, "\ind \eta" description] & & & \ind M 
		\end{tikzcd}
	\end{center}
	in which the outer square gives the left triangle in the previous diagram with $p$ instead of $p^{-1}$. We see that the top triangle commutes immediately, the left hand square commutes by naturality of $\eta$, the top right square commutes by definition of $\forget$ being strong monoidal, the bottom right square commutes by naturality of $\varepsilon$, and the bottom triangle commutes by the triangle identities.
	
	We must also check that $\ind$ sends module maps to module maps. Let $f\colon M \to N$ be a map of $\forget R$-modules. Then the diagram
	\begin{center}
		\begin{tikzcd}
		R \wedge \ind M \arrow[r, "1\wedge \ind f"] \arrow[d, "p^{-1}"'] & R \wedge \ind N \arrow[d, "p^{-1}"] \\
		\ind(\forget R \wedge M) \arrow[r, "\ind (1 \wedge f)"] \arrow[d, "\ind (a)"'] & \ind (\forget R \wedge N) \arrow[d, "\ind (a)"] \\
		\ind M \arrow[r, "\ind f"'] & \ind N
		\end{tikzcd}
	\end{center}
	commutes, by the naturality of $p$ and since $f$ is an $\forget R$ module map.

	It remains to show that these functors form an adjoint triple. It is enough to check that the units and counits are module maps.
	
	For the unit of the $\ind \dashv \forget$ adjunction we consider the diagram
	\begin{center}
		\begin{tikzcd}
		\forget R \wedge N \arrow[ddd, "a"'] \arrow[r, "1 \wedge \eta"] \arrow[ddr, "\eta"' description] & \forget R \wedge \forget\ind N \arrow[d, "\Phi"] \\
		& \forget (R \wedge \ind N) \arrow[d, "\forget p^{-1}"] \\
		& \forget\ind (\forget R \wedge N) \arrow[d, "\forget\ind (a)"] \\
		N \arrow[r, "\eta"'] & \forget \ind N
		\end{tikzcd}
	\end{center}
	in which the square commutes by naturality of $\eta$. For the triangle, it suffices to check that the following diagram commutes. 
	\begin{center}
		\begin{tikzcd}
		\forget R \wedge N \arrow[dd, "\eta"'] \arrow[r, "1 \wedge \eta"] & \forget R \wedge \forget\ind N \arrow[r, "\Phi"] \arrow[ddr, "\eta"' description] & \forget (R \wedge \ind N) \\
		& & \forget\ind\forget (R \wedge \ind N) \arrow[u, "\forget\epsilon"'] \\
		\forget\ind (\forget R \wedge N) \arrow[rr, "\forget\ind (1 \wedge \eta)"'] & & \forget\ind (\forget R \wedge \forget\ind N) \arrow[u, "\forget\ind \Phi^{-1}"'] 
		\end{tikzcd}
	\end{center}
	The left hand square commutes by naturality of $\eta$ and checking that the right hand square commutes can be done by considering the diagram
	\begin{center}
		\begin{tikzcd}
		\forget R \wedge \forget\ind N \arrow[dd, "\eta"'] \arrow[rr, "\Phi^{-1}"] & & \forget (R \wedge \ind N) \arrow[dl, "1" description] \\
		& \forget (R \wedge \ind N) \arrow[dr, "\eta" description] & \\
		\forget\ind (\forget R \wedge \forget\ind N) \arrow[rr, "\forget\ind\Phi^{-1}"'] & & \forget\ind\forget (R \wedge \ind N) \arrow[uu, "\forget\epsilon"'] 
		\end{tikzcd}
	\end{center}
	in which the left hand square commutes by naturality of $\eta$ and the triangle commutes by the triangle identities. The argument for the counit of the $\ind \dashv \forget$ adjunction follows similarly, as does the other adjunction.
\end{proof}

\bibliographystyle{alpha}
\bibliography{mybib}
\end{document}